\documentclass[11pt]{article}
\usepackage{definitions}

\title{Accelerating Stochastic Gradient Method Using Large Stepsize for Smooth Strongly Convex Minimization}
\author{
    Luwei Bai\footnote{School of Computing and Augmented Intelligence, Arizona State University (luweibai@asu.edu).}
    \and
    Baoyu Zhou\footnote{School of Computing and Augmented Intelligence, Arizona State University (baoyu.zhou@asu.edu).}
}
\date{}

\begin{document}
\maketitle

\begin{abstract}
We study predetermined stepsize schedules for stochastic gradient methods on $m$-strongly convex, $M$-smooth objectives with conditionally unbiased gradient estimates and bounded conditional variance. 
We derive an exact tight analysis on every schedule with stepsizes in $[0,\frac{2}{M+1}]$. We then characterize the globally optimal one- and two-step schedules over all nonnegative stepsizes. Both optimal schedules depend on the ratio of the initial optimality gap to the noise level and exhibit several distinct regimes. As the influence of stochastic noise diminishes, the optimal stepsizes become larger, reflecting a balance between the benefits of faster iterate convergence and the perturbations induced by stochastic noise. These solutions motivate \texttt{serial} and \texttt{silver} composition rules with explicit performance bounds. For arbitrary horizons, we construct a schedule combining a balanced Silver prefix with constant and decreasing short steps, with the prefix length selected through explicit thresholds. Its certified upper bound never exceeds the exact optimum over all short-step schedules and is strictly smaller whenever a nonempty Silver prefix is selected. Finally, we prove global optimality of the proposed schedule on an explicit sufficient interval of the initial-error-to-variance ratio.
\end{abstract}

\section{Introduction}
Consider the optimization problem
\begin{equation*}
\min_{x\in\mathbb{R}^d} f(x),
\end{equation*}
where $f:\mathbb{R}^d\to\mathbb{R}$ is $m$-strongly convex and $M$-smooth, with $0<m<M$. Without loss of generality, throughout the remainder of the paper, we normalize the strong convexity parameter by setting $m=1$ and consider $M>1$, rather than general values $M>m>0$. We denote the normalized function class by $\mathcal F_{1,M}$. For all $\{x,y\}\subset\mathbb{R}^d$,
\begin{align*}
f(y) \geq f(x)+\langle\nabla f(x),y-x\rangle+\frac12\norm{y-x}^2 \quad \text{and} \quad
\norm{\nabla f(x)-\nabla f(y)} \leq M\norm{x-y}.
\end{align*}
Let $x_\star\in\mathbb{R}^d$ denote the unique minimizer of $f$. Fixing a horizon $N\geq1$ and an initial radius $R\geq0$, we consider a deterministic initial point $x_0\in\mathbb{R}^d$ satisfying $\norm{x_0-x_\star}\leq R$.

Stochastic gradient (SG) methods replace the exact gradient at each iterate with its stochastic estimate. Thoroughout the whole paper, we assume that these estimates are conditionally unbiased and have uniformly bounded conditional variance; see Assumption~\ref{ass:oracle}. 
\begin{assumption}[Stochastic oracle]
\label{ass:oracle}
For any $k\in\{0,1,\ldots,N-1\}$, let $\mathcal H_k=\boldsymbol{\sigma}(x_0,g_0,\ldots,g_{k-1})$, with $\mathcal H_0=\boldsymbol{\sigma}(x_0)$, denote the $\boldsymbol{\sigma}$-algebra including the history before the oracle call at $x_k$. The query point $x_k$ is $\mathcal H_k$-measurable, and the oracle returns a random vector $g_k$ satisfying
\begin{align}\label{eq:unbiasedness}
\E[g_k\mid\mathcal H_k] =\nabla f(x_k) \quad \text{and} \quad \E[\norm{g_k-\nabla f(x_k)}^2\mid\mathcal H_k] \leq\sigma^2
\end{align}
almost surely, where $\sigma>0$. In addition, in this work, we define $\E[\cdot] = \E[\cdot\mid \mathcal H_0]$.
\end{assumption}

A central feature of SG convergence guarantees is the separation between the initial error and variance accumulation. Under Assumption~\ref{eq:unbiasedness}, the bound on the mean-squared distance takes the form of
\begin{equation*}
\E[\norm{x_N-x_\star}^2]\leq\tau_N R^2+V_N \sigma^2
\end{equation*}
with some quantities $\tau_N\geq 0$ and $V_N\geq 0$, where the first term ($\tau_N R^2$) captures the contribution of the initial error that remains after $N$ iterations, while the second term ($V_N \sigma^2$) quantifies the accumulation of stochastic noise. Classical analyses show that a fixed positive stepsize generally leaves a noise-dependent error floor, whereas suitably decreasing stepsizes yield vanishing error; see, e.g., \cite[Section~4.2]{bottou2018optimization}.

Existing schedule designs implement this balance through different initial stepsizes and decay rules. Examples include inverse-time schedules~\cite{bottou2018optimization,nguyen2018sgd}, polynomial decay~\cite{moulines2011nonasymptotic}, stagewise step decay~\cite{wang2021stepdecay}, constant--inverse-time schedules~\cite{stich2019unified}, and exponential or cosine decay~\cite{li2021second}. The common design principle is to make substantial progress with relatively large early steps and reduce the stepsizes as stochastic noise becomes more significant. Many standard SG analyses nevertheless restrict each stepsize to $\frac{1}{M}$ or a smaller bound, which facilitates control of the variance but excludes the much larger steps used in deterministic algorithms~\cite{bottou2018optimization,stich2019unified}.

For deterministic gradient descent (GD) methods, carefully designed large stepsizes can accelerate convergence. These stepsizes may exceed $\frac{2}{M}$ and need not decrease the objective at every iteration. Their benefit comes from coordinating progress across multiple iterations. With constant small stepsizes, GD requires $\mathcal{O}(\kappa\log(1/\epsilon))$ iterations to reach an iterate whose distance from the optimal solution is at most $\epsilon$, where $\kappa=\frac{M}{m}$ is the condition number and $\epsilon>0$ is the target accuracy~\cite{nesterov2018lectures}. The silver stepsize schedule~\cite{altschuler2025acceleration} is constructed for horizons $N=2^\ell$ with $\ell$ being any natural number. It combines short steps with aggressive large steps and attains an iteration complexity of $\mathcal{O}(\kappa^{\log_{1+\sqrt{2}}2}\log(1/\epsilon))$. This reduces the condition-number exponent from $1$ to $\log_{1+\sqrt2}2\approx0.7864$. Its coordinated steps guarantee faster progress over complete blocks. For comparison, Nesterov's accelerated method achieves an iteration complexity result of $\mathcal O(\sqrt{\kappa}\log(1/\epsilon))$ by utilizing momentum~\cite{nesterov2018lectures}. 
Subsequent work extends stepsize-based acceleration to smooth convex objectives, e.g.,~\cite{altschuler2025smooth}. Composition and concatenation constructions allow arbitrary prescribed horizons~\cite{grimmer2025composing,zhang2026accelerated}. Further extensions cover anytime schedules~\cite{zhang2025anytime} and proximal methods~\cite{bok2025accelerating}. Meanwhile,~\cite{ye2026silver} shows that the convergence rate of silver stepsize is tight up to subpolynomial factors for smooth convex objectives by constructing lower bounds, and it also establishes tightness of the corresponding anytime rate in the same sense. Novel stepsize schedules have also been developed for minimax problems, where negative stepsizes can enable convergence of gradient descent--ascent methods~\cite{shugart2025negative}. Together, these results show that well-designed aggressive large steps can improve convergence rates without requiring descent at every iteration.

Compared with deterministic GD, SG methods fundamentally change the stepsize-design problem because each step must balance iterate convergence against perturbations from stochastic noise. The paper~\cite{vernimmen2025empirical} numerically compares constant-step GD, long-step methods, and momentum-accelerated methods under relative gradient inexactness. Moreover, its complementary theoretical study~\cite{vernimmen2026worst} provides worst-case convergence guarantees for constant-step GD with relatively inexact gradients. These constant-step guarantees do not cover stepsizes larger than the classical threshold $\frac{2}{M}$. Such stepsizes are central to the long-step schedules considered in this paper. More recently,~\cite{grimmer2026exact} establishes exact worst-case bounds on the expected final-iterate suboptimality of SG methods for smooth convex objectives under conditionally unbiased noise with bounded variance, and characterizes a unique, monotonically decreasing optimal stepsize schedule when the variance bound is sufficiently large. Moreover,~\cite{bai2025generalization} proposes a two-step SG schedule for finite-support stochastic oracles. When noise is sufficiently small relative to the initial error, it improves upon the optimal constant stepsize $\frac{2}{M+m}$ for deterministic GD. Their analysis does not characterize the globally optimal two-step schedule. The resulting question is whether predetermined large steps can reduce the initial error faster while controlling variance accumulation. Both effects must be balanced to improve the overall algorithm performance.

We formulate this question as a finite-horizon minimax problem. Given a predetermined schedule $\boldsymbol{\alpha}=(\alpha_0,\ldots,\alpha_{N-1})\in\mathbb{R}_{\geq0}^N$, the SG method performs the updates
\begin{equation}
x_{k+1}=x_k-\alpha_k g_k\qquad \forall k\in\{0,1,\ldots,N-1\}. \label{eq:sg-update}
\end{equation}
The schedule is chosen before the oracle calls (see Assumption~\ref{ass:oracle}) are made and may depend on predetermined fixed parameters $(M,R,\sigma,N)$, but not on the realized oracle outputs. For each $f\in\mathcal F_{1,M}$, let $\mathfrak O_\sigma(f)$ denote the set of stochastic oracles for $f$ satisfying Assumption~\ref{ass:oracle}. For any stepsize schedule $\boldsymbol{\alpha}\in\mathbb{R}_{\geq0}^N$, let's define its worst-case terminal error under Assumption~\ref{ass:oracle} by
\begin{equation}
\D_N(\boldsymbol{\alpha}):=\sup_{\substack{f\in\mathcal F_{1,M},\ \mathcal O\in\mathfrak O_\sigma(f)\\x_0\in\mathbb R^d,\ \norm{x_0-x_\star}\leq R}}\ \E_{\mathcal O}\!\left[\norm{x_N-x_\star}^2\right].
\label{eq:worst-case-error}
\end{equation}
The expectation is taken over the oracle output randomness. The supremum ranges jointly over all functions $f:\mathbb R^d\to\mathbb R$ in $\mathcal F_{1,M}$, all possible stochastic oracles satisfying Assumption~\ref{ass:oracle}, and all admissible deterministic initial points. 

Our objective is to choose a schedule that minimizes this worst-case error:
\begin{align}
\D_N^\star :=\inf_{\boldsymbol{\alpha}\in\mathbb R_{\geq0}^N}\D_N(\boldsymbol{\alpha})  =\inf_{\boldsymbol{\alpha}\in\mathbb R_{\geq0}^N}\ \sup_{\substack{f\in\mathcal F_{1,M}\ \mathcal O\in\mathfrak O_\sigma(f)\\x_0\in\mathbb R^d,\ \norm{x_0-x_\star}\leq R}}\ \E_{\mathcal O}\!\left[\norm{x_N-x_\star}^2\right]. \label{eq:minimax-schedule}
\end{align}
Whenever the infimum is attained, a schedule $\boldsymbol{\alpha}^\star\in\operatorname*{arg\,min}_{\boldsymbol{\alpha}\in\mathbb{R}_{\geq0}^N}\D_N(\boldsymbol{\alpha})$ is called an \textit{optimal $N$-step schedule}.
Our approach begins by solving problem~\eqref{eq:minimax-schedule} exactly for one and two steps, i.e., $N\in\{1,2\}$. The one-step optimum always uses a stepsize no larger than $\frac{2}{M+1}$. The optimal two-step schedule takes different forms across the regions of $\gamma:=\frac{R^2}{\sigma^2}$. In the largest-$\gamma$ region, it is exactly the two-step silver stepsizes, containing a large stepsize~\cite{altschuler2025acceleration}. This reveals the possibility of improving SG methods through schedules containing large stepsizes. Moreover, two ways of composing optimal one-step schedules recover the two-step optimum below and above a threshold in $\gamma$, respectively. These constructions motivate the \texttt{serial} and \texttt{silver} joins. In particular, the \texttt{silver} join introduces large stepsizes into the resulting schedule. Combining a prefix constructed through \texttt{silver} joins with a short-step continuation connected through \texttt{serial} joins then balances rapid initial-error reduction against variance accumulation over horizons with arbitrary lengths. 

Our three-fold contributions are as follows.
\begin{enumerate}
\item We establish an exact worst-case characterization for every predetermined short-step schedule with stepsizes in the interval $[0,\frac{2}{M+1}]$. Applying this characterization to commonly used schedules provides tight performance benchmarks for comparisons with our proposed stepsize schedules in Section~\ref{sec:arbitrary-horizon-schedule}.
\item We characterize the globally optimal one- and two-step schedules over all nonnegative stepsizes. The two-step optimum exhibits four regimes determined by $\gamma$, including an explicit transition to the two-step silver stepsizes~\cite{altschuler2025acceleration}. These optimal solutions motivate two composition rules, \texttt{serial} and \texttt{silver} joins, which combine stepsize blocks of arbitrary lengths $p$ and $q$ into a block of length $p+q$, together with an explicit upper bound on the resulting algorithm performance. We further establish that repeated \texttt{serial} composition based on the optimal one-step rule yields the optimal schedule among all predetermined short-step schedules (i.e., $\alpha_k \in [0,\frac{2}{M+1}]$ $\forall k\in\{0,1,\ldots,N-1\}$).
\item We construct a predetermined schedule for any given horizon by combining a \texttt{silver} prefix with a short-step continuation obtained through \texttt{serial} joins. We derive an explicit upper bound on its worst-case terminal error (see~\eqref{eq:worst-case-error}). For a fixed horizon, we evaluate this bound over all possible \texttt{silver} prefix lengths and select the one that gives the smallest bound. An empty \texttt{silver} prefix recovers the purely \texttt{serial} schedule, which is optimal among short-step schedules. In contrast, a nonempty \texttt{silver} prefix yields an upper bound strictly below the exact worst-case error of the purely \texttt{serial} schedule, thereby certifying an improvement over every predetermined short-step schedule. We further prove that, for every horizon $N\ge2$, sufficiently large $\gamma$ ensures that our proposed schedule uses at least one stepsize larger than $\bar\alpha$ and achieves a strictly smaller worst-case expected terminal error than the optimal short-step schedule. In contrast, for sufficiently small $\gamma$, our proposed schedule is globally optimal among all predetermined nonnegative stepsize schedules.
\end{enumerate}

Figure~\ref{fig:introduction-schedule-pattern} illustrates our construction for different initial ratios $\gamma(:=\frac{R^2}{\sigma^2})$ and horizons $N$, with $M=20$ fixed. For any fixed $(M,N,\gamma)$, the construction in Section~\ref{sec:arbitrary-horizon-schedule} determines the entire stepsize schedule. 
The schedule has up to three phases. The \textit{acceleration phase} coordinates large and small steps to reduce the initial error rapidly. The \textit{constant-stepsize phase} uses the constant stepsize $\frac{2}{M+1}$, and the \textit{decreasing-stepsize phase} controls variance accumulation. All stepsizes and phase lengths are predetermined once parameters $(M,N,\gamma)$ are given.

In these examples, a larger initial ratio $\gamma$ leads to a longer acceleration phase. For example, when $(N,\gamma)=(200,10)$, our predetermined stepsize schedule includes $7$ acceleration steps, $18$ constant steps, and $175$ decreasing steps. In comparison, when $(N,\gamma)=(200,10^3)$, our schedule includes $21$ acceleration steps, $21$ constant steps, and $158$ decreasing steps. Thus, the initial parameter selection determines both which phases appear and how long each phase lasts. Large steps in the acceleration phase can substantially exceed $\frac{2}{M}$. Such steps do not guarantee a decrease in the expected squared distance to the minimizer \textit{at every iteration}. Our performance guarantee ensures a net decrease in the expected squared distance to the minimizer \textit{over the whole schedule}.
\begin{figure}[!htbp]
\centering
\includegraphics[width=\linewidth]{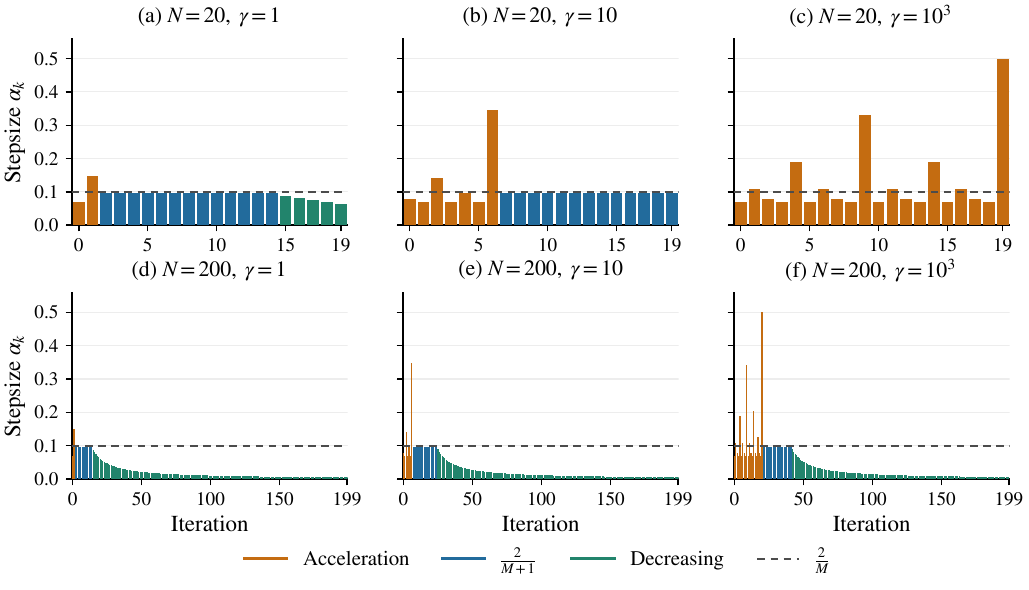}
\caption{Three phases of predetermined stepsize schedules with $M=20$. Each bar gives the stepsize $\alpha_k$ at iteration index $k\in\{0,1,\ldots,N-1\}$.}
\label{fig:introduction-schedule-pattern}
\end{figure}

\vspace{-5ex}

\subsection{Notation}
In this paper, we let $\mathbb R_{\geq0}:=[0,\infty)$ be the nonnegative real number set. Let $\mathbb{N}:=\{0,1,2,\ldots\}$ be the natural number set. For any $n\in\mathbb{N}$, let $[n]$ denote the set of all natural numbers no larger than $n$, i.e., $[n]:=\{0,1,\ldots,n\}$. If $n<0$, then $[n] = \varnothing$. For any vectors $\{x,y\}\subset\mathbb R^d$, let $\langle x,y\rangle:=x^{\mathsf T}y$ and $\|x\|:=\sqrt{\langle x,x\rangle}$ denote the Euclidean inner product and the Euclidean norm, respectively. For any $x\in\mathbb{R}^d$ and $i\in[d]\backslash\{0\}$, let $[x]_i$ denote the $i$th element of vector $x$. For any real number $b\in\mathbb{R}$, let $\lfloor b\rfloor$ represent the largest integer not exceeding $b$; let $\lceil b\rceil$ be the smallest integer greater than or equal to $b$; and $(b)_+:=\max\{b,0\}$. Logarithms without an indicated base are natural logarithms. Throughout, $\mathcal F_{1,M}$ is the class of $1$-strongly convex, $M$-smooth functions, and $x_\star\in\mathbb{R}^d$ denotes the unique minimizer of function $f\in\mathcal F_{1,M}$. The horizon is $N$, $R^2$ is an upper bound on the initial squared distance $\|x_0-x_\star\|^2$, and $\sigma^2$ is the conditional variance bound specified in Assumption~\ref{ass:oracle}.
Throughout the remainder of the paper, we suppress the dependence of expectation quantities on the stochastic oracle $\mathcal O$, since all equalities and inequalities involving these quantities hold for any oracle satisfying Assumption~\ref{ass:oracle}. Accordingly, we write $\E[\cdot]$ in place of $\E_{\mathcal O}[\cdot]$. We also use
$\E[\cdot]$ to denote conditional expectation given the initial point $x_0$, and write $\E[\,\cdot\mid\mathcal H_k]$ for conditional expectation given the history $\mathcal H_k$, which consists of the initial point $x_0\in\R^d$ and all stochastic oracle outputs generated before $x_k\in\R^d$, i.e., $\{g_t\}_{t=0}^{k-1}$.
We also define the quantities
\begin{equation}
\gamma:=\frac{R^2}{\sigma^2},\qquad \bar\alpha:=\frac{2}{M+1},\qquad \text{and} \qquad \gamma_1:=\frac{2}{M-1}.
\label{eq:one-two-notation}
\end{equation}
Here, $\gamma$ is the initial error ratio. We call $\bar\alpha$ the capped stepsize and $(\bar\alpha,\bar\alpha)$ the capped pair; a short-step schedule has every stepsize in $[0,\bar\alpha]$. The parameter $\gamma_1$ serves as a threshold: whenever $\gamma > \gamma_1$, the optimal one-step stepsize is $\bar\alpha$ (see Theorem~\ref{thm:one-step}). 
A predetermined schedule is denoted by $\boldsymbol\alpha=(\alpha_0,\ldots,\alpha_{N-1})\in\mathbb R_{\geq0}^N$. Its worst-case terminal error $\D_N(\boldsymbol\alpha)$ and the optimal value $\D_N^\star$ are defined in~\eqref{eq:worst-case-error} and~\eqref{eq:minimax-schedule}, respectively.
The upper bound on the worst-case terminal error can be expressed as $\tau R^2+V\sigma^2$, which decomposes the error into contributions from the initial error and stochastic noise. Here, $\tau$ and $V$ are the respective coefficients, quantifying the propagation of the initial error and the accumulation of stochastic variance; their values are specified for each schedule.
Moreover, for a scalar map $\mathcal T:[0,\infty)\to[0,\infty)$, we write $\mathcal T^{\circ j}(\gamma)$ for the value obtained by applying $\mathcal T(\cdot)$ for $j$ times, starting with $\gamma$ and using each output as the next input. We use the convention
\begin{equation}
\mathcal T^{\circ0}(\gamma)=\gamma \qquad \text{and} \qquad \mathcal T^{\circ(j+1)}(\gamma)=\mathcal T\!\left(\mathcal T^{\circ j}(\gamma)\right)\quad \forall j\in\mathbb{N}.
\label{eq:serial-map-iteration}
\end{equation}

\section{Short stepsize schedule for SG methods}
Classical analyses of SG methods for smooth strongly convex optimization commonly impose restrictions such as $\{\alpha_k\}_{k\in\mathbb{N}}\subset(0,\frac{1}{M}]$, or with smaller upper bounds, when studying constant and decreasing stepsizes; see, for example, \cite[Section~4.2]{bottou2018optimization}. In this paper, we 
refer to stepsize schedules contained within an even larger interval $[0,\frac{2}{M+1}]\supset(0,\frac{1}{M}]$ as short-step schedules. The cutoff $\frac{2}{M+1}$ is natural for squared-distance performance. Under an exact-gradient update for minimizing $f\in\mathcal F_{1,M}$, the distance to the minimizer contracts by at most a factor of $\max\{|1-\alpha_k|,|1-M\alpha_k|\}$. For $\alpha_k\in[0,\frac{2}{M+1}]$, this factor reduces to $1-\alpha_k$. Thus, for any predetermined stepsize schedule within this interval, the quadratic function $f(u)=\frac12u^2$, together with an additive-noise oracle, attains the corresponding multi-step bound. This establishes that the bound obtained by composing the one-step estimates is tight.

For constant stepsizes within the interval $[0,\frac{2}{M+1}]$, 
\cite[Theorem~4.1]{rubbens2025computer} establishes an exact worst-case bound under conditional unbiasedness and bounded variance assumptions. Their argument extends directly to arbitrary predetermined short-step schedules. We present such extension below (see Proposition~\ref{prop:exact}) to obtain an exact reference value for comparison with the performance upper bounds of the stepsize schedules developed later in our paper.

\begin{proposition}[Tight analysis of predetermined short stepsize schedule]
\label{prop:exact}
Under Assumption~\ref{ass:oracle}, let $N\geq1$ and let $\boldsymbol{\alpha}=(\alpha_0,\ldots,\alpha_{N-1})$ be a predetermined schedule satisfying $0\leq\alpha_k\leq\frac{2}{M+1}$ for every $k\in[N-1]$. Define
\begin{equation}
\tau_N(\boldsymbol{\alpha}):=\prod_{k=0}^{N-1}(1-\alpha_k)^2 \qquad\text{and}\qquad V_N(\boldsymbol{\alpha}):=\sum_{k=0}^{N-1}\left(\alpha_k^2\prod_{j=k+1}^{N-1}(1-\alpha_j)^2\right), \label{eq:coefficients}
\end{equation}
then we have
\begin{equation}
\D_N(\boldsymbol{\alpha})=\tau_N(\boldsymbol{\alpha})\cdot R^2+V_N(\boldsymbol{\alpha})\cdot\sigma^2. \label{eq:exact}
\end{equation}
Here, by convention, $\prod_{j=N}^{N-1}(1-\alpha_j)^2 = 1$ and $V_1(\boldsymbol{\alpha})=\alpha_0^2$.
The value in~\eqref{eq:exact} is attained by a quadratic with a Gaussian additive-noise oracle.
\end{proposition}

\begin{proof}
Given $f\in\mathcal{F}_{1,M}$ hving the unique minimizer $x_\star\in\mathbb{R}^d$, the function $x\mapsto (f(x)-\frac{\|x-x_\star\|^2}{2})$ is convex and $(M-1)$-smooth, and its gradient vanishes at $x_\star$. Co-coercivity therefore gives
$\langle x-x_\star,\nabla f(x)-x+x_\star\rangle\geq\frac{1}{M-1}\norm{\nabla f(x)-x+x_\star}^2$ for every $x\in\mathbb{R}^d$. Since $0\leq\alpha_k\leq\frac{2}{M+1}<1$, expanding the squared norm yields
\begin{align}
&\norm{x-x_\star-\alpha_k\nabla f(x)}^2 = \norm{(1-\alpha_k)(x-x_\star) + \alpha_k(x-x_\star - \nabla f(x))}^2 \notag \\
= \ &(1-\alpha_k)^2\norm{x-x_\star}^2 + \alpha_k^2\norm{\nabla f(x)-x+x_\star}^2 - 2\alpha_k(1-\alpha_k)\langle x-x_\star,\nabla f(x)-x+x_\star\rangle \label{eq:contraction} \\
\leq \ &(1-\alpha_k)^2\norm{x-x_\star}^2- \alpha_k\left(\tfrac{2(1-\alpha_k)}{M-1} - \alpha_k\right) \norm{\nabla f(x)-x+x_\star}^2 \leq (1-\alpha_k)^2\norm{x-x_\star}^2. \notag
\end{align}
By Assumption~\ref{ass:oracle} and~\eqref{eq:contraction}, we directly have
\begin{align*}
&\E[\norm{x_{k+1}-x_\star}^2\mid\mathcal H_k] = \E[\norm{x_k-x_\star - \alpha_kg_k}^2\mid\mathcal H_k] \\
= \ &\E[\norm{(x_k-x_\star - \alpha_k\nabla f(x_k)) + \alpha_k(\nabla f(x_k) - g_k)}^2\mid\mathcal H_k] \\
= \ &\|x_k-x_\star - \alpha_k\nabla f(x_k)\|^2 + \alpha_k^2\E[\|g_k - \nabla f(x_k)\|^2 \mid\mathcal H_k] \leq(1-\alpha_k)^2\norm{x_k-x_\star}^2+\alpha_k^2\sigma^2.
\end{align*}
Taking total expectations, we further obtain
\begin{equation}
\E[\norm{x_{k+1}-x_\star}^2] = \E[\E[\norm{x_{k+1}-x_\star}^2 \mid \mathcal H_k]] \leq (1-\alpha_k)^2\E[\norm{x_k-x_\star}^2]+\alpha_k^2\sigma^2,\quad \forall k\in[N-1]. \label{eq:recurrence}
\end{equation}
Iterating~\eqref{eq:recurrence} and using $\norm{x_0-x_\star}^2\leq R^2$ shows that the right-hand side of~\eqref{eq:exact} is an upper bound on $\D_N(\boldsymbol{\alpha})$.
To prove that~\eqref{eq:exact} holds true, let's consider the initial iterate $x_0 = R\in\R$ and function $f:\R\to\R$ being $f(x)=\frac12x^2$. Meanwhile, for any $k\in\mathbb{N}$, we consider the stochastic oracle generating stochastic gradient estimates $g_k=\nabla f(x_k)+\xi_k$ with $\xi_k\overset{\mathrm{i.i.d.}}{\sim}\mathcal N(0,\sigma^2)$.
Consequently, $\E[\norm{x_{k+1}-x_\star}^2]=(1-\alpha_k)^2\E[\norm{x_k-x_\star}^2]+\alpha_k^2\sigma^2$ for all $k\in[N-1]$. Since $\norm{x_0-x_\star}^2=R^2$, iterating this equality attains the upper bound and proves~\eqref{eq:exact}.
\end{proof}


Table~\ref{tab:common-short-step-exact} reports the exact coefficients $(\tau_N(\boldsymbol{\alpha}),V_N(\boldsymbol{\alpha}))$ associated with the initial error and variance
accumulation, respectively, in the bound $\D_N(\boldsymbol{\alpha})=\tau_N(\boldsymbol{\alpha})\cdot R^2+V_N(\boldsymbol{\alpha})\cdot\sigma^2$ for various short-step schedules satisfying $\alpha_k\in[0,\frac{2}{M+1}]$ $
\forall k\in[N-1]$. These coefficients are obtained by applying the tight worst-case bound established in Proposition~\ref{prop:exact} to each stepsize schedule. In the table, we have $k\in[N-1]$, $\alpha_0\in[0,\frac{2}{M+1}]$ being the prescribed first stepsize, $r\in(0,1)$ being the polynomial exponent (for Polynomial schedule in~\cite{moulines2011nonasymptotic}), and $N=JL$ specifying $J$ epochs of $L$ iterations (for Halving schedule in~\cite{wang2021stepdecay}). In addition, the constant-prefix length $p$ and the inverse-time coefficients used in the table (for schedules proposed in~\cite{nguyen2018sgd} and~\cite{stich2019unified}) are defined by
$p=
\begin{cases}
N&\text{if }N-1\leq2M\\
\lfloor N/2\rfloor&\text{if }N-1>2M,
\end{cases}$
and, for each $n\in\mathbb{N}$, $\tau_n^{\mathrm{inv}}:=\left[\frac{(4M-2)(4M-1)}{(4M+n-2)(4M+n-1)}\right]^2$ and $V_n^{\mathrm{inv}}:=\frac{4\sum_{j=1}^{n}(4M+j-2)^2}{(4M+n-2)^2(4M+n-1)^2}$.

\begingroup
\footnotesize
\setlength{\tabcolsep}{2pt}
\setlength{\aboverulesep}{2pt}
\setlength{\belowrulesep}{2pt}
\renewcommand{\arraystretch}{1.08}
\newsavebox{\SGcataloguebox}
\newcommand{\SGfit}[1]{\sbox{\SGcataloguebox}{$\displaystyle#1$}\ifdim\wd\SGcataloguebox>\linewidth\resizebox{\linewidth}{!}{\usebox{\SGcataloguebox}}\else\usebox{\SGcataloguebox}\fi}
\begin{longtable}{@{}>{\centering\arraybackslash}m{0.15\linewidth}>{\centering\arraybackslash}m{0.20\linewidth}>{\centering\arraybackslash}m{0.63\linewidth}@{}}
\caption{Exact coefficients in $\D_N=\tau_NR^2+V_N\sigma^2$ for common short-step SG schedules.}\label{tab:common-short-step-exact}\\
\toprule
Schedule & $\alpha_k$ & Exact coefficients\\
\midrule
\endfirsthead
\toprule
Schedule & $\alpha_k$ & Exact coefficients\\
\midrule
\endhead
\SGfit{\text{Constant~\cite{rubbens2025computer}}} & \SGfit{\alpha_0} & \SGfit{\tau_N=(1-\alpha_0)^{2N}}; \quad \SGfit{V_N=\tfrac{\alpha_0}{2-\alpha_0}(1-\tau_N)}\\
\midrule
\SGfit{\text{Inverse~\cite{bottou2018optimization}}} & \SGfit{\tfrac{\alpha_0}{1+k\alpha_0}} & \SGfit{\tau_N=\tfrac{(1-\alpha_0)^2}{(1+(N-1)\alpha_0)^2}}; \quad\SGfit{V_N=\tfrac{N\alpha_0^2}{(1+(N-1)\alpha_0)^2}}\\
\midrule
\SGfit{\text{Inverse-$2$~\cite{nguyen2018sgd}}} & \SGfit{\frac{2}{4M+k}} & \SGfit{\tau_N=\tau_N^{\mathrm{inv}}}; \quad \SGfit{V_N=V_N^{\mathrm{inv}}}\\
\midrule
\SGfit{\text{Two-phase~\cite{stich2019unified}}} & \SGfit{
\begin{cases}
\dfrac{1}{2M},&0\le k<p,\\[3pt]\dfrac{2}{4M+k-p},&p\le k<N.
\end{cases}
} & \SGfit{\tau_N=\left(1-\frac{1}{2M}\right)^{2p}\tau_{N-p}^{\mathrm{inv}}}; \quad \SGfit{V_N=\frac{1-(1-\frac{1}{2M})^{2p}}{4M-1}\tau_{N-p}^{\mathrm{inv}}+V_{N-p}^{\mathrm{inv}}}\\
\midrule
\SGfit{\text{Polynomial~\cite{moulines2011nonasymptotic}}} & \SGfit{\frac{\alpha_0}{(k+1)^r}} & \SGfit{\tau_N=\prod_{k=0}^{N-1}\left(1-\frac{\alpha_0}{(k+1)^r}\right)^2};\par\SGfit{V_N=\alpha_0^2\sum_{k=0}^{N-1}\frac{1}{(k+1)^{2r}}\prod_{j=k+1}^{N-1}\left(1-\frac{\alpha_0}{(j+1)^r}\right)^2}\\
\midrule
\SGfit{\text{Halving~\cite{wang2021stepdecay}}} & \SGfit{\frac{\alpha_0}{2^{\lfloor\frac{k}{L}\rfloor}}} & \SGfit{\tau_N=\prod_{\ell=1}^{J}(1-\alpha_0 2^{1-\ell})^{2L}};\par\SGfit{V_N=\sum_{\ell=1}^{J}\frac{\alpha_0(1-(1-\alpha_0 2^{1-\ell})^{2L})}{2^\ell-\alpha_0}\prod_{s=\ell+1}^{J}(1-\alpha_0 2^{1-s})^{2L}}\\
\midrule
\SGfit{\text{Exp.~\cite{li2021second}}} & \SGfit{\frac{1}{M N^{\frac{k+1}{N}}}} & \SGfit{\tau_N=\prod_{k=0}^{N-1}\left(1-\frac{1}{M N^{\frac{k+1}{N}}}\right)^2};\par\SGfit{V_N=\frac{1}{M^2}\sum_{k=0}^{N-1}N^{-2\frac{k+1}{N}}\prod_{j=k+1}^{N-1}\left(1-\frac{1}{M N^{\frac{j+1}{N}}}\right)^2}\\
\midrule
\SGfit{\text{Cosine~\cite{li2021second}}} & \SGfit{\frac{1}{M}\cos^2\!\left(\frac{\pi(k+1)}{2N}\right)} & \SGfit{\tau_N=\prod_{k=0}^{N-1}\left(1-\frac{1}{M}\cos^2\!\left(\frac{\pi(k+1)}{2N}\right)\right)^2};\par\SGfit{V_N=\frac{1}{M^2}\sum_{k=0}^{N-1}\cos^4\!\left(\frac{\pi(k+1)}{2N}\right)\prod_{j=k+1}^{N-1}\left(1-\frac{1}{M}\cos^2\!\left(\frac{\pi(j+1)}{2N}\right)\right)^2}\\
\midrule
\SGfit{\text{Linear~\cite{goodfellow2016deep}}} & \SGfit{\frac{N-k}{MN}} & \SGfit{\tau_N=\prod_{j=1}^{N}\left(1-\frac{j}{MN}\right)^2}; \quad \SGfit{V_N=\frac{1}{M^2N^2}\sum_{j=1}^{N}j^2\prod_{\ell=1}^{j-1}\left(1-\frac{\ell}{MN}\right)^2}\\
\bottomrule
\end{longtable}
\endgroup

Every displayed value in~Table~\ref{tab:common-short-step-exact} is attained by the quadratic function with additive-noise oracle used in Proposition~\ref{prop:exact}. These exact formulas provide benchmarks for comparing the performance of the stepsize schedules developed in this work; see Section~\ref{sec:arbitrary-horizon-schedule}.

\section{Large Stepsize Schedules}
To motivate the use of large stepsizes in SG methods, we first characterize the optimal one- and two-step schedules in Section~\ref{subsec:one-two-step-optimal}. The construction of the two-step schedule from the one-step schedule then motivates two composition rules, called \texttt{serial} and \texttt{silver} joins, each of which combines schedules of arbitrary lengths $p$ and $q$ into a schedule of length $p+q$. Applying these rules recursively yields schedules for arbitrary horizons.

\subsection{Optimal One- and Two-Step Schedules}
\label{subsec:one-two-step-optimal}
We begin by asking which stepsizes minimize the worst-case expected squared distance to the minimizer (see~\eqref{eq:worst-case-error}) after one or two SG iterations. In the deterministic setting, the optimal one-step stepsize for $\mathcal{F}_{1,M}$ is the classical choice $\frac{2}{M+1}$~\cite{altschuler2018greed}. Moreover, \cite{altschuler2025acceleration} characterize the optimal two-step schedule, which combines a stepsize below $\frac{2}{M+1}$ with another stepsize larger than $\frac{2}{M+1}$. 
(The larger stepsize can even exceed $\frac{2}{M}$.) This splitting of the optimal two-step stepsize into a short and a long step motivates their recursive construction of the silver stepsize schedule.

For SG methods, stepsize selection must also account for stochastic noise. Larger stepsizes may improve contraction of the initial error but amplify the accumulated stochastic noise, so the optimal choice of stepsize schedules depends on both $R$ and $\sigma$. When the initial distance to the optimal solution is small relative to the stochastic noise level, smaller stepsizes are preferable. Building on the work of~\cite{altschuler2025acceleration}, we investigate whether the two-step advantage of large stepsizes persists under stochastic noise and how this advantage depends on the relative values of $R$ and~$\sigma$.

For simplicity, we use the notation $(\gamma,\bar\alpha,\gamma_1)$ defined in~\eqref{eq:one-two-notation} throughout the following analyses.

\textbf{The Optimal One-Step Stepsize.} We present the optimal one-step stepsize in Theorem~\ref{thm:one-step}.

\begin{theorem}[Optimal one-step stepsize]\label{thm:one-step}
For every $M>1$, $R\ge0$ and $\sigma>0$, the unique optimal nonnegative one-step stepsize (solving~\eqref{eq:minimax-schedule} with $N=1$) is
\begin{equation}
\label{eq:one-step-optimal}
\alpha_*(\gamma) =
\begin{cases}
\tfrac{\gamma}{1+\gamma} &\text{if }0\leq\gamma\leq\tfrac{2}{M-1}\\[4pt]
\bar\alpha &\text{if }\gamma>\tfrac{2}{M-1}
\end{cases}
\end{equation}
with
\begin{equation}
\min_{\alpha\ge0}\D_1(\alpha)=\sigma^2\mathcal T(\gamma) \qquad \text{where} \qquad \mathcal T(\gamma)=
\begin{cases}
\tfrac{\gamma}{1+\gamma} &\text{if }0\le \gamma\le\tfrac{2}{M-1}\\[3pt](1-\bar\alpha)^2\gamma+\bar\alpha^2 &\text{if }\gamma>\tfrac{2}{M-1}.
\end{cases}
\label{eq:one-step-optimal-value}
\end{equation}
\end{theorem}

\begin{proof}
This proof is deferred to Appendix~\ref{app:one-step-optimal-proof}.
\end{proof}

The optimal one-step schedule always uses a short stepsize. It takes $\bar\alpha=\frac{2}{M+1}$ when $\gamma$ is large enough (i.e., $\gamma > \frac{2}{M-1}$) and becomes $\frac{\gamma}{1+\gamma}$ otherwise, which diminishes to zero as $\gamma\to 0$. This recovers the classical conclusion that SG methods favor small stepsizes when $\gamma$ is small, as the initial point already lies in a noise-dominated neighborhood of the minimizer. The one-step result does not reveal whether large stepsizes can reduce the initial error more rapidly and bring the iterates into the noise-dominated neighborhood faster. To investigate this potential advantage, we next characterize the optimal two-step schedule.

\textbf{The Optimal Two-Step Schedule.} \ Let's define parameters
\begin{equation}
\eta=\sqrt{2M^2-2M+1},\qquad \hat a=\frac{M-1+\eta}{M},\qquad \text{and} \qquad \gamma_2=\frac{2(M^2+3)}{(M-1)^3}.\label{eq:two-step-parameters}
\end{equation}
For $\zeta\ge1$, we further define
\begin{equation}
\begin{aligned}
T(\zeta)=\frac{(M-1)^4\zeta^2}{(M\zeta+1)^2(\zeta+M)^2} \ \ \text{and} \ \ U(\zeta)=\frac{(\zeta+1)^2[(2M^2-2M+1)\zeta^2+2M\zeta+1]}{(M\zeta+1)^2(\zeta+M)^2},
\end{aligned}
\label{eq:two-step-TU}
\end{equation}
and then set
\begin{equation}
\gamma_\star=\frac{U(\hat a)-U(1)}{T(1)-T(\hat a)}.\label{eq:two-step-crossing}
\end{equation}
Given $\gamma \geq 0$ and $\bar{\boldsymbol{\alpha}}=(\bar\alpha,\bar\alpha)$, we consider the predetermined two-step schedule
\begin{equation}
\boldsymbol{\alpha}^\star(\gamma)=(\alpha_0^\star(\gamma),\alpha_1^\star(\gamma)) =
\begin{cases}
\left(\tfrac{\gamma}{1+\gamma},\tfrac{\gamma}{1+2\gamma}\right) &\text{if }0\le\gamma\le\gamma_1\\[6pt]\left(\bar\alpha,\tfrac{\gamma(M-1)^2+4}{\gamma(M-1)^2+4+(M+1)^2}\right)&\text{if }\gamma_1\le\gamma\le\gamma_2\\[6pt](\bar\alpha,\bar\alpha)&\text{if }\gamma_2\le\gamma<\gamma_\star\\[3pt]\left(\tfrac2{1+\eta},\tfrac2{1+2M-\eta}\right)&\text{if }\gamma\ge\gamma_\star.
\end{cases}
\label{eq:two-step-optimal-pair}
\end{equation}
For simplicity, when the dependence on $\gamma$ is clear from the context, we suppress it and write $(\boldsymbol{\alpha}^\star,\alpha_0^\star,\alpha_1^\star)$ in place of $(\boldsymbol{\alpha}^\star(\gamma),\alpha_0^\star(\gamma),\alpha_1^\star(\gamma))$ throughout the remainder of the paper.

Next, we present the optimal two-step schedule in the following theorem.

\begin{theorem}[Optimal two-step schedule]\label{thm:two-step}
Under Assumption~\ref{ass:oracle}, the thresholds satisfy $0<\gamma_1<\gamma_2<\gamma_\star$. For every $\gamma\ge0$ with $\gamma\ne\gamma_\star$, the two-step schedule in \eqref{eq:two-step-optimal-pair} is the unique global minimizer of $\D_2(\boldsymbol{\alpha})$ over $(\alpha_0, \alpha_1)\in\R_{\ge0}^2$. At $\gamma=\gamma_\star$, the set of global minimizers is exactly
\begin{equation*}
\operatorname*{arg\,min}_{(\alpha_0, \alpha_1)\in\R_{\ge0}^2}\D_2(\boldsymbol{\alpha})=\left\{(\bar\alpha,\bar\alpha),\left(\tfrac2{1+\eta},\tfrac2{1+2M-\eta}\right)\right\}.
\end{equation*}
For every $\gamma\ge0$, the optimal worst-case error is
\begin{equation}
\min_{\alpha_0,\alpha_1\ge0} \D_2(\boldsymbol{\alpha})=\D_2(\boldsymbol{\alpha}^\star) =
\begin{cases}
\tfrac{\gamma}{1+2\gamma}\cdot \sigma^2 &\text{if } 0\le\gamma\le\gamma_1\\
\tfrac{\gamma(M-1)^2+4}{\gamma(M-1)^2+4+(M+1)^2}\cdot \sigma^2 &\text{if }\gamma_1\le\gamma\le\gamma_2\\
(\gamma T(1)+U(1))\cdot \sigma^2 &\text{if }\gamma_2\le\gamma\le\gamma_\star\\[3pt]
(\gamma T(\hat a)+U(\hat a))\cdot \sigma^2 &\text{if }\gamma\ge\gamma_\star.
\end{cases}
\label{eq:two-step-optimal-value}
\end{equation}
The adjacent stepsize formulas~\eqref{eq:two-step-optimal-pair} agree at $\gamma_1$ and $\gamma_2$, and the adjacent value formulas~\eqref{eq:two-step-optimal-value} agree at $\gamma_1$, $\gamma_2$, and $\gamma_\star$.
\end{theorem}

\begin{proof}
This proof is deferred to Appendix~\ref{app:two-step-optimal-proof}.
\end{proof}


The proofs of Theorems~\ref{thm:one-step} and~\ref{thm:two-step} are deferred to the appendix and follow similar approaches. 
For $k\in\{1,2\}$, we derive an upper-bounding function $U_k$ for $\D_k$ on a suitable subset of the stepsize domain containing the proposed schedule $\boldsymbol{\alpha}^\star$. We then construct explicit objective functions and admissible stochastic oracles to obtain a lower-bounding function $L_k$ valid over the entire domain. At the proposed schedule, we show that
\begin{equation*}
U_k(\boldsymbol{\alpha}^\star)=\D_k(\boldsymbol{\alpha}^\star)=L_k(\boldsymbol{\alpha}^\star)=\min_{\boldsymbol{\alpha}\in\mathbb{R}_{\geq0}^k}L_k(\boldsymbol{\alpha})\qquad \text{for } k\in\{1,2\}.
\end{equation*}
Since $L_k(\boldsymbol{\alpha})\leq\D_k(\boldsymbol{\alpha})$ for every nonnegative schedule $\boldsymbol{\alpha}$, this identity establishes the global optimality of the schedules proposed in~\eqref{eq:one-step-optimal} and~\eqref{eq:two-step-optimal-pair}, together with the exact optimal values shown in~\eqref{eq:one-step-optimal-value} and~\eqref{eq:two-step-optimal-value}.

Figures~\ref{fig:schedule-structure}--\ref{fig:optimal-value-gain} illustrate the thresholds $\{\gamma_1,\gamma_2,\gamma_{\star}\}$, the optimal stepsize schedule in~\eqref{eq:two-step-optimal-pair}, and the relative worst-case terminal error $\frac{\D_2(\boldsymbol{\alpha}^\star)}{\D_2(\bar{\boldsymbol{\alpha}})}$, with the capped pair $(\bar\alpha,\bar\alpha)$ serving as the benchmark.
\begin{figure}[ht]
\centering
\includegraphics[width=\textwidth]{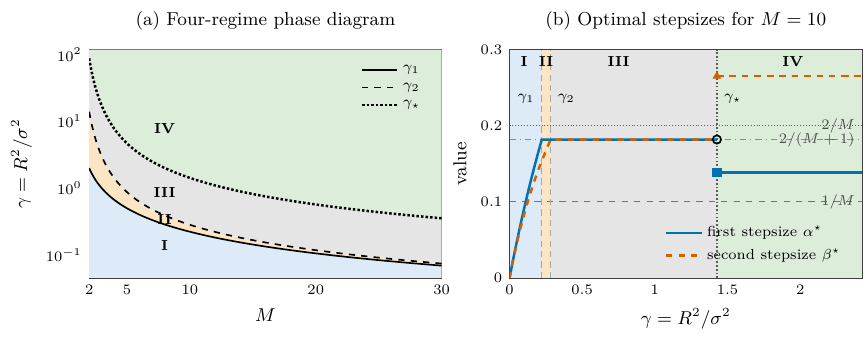}
\caption{Optimal two-step stepsize schedule structure. Subfigure~\textup{(a)} shows the four regimes in~\eqref{eq:two-step-optimal-pair}. Subfigure~\textup{(b)} shows the values of the first stepsize $\alpha_0^\star$ and the second stepsize $\alpha_1^\star$ across all regimes when $M=10$.}
\label{fig:schedule-structure}
\end{figure}
\begin{figure}[ht]
\centering
\includegraphics[width=\textwidth]{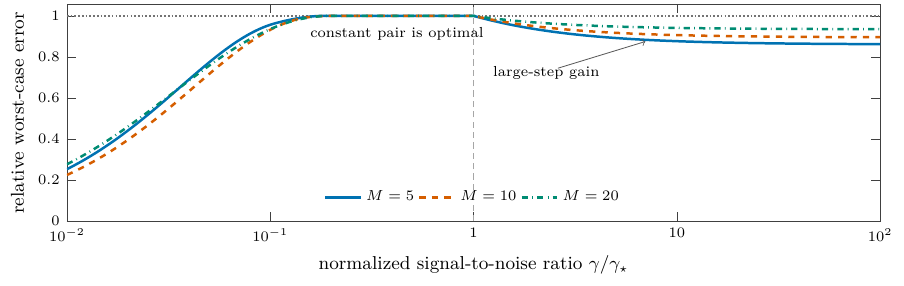}
\caption{Relative worst-case error achieved by the stepsize pair $(\alpha_0^\star,\alpha_1^\star)$ in~\eqref{eq:two-step-optimal-pair}, with the error attained by the capped pair $(\bar\alpha,\bar\alpha)$ serving as the benchmark; that is, $\frac{\D_2(\boldsymbol{\alpha}^\star)}{\D_2(\bar{\boldsymbol{\alpha}})}$.}
\label{fig:optimal-value-gain}
\end{figure}
\par In particular, Figure~\ref{fig:schedule-structure}(a) shows how the three thresholds partition the $(M,\gamma)$-space into the four regimes of~\eqref{eq:two-step-optimal-pair}, while Figure~\ref{fig:schedule-structure}(b) illustrates the corresponding two-step stepsizes for $M=10$. These regimes describe how the optimal two-step stepsize schedule $(\alpha_0^\star,\alpha_1^\star)$ changes as the ratio $\gamma = \frac{R^2}{\sigma^2}$ increases. In Regime~I ($0\leq \gamma\leq \gamma_1$), $R$ is small relative to the noise level $\sigma$, so both stepsizes are small to limit the adverse effect of noise perturbations. Regime~II ($\gamma_1 \le \gamma \le \gamma_2$) is a transition: the first stepsize $\alpha_0^\star$ remains constant at $\bar\alpha=\frac{2}{M+1}$, while the second stepsize $\alpha_1^\star$ continues to increase with $\gamma$ as the relative influence of stochastic noise decreases. In Regime~III ($\gamma_2\leq \gamma\leq \gamma_\star$), the capped pair $(\bar\alpha,\bar\alpha)$ is optimal, representing an intermediate balance between iterate convergence and noise perturbations. In Regime~IV, when $\gamma$ exceeds $\gamma_\star$, the initial-error term dominates the variance accumulation term, and our proposed optimal stepsize pair $(\alpha_0^\star,\alpha_1^\star)$ switches to the deterministic two-step silver stepsize schedule, whose second stepsize $\alpha_1^\star$ may take a long step that no longer guarantees per-iteration descent~\cite{altschuler2025acceleration}. Figure~\ref{fig:optimal-value-gain} compares the worst-case terminal error achieved by the optimal two-step stepsize schedule $(\alpha_0^\star,\alpha_1^\star)$ with the error attained by the capped pair $(\bar\alpha,\bar\alpha)$. The ratio $\frac{\D_2(\boldsymbol{\alpha}^\star)}{\D_2(\bar{\boldsymbol{\alpha}})}$ equals one throughout Regime~III ($\gamma_2\leq\gamma\leq \gamma_\star$), where $(\bar\alpha,\bar\alpha)$ is optimal, and is smaller than one in all other regimes. The horizontal axis in Figure~\ref{fig:optimal-value-gain} is normalized by $\gamma_\star$, which is fixed for a given $M$. The axis could instead be normalized by $\gamma_1$ or $\gamma_2$; the choice affects only the visualization.

\subsection{Two Composition Rules}\label{subsec:composition-rules}
The one- and two-step optimal schedules reveal two distinct approaches to constructing an optimal two-step schedule from one-step schedules. These approaches naturally motivate two composition rules for constructing longer schedules.

By observation, the first three two-step regimes in~\eqref{eq:two-step-optimal-pair} coincide with stepsizes satisfying $\alpha_0^\star=\alpha_*(\gamma)$ and $\alpha_1^\star=\alpha_*(\mathcal T(\gamma))$; see~\eqref{eq:one-step-optimal}--\eqref{eq:one-step-optimal-value}.
Therefore, the second stepsize $\alpha_1^\star$ is obtained by applying the optimal one-step stepsize to the expected error bound after the first step. This suggests composing two blocks by retaining their stepsizes and running them in sequence, with the second block chosen using the error bound after the first. Because this construction 
depends on the error bound rather than realized iterates, the complete schedule remains predetermined.

In the meantime, the fourth regime in~\eqref{eq:two-step-optimal-pair} suggests a different construction. Applying the capped one-step schedule twice gives $(\bar\alpha,\bar\alpha)$, whereas the optimal two-step schedule pair jointly shortens the first stepsize and enlarges the second stepsize. For any function belonging to $\mathcal F_{1,M}$, its curvatures always stay in $[1,M]$, and we call $1$ and $M$ the \textit{endpoint curvatures}. When running the SG method with stepsizes $(\alpha_0,\alpha_1)$ to minimize quadratic functions $\frac{\lambda}{2}\|x-x_\star\|^2$, where $\lambda\in[1,M]$, the SG update results in $\|x_2 - x_\star\| = |(1-\lambda\alpha_0)(1-\lambda\alpha_1)|\cdot\|x_0-x_\star\|$. In~\eqref{eq:composition-two-step-motivation}, we describe this initial-error multiplier, i.e., $|(1-\lambda\alpha_0)(1-\lambda\alpha_1)|$, at the two endpoint curvatures $\lambda\in\{1,M\}$. For the fourth-regime stepsize pair in \eqref{eq:two-step-optimal-pair}, using parameters defined in~\eqref{eq:one-two-notation} and~\eqref{eq:two-step-parameters}, we have
\begin{equation}
|(1-\alpha_0^\star)(1-\alpha_1^\star)|=|(1-M\alpha_0^\star)(1-M\alpha_1^\star)|=\frac{1-\hat a/M}{1+\hat a/M}<(1-\bar\alpha)^2.\label{eq:composition-two-step-motivation}
\end{equation}
The first equality balances the two endpoint curvatures, and the strict inequality improves their common contraction over the capped pair $(\bar\alpha,\bar\alpha)$. We seek a multi-step composition that retains this balance while modifying only the last step of each child block.
Moreover, such a recursive multi-step composition policy also requires characterizing the relationship between a block’s last step and its full contraction product, which measures the decay of the initial error. To identify this relation in the one- and two-step optimal schedules, we define
\begin{equation}
e(t)=\frac{1+Mt}{M(1+t)}\qquad \text{and} \qquad \tau(t)=\left(\frac{1-t}{1+t}\right)^2,\qquad \forall t\in(0,1).\label{eq:composition-terminal-parameters}
\end{equation}
We refer to $\bar\alpha$ in~\eqref{eq:one-two-notation} as the \textit{capped stepsize}, since it is the largest stepsize that the optimal one-step schedule can take (see~\eqref{eq:one-step-optimal}). We also refer to the stepsize pair in the last regime of~\eqref{eq:two-step-optimal-pair}, corresponding to $\gamma\geq \gamma_\star$, as the \textit{two-step silver stepsizes}, following the original definition in~\cite{altschuler2025acceleration}.
The capped stepsize and the two-step silver stepsizes satisfy
\begin{equation}
\begin{aligned}
\bar\alpha&=e(1/M),\qquad |1-\bar\alpha|=|1-M\bar\alpha|=\sqrt{\tau(1/M)},\\ \alpha_0^\star&=e\!\left(\frac1{M\hat a}\right),\qquad \alpha_1^\star=e(\hat a/M),\qquad \text{and} \qquad T(\hat a)=\tau(\hat a/M),
\end{aligned}
\label{eq:composition-one-two-interface}
\end{equation}
where parameter $\hat{a}$ and function $T$ are defined in~\eqref{eq:two-step-parameters}--\eqref{eq:two-step-TU}.
From~\eqref{eq:composition-one-two-interface}, we observe that the capped stepsize and the two-step silver stepsizes share similar properties: by~\eqref{eq:composition-two-step-motivation} and~\eqref{eq:composition-one-two-interface},
\begin{itemize}
\item the capped stepsize satisfies $\bar\alpha=e(t)$ and $|1-\bar\alpha|=|1-M\bar\alpha|=\frac{1-t}{1+t}$ with $t=\frac1M$, and
\item the two-step silver stepsizes have its last stepsize $\alpha_1^\star=e(t)$ and $|(1-\alpha_0^\star)(1-\alpha_1^\star)| = \left|(1-M\alpha_0^\star)(1-M\alpha_1^\star)\right| = \frac{1-t}{1+t}$ with $t=\frac{\hat a}{M}$.
\end{itemize}
Thus, in both examples, a parameter $t$ determines the last stepsize through the function $e(t)$ (see~\eqref{eq:composition-terminal-parameters}). Moreover, for the quadratic objectives $\frac{1}{2}\|x\|^2$ and $\frac{M}{2}\|x\|^2$, the corresponding contraction products, which measure the decay of the initial error, both equal to $\frac{1-t}{1+t}$. 
We call $t$ the \textit{terminal parameter} because it parametrizes the last stepsize in these schedules.

We generalize aforementioned common properties 
by requiring a block $A=(a_0,\ldots,a_{p-1})$ with terminal parameter $u\in(0,1)$ to satisfy
\begin{equation}
a_{p-1}=e(u)\qquad  \text{and} \qquad \left|\prod_{k=0}^{p-1}(1-a_k)\right|=\left|\prod_{k=0}^{p-1}(1-Ma_k)\right|=\frac{1-u}{1+u}=\sqrt{\tau(u)}.\label{eq:composition-block-interface}
\end{equation}

Suppose another block $B=(b_0,\ldots,b_{q-1})$ satisfies the same relations with terminal parameter $v\in(0,1)$. By retaining both internal prefixes and replacing the terminal stepsizes $(e(u),e(v))$ with $(e(y),e(z))$, where $\{y,z\}\subset(0,1)$ are also terminal parameters. By~\eqref{eq:composition-block-interface}, the two endpoint-product magnitudes of schedule $A$, i.e., 
$\left|\prod_{k=0}^{p-1}(1-a_k)\right|$ and $\left|\prod_{k=0}^{p-1}(1-Ma_k)\right|$, are both $\frac{1-u}{1+u}$, and those of $B$ are $\frac{1-v}{1+v}$. The joined block replaces only the terminal stepsizes $(e(u),e(v))$ with $(e(y),e(z))$. Thus, each original terminal factor would be replaced by its updated counterpart. By 
\begin{equation}
1-e(t)=\frac{M-1}{M(1+t)}\qquad \text{and} \qquad Me(t)-1=\frac{(M-1)t}{1+t} \qquad \forall t\in(0,1),\label{eq:composition-terminal-factors}
\end{equation}
the product magnitude at endpoint curvature being $1$ (for minimizing function $\frac{1}{2}\|x-x_\star\|^2$) is
\begin{equation*}
    \frac{1-u}{1+u}\cdot\frac{1-e(y)}{1-e(u)}\cdot\frac{1-v}{1+v}\cdot\frac{1-e(z)}{1-e(v)} = \frac{(1-u)(1-v)}{(1+y)(1+z)}.
\end{equation*}
Similarly, the product magnitude at endpoint curvature $M$ becomes
\begin{equation*}
    \frac{1-u}{1+u}\cdot\frac{Me(y)-1}{Me(u)-1}\cdot\frac{1-v}{1+v}\cdot\frac{Me(z)-1}{Me(v)-1} = \frac{yz}{uv}\cdot\frac{(1-u)(1-v)}{(1+y)(1+z)}. 
\end{equation*}
Keeping the endpoint balance from \eqref{eq:composition-two-step-motivation} essentially requires $yz=uv$. The terminal stepsize of the joined block is 
$e(z)$, so preserving the full interface \eqref{eq:composition-block-interface} requires the common product to equal $\sqrt{\tau(z)}$. This gives
\begin{equation}
yz=uv\quad \text{and} \quad \frac{(1-u)(1-v)}{(1+y)(1+z)} = \frac{1-z}{1+z},\quad \text{implying} \ \ z^2-(u+v-2uv)z-uv=0.\label{eq:composition-kernel-identities}
\end{equation}
In fact, substituting the first condition into the second yields the third. The second condition lets the joined block participate in future join using the same rule. In the original two-step construction, the children blocks have $u=v=\frac{1}{M}$, and the replacements have $y=\frac{1}{M\hat a}$ and $z=\frac{\hat a}{M}$, exactly as in \eqref{eq:composition-one-two-interface}. Thus the proposed rule extends the observed one-to-two-step construction.

We refer to a finite ordered list of nonnegative stepsizes as a \textit{stepsize block} and to the one-step block $(\bar\alpha)$, which has terminal parameter $\frac{1}{M}$, as the \textit{seed block}.
\begin{definition}[\texttt{serial} join]
\label{def:serial-join}
For any $\{p,q\}\subset\mathbb{N}\backslash\{0\}$, consider two stepsize blocks
\begin{equation*}
A(\gamma_A)=(a_0,\ldots,a_{p-1}) \qquad \text{and} \qquad B(\gamma_B)=(b_0,\ldots,b_{q-1})
\end{equation*}
generated according to their respective block construction policies with initial error ratios $\gamma_A$ and $\gamma_B$. 
Suppose that applying block $A(\gamma_A)$ yields the uniform guarantee $\E[\|x_p-x_\star\|^2]\le\tau_A\|x_0-x_\star\|^2+V_A\sigma^2$ for any initial iterate $x_0\in\R^d$, any $f\in\mathcal{F}_{1,M}$, and any stochastic oracle satisfying Assumption~\ref{ass:oracle}. If the initial iterate satisfies $\|x_0-x_\star\|^2\le\gamma_A\sigma^2$, then the expected squared error after applying block $A(\gamma_A)$ satisfies $\E[\|x_p-x_\star\|^2] \leq (\tau_A\gamma_A + V_A)\sigma^2$. We can therefore define $\hat\gamma_B:=\tau_A\gamma_A+V_A$ as an updated initial error ratio for the second block. Specifically, we construct 
$B(\hat\gamma_B)=(\hat b_0,\ldots,\hat b_{q-1})$ using the same construction policy as $B(\gamma_B)$, but with the updated initial error ratio $\hat\gamma_B$. The \texttt{serial} join of $A$ and $B$ is then defined as 
\begin{equation*}
(A\oplus B)(\gamma_A)=(a_0,\ldots,a_{p-1},\hat b_0,\ldots,\hat b_{q-1}).
\end{equation*}
\end{definition}

In Definition~\ref{def:serial-join}, $\gamma_A$ and $\gamma_B$ specify the initial error ratios used to construct blocks $A$ and $B$, respectively. The \texttt{serial} join first applies block $A$ and then applies block $B$, starting from the terminal iterate of $A$. Importantly, the updates initial error ratios $(\gamma_A,\hat\gamma_B)$ and all stepsizes in the resulting schedule are determined before any stochastic oracle calls are made.

\begin{definition}[\texttt{silver} join and \texttt{silver} family]\label{def:silver-join}
The \texttt{silver} family is a class of stepsize blocks, including the seed block $(\bar\alpha)$ whose terminal parameter is $\frac{1}{M}$.
For any $\{p,q\}\subset\mathbb{N}_+$ and two blocks $A=(a_0,\ldots,a_{p-1})$ and $B=(b_0,\ldots,b_{q-1})$ with terminal parameters $u\in(0,1)$ and $v\in(0,1)$, respectively, by setting
\begin{equation}
z=\frac{u+v-2uv+\sqrt{(u+v-2uv)^2+4uv}}2\qquad \text{and} \qquad y=\frac{uv}{z},\label{eq:silver-join-parameters}
\end{equation}
we define the \texttt{silver} join $\bowtie$ as
\begin{equation}
A\bowtie B=(a_0,\ldots,a_{p-2},e(y),b_0,\ldots,b_{q-2},e(z)),\label{eq:silver-join-definition}
\end{equation}
whose terminal parameter is $z$. The resulting schedule $A\bowtie B$ belongs to the \texttt{silver} family if blocks $A$ and $B$ both belong to the \texttt{silver} family. We write $\mathcal S_N$ to represent the class of all $N$-step schedules that belong to the \texttt{silver} family.
\end{definition}
Equation~\eqref{eq:silver-join-parameters} implies $0<y<\min\{u,v\}\le\max\{u,v\}<z<1$. Since $e$ is a monotonically increasing function (see~\eqref{eq:composition-terminal-parameters}), the \texttt{silver} join reduces the terminal stepsize of the first block while increasing that of the second block. In this paper, we apply the \texttt{silver} join only to two blocks from the \texttt{silver} family, rather than to arbitrary short-step block.

\subsection{Performance Guarantees under Composition}\label{subsec:composition-bounds}
We say that the stepsize schedule $A=(a_0,\ldots,a_{p-1})$ admits a block guarantee with coefficients $(\tau_A,V_A)\in\R_{\geq0}^2$ if, for any starting point $x_0\in\R^d$, any $f\in\mathcal F_{1,M}$, and any stochastic oracle satisfying Assumption~\ref{ass:oracle}, the terminal iterate always satisfies
\begin{equation}
\E[\|x_p^A-x_\star\|^2]\le\tau_A\|x_0-x_\star\|^2+V_A\sigma^2,\label{eq:composition-uniform-guarantee}
\end{equation}
where $\sigma^2$ is the (conditional) variance bound defined in Assumption~\ref{ass:oracle}.

The following two theorems establish the resulting guarantees under the \texttt{serial} and \texttt{silver} composition rules, respectively.
\begin{theorem}[\texttt{Serial} bound]\label{thm:serial-bound}
Let $A(\gamma_A)$ and $B(\gamma_B)$ be stepsize blocks of lengths $p$ and $q$, respectively. Suppose that $A(\gamma_A)$ has a uniform guarantee with coefficients $(\tau_A,V_A)$ as in~\eqref{eq:composition-uniform-guarantee}. To form $(A\oplus B)(\gamma_A)$ as in Definition~\ref{def:serial-join}, set $\hat\gamma_B=\tau_A\gamma_A+V_A$ and update $\gamma_B = \hat\gamma_B$. Let $(\tau_B,V_B)$ be any uniform guarantee coefficients for $B(\gamma_B)$ evaluated at the updated $\gamma_B$. Then the joined block has a uniform guarantee with coefficients $(\tau_{A\oplus B},V_{A\oplus B})$ and satisfies
\begin{equation*}
\tau_{A\oplus B}=\tau_A\tau_B,\quad V_{A\oplus B}=\tau_B V_A+V_B,\quad \text{and} \quad \D_{p+q}(A\oplus B)\le\tau_{A\oplus B}R^2+V_{A\oplus B}\sigma^2.
\end{equation*}
\end{theorem}
\begin{proof}
Let's consider an arbitrary admissible instance for selections of the initial iterate $x_0\in\R^d$, objective function $f\in\mathcal F_{1,M}$, and stochastic oracle satisfying Assumption~\ref{ass:oracle}, and we consider the iterate sequence generated by iterative updates described in~\eqref{eq:sg-update} using the schedule $A\oplus B$. By Definition~\ref{def:serial-join}, the last $q$ steps use $B(\hat\gamma_B)$ and start at $x_p$. Conditional on $\mathcal H_p$, the point $x_p$ is fixed, and the remaining oracle satisfies~\eqref{eq:unbiasedness} under the conditional law for almost every history by disintegration. Applying~\eqref{eq:composition-uniform-guarantee} to $B(\hat\gamma_B)$ gives
\begin{equation*}
\E[\|x_{p+q}-x_\star\|^2\mid\mathcal H_p]\le\tau_B\|x_p-x_\star\|^2+V_B\sigma^2.
\end{equation*}
Taking expectations and applying~\eqref{eq:composition-uniform-guarantee} to $A(\gamma_A)$ gives
\begin{align*}
\E[\E[\|x_{p+q}-x_\star\|^2\mid\mathcal H_p]]&\le\E[\tau_B\|x_p-x_\star\|^2+V_B\sigma^2] \\
&= \tau_B\E[\|x_p-x_\star\|^2] +V_B\sigma^2 \leq \tau_A\tau_B\|x_0-x_\star\|^2 + (\tau_BV_A + V_B)\sigma^2,
\end{align*}
which concludes the statement by the help of $\|x_0-x_\star\|\le R$.
\end{proof}

\begin{theorem}[\texttt{Silver} bound]\label{thm:silver-bound}
Let $A=(a_0,\ldots,a_{p-1})\in\mathcal S_p$ and $B=(b_0,\ldots,b_{q-1})\in\mathcal S_q$ belong to the \texttt{silver} family and have terminal parameters $u\in(0,1)$ and $v\in(0,1)$, respectively. With functions $e(\cdot)$ and $\tau(\cdot)$ defined in~\eqref{eq:composition-terminal-parameters}, let's suppose both
blocks $A$ and $B$ have uniform guarantees as in~\eqref{eq:composition-uniform-guarantee}, where coefficients of block guarantees and performance bounds are
\begin{equation}
\begin{aligned}
\tau_A&=\tau(u),\qquad V_A=\sum_{k=0}^{p-2}\left(a_k^2\prod_{j=k+1}^{p-1}(1-Ma_j)^2\right)+a_{p-1}^2,\\
\tau_B&=\tau(v),\qquad V_B=\sum_{k=0}^{q-2}\left(b_k^2\prod_{j=k+1}^{q-1}(1-Mb_j)^2\right)+b_{q-1}^2,\\
\D_p(A)&\le\tau_A R^2+V_A\sigma^2,\qquad \text{and} \qquad \D_q(B)\le\tau_B R^2+V_B\sigma^2.
\end{aligned}
\label{eq:silver-performance-bound}
\end{equation}
Using $(y,z)$ defined in~\eqref{eq:silver-join-parameters} and  setting
\begin{equation}
s_A=\frac{\tau(z)}{\tau(u)}\qquad \text{and} \qquad s_B=\left(\frac{z(1+v)}{v(1+z)}\right)^2,\label{eq:silver-inheritance-scales}
\end{equation}
we have their \texttt{silver} join $A\bowtie B$ has the uniform guarantee as presented in~\eqref{eq:composition-uniform-guarantee} with
\begin{equation}
\begin{aligned}
\tau_{A\bowtie B}&=\tau(z)=s_A\tau_A,\\
V_{A\bowtie B}&=s_A[V_A-e(u)^2]+s_B[V_B-e(v)^2]+s_B\tau_B e(y)^2+e(z)^2,\\
\text{and}\quad \D_{p+q}(A\bowtie B)&\le\tau_{A\bowtie B}R^2+V_{A\bowtie B}\sigma^2.
\end{aligned}
\label{eq:silver-bound-transform}
\end{equation}
\end{theorem}
\begin{proof}
This proof is deferred to Appendix~\ref{app:silver-weights}.
\end{proof}


Next, we discuss long schedules constructed by the combination of \texttt{serial} and \texttt{silver} joins.

\noindent \textbf{Schedules constructed using only \texttt{serial} joins.}
Starting from the optimal one-step rule in Theorem~\ref{thm:one-step}, we construct a schedule of any length $N$ by \texttt{serial} joins. The first one-step block uses $\alpha_*(\gamma)$ and gives the error bound $\sigma^2\mathcal T(\gamma)$, where $\gamma$ is defined in~\eqref{eq:one-two-notation} and $(\alpha_*(\gamma),\mathcal T(\gamma))$ are defined in~\eqref{eq:one-step-optimal}--\eqref{eq:one-step-optimal-value}. By choosing each subsequent one-step block by applying the same rule to the expectation error bound after taking all preceding steps, we use the composition notation in~\eqref{eq:serial-map-iteration} and generate the resulting $N$-step schedule with
\begin{equation}
\alpha_i=\alpha_*\!\left(\mathcal T^{\circ i}(\gamma)\right),\qquad \forall i\in[N-1].
\label{eq:pure-serial-steps}
\end{equation}
For any fixed $(M,N,\gamma)$, all these one-step blocks are determined before stochastic oracle calls. By Definition~\ref{def:serial-join}, a \texttt{serial} join concatenates its two input lists. Thus, joining the one-step blocks $(\alpha_0),\ldots, (\alpha_{N-1})$ in this order always produces the list in~\eqref{eq:pure-serial-steps}, regardless of how the successive joins are parenthesized. The input ratio is updated from one block to the next. Since the one-step rule in~\eqref{eq:one-step-optimal} returns a stepsize in $[0,\bar\alpha]$ at every input ratio, the same cap applies to all the one-step blocks. Hence the largest stepsize
in the joined schedule~\eqref{eq:pure-serial-steps} is at most $\bar\alpha$.

Next lemma analyzes properties of stepsize schedules that only use \texttt{serial} joins.

\begin{lemma}[Optimality of the purely \texttt{serial} schedule]\label{lem:serial-short-optimal}
Under Assumption~\ref{ass:oracle}, the $N$-step purely \texttt{serial} schedule in~\eqref{eq:pure-serial-steps} minimizes the exact worst-case error in Proposition~\ref{prop:exact} over all predetermined schedules with stepsizes in $[0,\bar\alpha]$. In particular,
\begin{equation}
\min_{\boldsymbol{\alpha}\in[0,\bar\alpha]^N}\D_N(\boldsymbol{\alpha})=\sigma^2\mathcal T^{\circ N}(\gamma).\label{eq:serial-short-optimal-value}
\end{equation}
Here, parameters $(\gamma,\bar\alpha)$ are defined in~\eqref{eq:one-two-notation}, and $\mathcal T^{\circ N}(\gamma)$ is defined in~\eqref{eq:serial-map-iteration} and~\eqref{eq:one-step-optimal-value}.
\end{lemma}

\begin{proof}
By Proposition~\ref{prop:exact}, it suffices to minimize
$\tau_N(\boldsymbol{\alpha})R^2+V_N(\boldsymbol{\alpha})\sigma^2$
in~\eqref{eq:exact}, with coefficients defined
in~\eqref{eq:coefficients}.
For a fixed schedule $\boldsymbol{\alpha}\in[0,\bar\alpha]^N$,
the coefficient formulas in~\eqref{eq:coefficients} give $\frac{\D_{k+1}(\boldsymbol{\alpha})}{\sigma^2} = (1-\alpha_k)^2\cdot\frac{\D_k(\boldsymbol{\alpha})}{\sigma^2}+\alpha_k^2$ for all $k\in[N-1]\backslash\{0\}$.
For every $\gamma\ge0$, Theorem~\ref{thm:one-step} gives
$\min_{0\le \alpha\le\bar\alpha}\{(1-\alpha)^2\gamma+\alpha^2\}=\mathcal T(\gamma)$,
attained at $\alpha=\alpha^*(\gamma)$ defined in~\eqref{eq:one-step-optimal}.

We prove by induction that
$\D_k(\boldsymbol{\alpha})/\sigma^2\ge\mathcal T^{\circ k}(\gamma)$.
The case $k=1$ follows from Theorem~\ref{thm:one-step}.
If such claim holds at $k$, then
\begin{equation*}
\frac{\D_{k+1}(\boldsymbol{\alpha})}{\sigma^2}
=(1-\alpha_k)^2\cdot\frac{\D_k(\boldsymbol{\alpha})}{\sigma^2}+\alpha_k^2\ge(1-\alpha_k)^2\cdot\mathcal T^{\circ k}(\gamma)+\alpha_k^2\ge\mathcal T\bigl(\mathcal T^{\circ k}(\gamma)\bigr)
=\mathcal T^{\circ(k+1)}(\gamma).
\end{equation*}
Hence, $\D_N(\boldsymbol{\alpha}) \ge\sigma^2\mathcal T^{\circ N}(\gamma)$ for every admissible schedule $\boldsymbol{\alpha}\in[0,\bar\alpha]^N$.
Now let's choose
$\alpha_k^\star=\alpha^*(\mathcal T^{\circ k}(\gamma))$
for every $k\in[N-1]$ as in~\eqref{eq:pure-serial-steps}.
The first step attains
$\D_1(\alpha_0^\star)/\sigma^2=\mathcal T(\gamma)$.
The recurrence and one-step optimality then give
$\D_k((\alpha_0^\star,\ldots,\alpha_{k-1}^\star))/\sigma^2
=\mathcal T^{\circ k}(\gamma)$ by induction.
Thus this schedule attains the lower bound, proving $\min_{\boldsymbol{\alpha}\in[0,\bar\alpha]^N}
\D_N(\boldsymbol{\alpha})
=\sigma^2\mathcal T^{\circ N}(\gamma)$.
\end{proof}

\noindent \textbf{Schedules constructed using only \texttt{silver} joins.}
Starting from the capped seed $(\bar\alpha)$ in Definition~\ref{def:silver-join}, recursive \texttt{silver} joins also produce schedules of any length $N$. Here the choice of left and right block lengths at each join can change the resulting stepsize schedule. Once all these choices are fixed, the complete schedule is predetermined.

The one-step seed has its only step $\bar\alpha$ at the last position. Suppose two input \texttt{silver} blocks $A=(a_0,\ldots,a_{p-1})$ and $B=(b_0,\ldots,b_{q-1})$ have their largest steps at the last positions, with $a_{p-1}=e(u)$ and $b_{q-1}=e(v)$ for their terminal parameters $\{u,v\}\subset(0,1)$. The \texttt{silver} join in~\eqref{eq:silver-join-definition} replaces these two steps by $e(y)$ and $e(z)$, respectively, where $\{y,z\}\subset(0,1)$ are also terminal parameters. The parameters in~\eqref{eq:silver-join-parameters} satisfy $y<\min\{u,v\}\le\max\{u,v\}<z$. Since function $E(\cdot)$ in~\eqref{eq:composition-terminal-parameters} is strictly monotonically increasing,
\begin{equation*}
e(y)<\min\{a_{p-1},b_{q-1}\}\le\max\{a_{p-1},b_{q-1}\}<e(z).
\end{equation*}
Every retained internal step is at most $\max\{a_{p-1},b_{q-1}\}$ and hence is also smaller than $e(z)$. Applying this argument recursively from the seed shows that the last step of every \texttt{silver} block is its unique largest stepsize, and each \texttt{silver} join makes this last step strictly larger than every stepsize of either input block. Consequently, every \texttt{silver} block with at least two steps contains a stepsize larger than $\bar\alpha$.

\noindent \textbf{Combining the two joins.}
The \texttt{silver} join in Definition~\ref{def:silver-join} takes two \texttt{silver} blocks as inputs. The \texttt{serial} join in Definition~\ref{def:serial-join} accepts arbitrary stepsize blocks, including \texttt{silver} blocks and blocks already constructed by \texttt{serial} joins. We can therefore first construct blocks using \texttt{silver} joins and \texttt{serial} joins, then using \texttt{serial} joins to combine them as one schedule. The resulting mixed schedule retains all the large stepsizes of its \texttt{silver} block.

\begin{rem}
The optimality result in Lemma~\ref{lem:serial-short-optimal} explains the roles of the two composition rules. Successive \texttt{serial} joins using the optimal one-step choices attain the smallest worst-case error when all stepsizes are restricted to $[0,\bar\alpha]$. Consequently, any schedule that improves on~\eqref{eq:serial-short-optimal-value} must contain a stepsize larger than $\bar\alpha$. The \texttt{silver} join introduces such larger steps by jointly modifying the final stepsizes of its two constituent blocks, allowing a smaller initial-error coefficient.
\end{rem}

Under bounded-variance noise, decreasing stepsizes is necessary to suppress variance accumulation and obtain convergence. This suggests using \texttt{silver} joins to construct an initial block with large stepsizes, followed by \texttt{serial} joins using the optimal one-step rule when smaller steps are needed. The \texttt{silver} join allows improvement beyond the short-step class, while successive \texttt{serial} joins recover the optimal short-step schedule established above. Whether the \texttt{silver} part improves the overall guarantee depends on the balance between its smaller initial-error coefficient and its variance accumulation.

Given the initial ratio $\gamma=\frac{R^2}{\sigma^2}$, the remaining question is how many steps to allocate to the \texttt{silver} block before switching to short steps connected by \texttt{serial} joins. The comparison between the two joins illustrates this balance. For two \texttt{silver} blocks $A$ and $B$, the \texttt{silver} join has a smaller initial-error coefficient than their \texttt{serial} composition. In terms of noise-accumulation coefficients, if $V_{A\bowtie B}>V_{A\oplus B}$, their performance upper bounds (see~\eqref{eq:worst-case-error} and~\eqref{eq:exact}) cross at
\begin{equation}
\gamma_{A,B}=\frac{V_{A\bowtie B}-V_{A\oplus B}}{\tau_{A\oplus B}-\tau_{A\bowtie B}}.\label{eq:composition-crossover}
\end{equation}
\texttt{Serial} composition gives the smaller upper bound for $\gamma<\gamma_{A,B}$, and the \texttt{silver} join gives the smaller upper bound for $\gamma>\gamma_{A,B}$; the bounds agree when $\gamma=\gamma_{A,B}$. If $V_{A\bowtie B}\le V_{A\oplus B}$, the \texttt{silver} join gives an upper bound no larger for any $\gamma\ge0$. For two seed blocks, the threshold in~\eqref{eq:composition-crossover} equals $\gamma_\star$.

In the next section, we give an explicit predetermined schedule for every horizon $N$. Its \texttt{silver}-prefix length is selected from the initial parameters, and the remaining one-step schedules are connected by \texttt{serial} joins. The resulting upper bound is no larger than the bounds of either the entirely \texttt{serial} schedule or the full \texttt{silver} block, with large stepsizes included whenever a \texttt{silver} prefix is selected.

\section{A Predetermined Stepsize Schedule for Arbitrary Horizons}\label{sec:arbitrary-horizon-schedule}
We construct the schedule in two stages. First, we form a \texttt{silver} block of $p$ steps from the seed $(\bar\alpha)$ using the \texttt{silver} join in Definition~\ref{def:silver-join}. For any natural number $k\ge2$, let's construct block $S_k$ using the \texttt{silver} join. If $k$ is an even number, we combine two \texttt{silver} blocks of length $\frac{k}{2}$. If $k$ is an odd number, we join a \texttt{silver} block of length $\frac{k-1}{2}$ with another \texttt{silver} block of length $\frac{k+1}{2}$, in that order. For example, the six-step \texttt{silver} block is constructed as
\begin{equation*}
S_6=\underbrace{\bigl[S_1\bowtie\underbrace{(S_1\bowtie S_1)}_{S_2}\bigr]}_{S_3}\ \bowtie\ \underbrace{\bigl[S_1\bowtie\underbrace{(S_1\bowtie S_1)}_{S_2}\bigr]}_{S_3},\qquad \text{where} \ \  S_1=(\bar\alpha).
\end{equation*}
Every $\bowtie$ applies the terminal-step replacements following Definition~\ref{def:silver-join}. Thus the \texttt{silver} block is determined by its length $p$ and smooth parameter $M$.

Second, we append the remaining $(N-p)$ one-step schedules by \texttt{serial} joins. Starting from the deterministic error bound for the \texttt{silver} prefix, we apply the optimal one-step rule from Theorem~\ref{thm:one-step} repeatedly to the successive bounds. This continuation has an initial part of $q_p$ steps equal to $\bar\alpha$, followed by decreasing steps. The construction has the following form when all three parts are present:
\begin{equation*}
\underbrace{\boxed{S_p}}_{\substack{p\text{ steps}\\\text{\texttt{silver} joins}}}\ \oplus\ \underbrace{\underbrace{\boxed{\bar\alpha,\ldots,\bar\alpha}}_{q_p\text{ steps}}\ \oplus\ \underbrace{\boxed{\alpha_{p+q_p},\ldots,\alpha_{N-1}}}_{N-p-q_p\text{ steps}}}_{\text{\texttt{serial} joins of one-step schedules}}.
\end{equation*}
The short-step schedule after the \texttt{silver} block begins at iteration $p$. 
The choices $p=0$ and $p=N$ give the all-\texttt{serial} and all-\texttt{silver} schedules, respectively. We now specify the two lengths and the complete stepsize list directly from parameters $(N,M,R,\sigma)$.

\subsection{Definition of the Schedule}
We first fix horizon $N\ge1$ and use parameters $(\gamma,\bar\alpha,\gamma_1)$ in~\eqref{eq:one-two-notation} to define
\begin{equation}
\rho=(1-\bar\alpha)^2.\label{eq:arbitrary-horizon-rho}
\end{equation}
We then specify the \texttt{silver} blocks, their error coefficients, and the thresholds needed to select the prefix length.

\noindent \textbf{\texttt{Silver} blocks and their error bounds.}
By defining
\begin{equation}
S_1=(\bar\alpha),\qquad t_1=1/M,\qquad \text{and} \qquad S_k=S_{\lfloor k/2\rfloor}\bowtie S_{\lceil k/2\rceil}\quad \forall k\ge2,\label{eq:balanced-silver-blocks}
\end{equation}
terminal parameters $\{t_k\}_{k\in[N]\backslash{0}}$ are obtained from the two constituent blocks by~\eqref{eq:silver-join-parameters}. For $S_k=(s_{k,0},\ldots,s_{k,k-1})$, we set
\begin{equation}
\begin{aligned}
\tau_0 &= 1, \qquad \tau_k=\left(\frac{1-t_k}{1+t_k}\right)^2\quad \forall k\ge1, \\
V_0 &= 0, \qquad V_1=\bar\alpha^2, \qquad V_k=\sum_{i=0}^{k-2}\left(s_{k,i}^2\prod_{j=i+1}^{k-1}(1-Ms_{k,j})^2\right)+s_{k,k-1}^2\quad \forall k\ge2,\\ 
\text{and} \quad z_k&=\tau_k\gamma+V_k\quad \forall k\in[N].
\end{aligned}
\label{eq:balanced-silver-coefficients}
\end{equation}
Theorem~\ref{thm:silver-bound} guarantees $\D_k(S_k)\le\tau_kR^2+V_k\sigma^2=\sigma^2z_k$ for any $k\in[N]\backslash\{0\}$.
Thus, $\tau_k$ and $V_k$ measure contributions of the initial error and the variance accumulation to the upper bound, respectively. We use the scaled upper bound, $z_k$, 
to choose short steps after the \texttt{silver} block. For $k=0$, no \texttt{silver} step is taken and $z_0=\gamma$. Below, $(\tau_k,V_k)$ determine the prefix-selection thresholds, while $z_p$ determines the stepsizes after the selected prefix.

\begin{definition}[Predetermined $N$-step schedule]\label{def:arbitrary-horizon-schedule}
We use the \texttt{silver} blocks and the associated coefficients $(\tau_k,V_k)$, defined in~\eqref{eq:balanced-silver-blocks}--\eqref{eq:balanced-silver-coefficients}, to characterize our predetermined $N$-step schedule, which consists of a \texttt{silver} prefix followed by a short-step block.

\emph{\texttt{silver}-prefix length.}
For each candidate prefix length $k\in[N]$, let's define
\begin{equation}
W_k(\gamma):=\mathcal T^{\circ(N-k)}(z_k) =\mathcal T^{\circ(N-k)}(\tau_k\gamma+V_k), 
\label{eq:prefix-budget-values}
\end{equation}
where $\mathcal T$ and its successive compositions are defined in~\eqref{eq:one-step-optimal-value} and~\eqref{eq:serial-map-iteration}.
Let $p\in[N]$ denote 
the index such that for all $i\in[N]\backslash\{p\}$, 
\begin{equation}
\begin{cases}
W_p(\gamma)<W_i(\gamma)&\text{if }0\le i<p,\\
W_p(\gamma)\le W_i(\gamma)&\text{if }p<i\le N,
\end{cases}
\label{eq:prefix-length-selection}
\end{equation}
implying $p$ is the smallest index attaining the minimum value among $\{W_i(\gamma):i\in[N]\}$. 

\emph{Short-step block.} With the 
\texttt{silver}-prefix length $p\in[N]$, let's define
\begin{equation}
\begin{aligned}
q_p=\min\left\{N-p,\left\lceil\frac{\log\!\left(\max\left\{1,\dfrac{z_p-1/M}{\gamma_1-1/M}\right\}\right)}{\log(1/\rho)}\right\rceil\right\} \quad \text{and} \quad \widehat z_p=\frac1M+\rho^{q_p}\left(z_p-\frac1M\right). 
\end{aligned}
\label{eq:explicit-serial-cap-count}
\end{equation}
where $\rho$ is defined in~\eqref{eq:arbitrary-horizon-rho} and $\gamma_1$ is defined in~\eqref{eq:one-two-notation}.
Here, we consider $z_p=\tau_p\gamma+V_p$ defined in~\eqref{eq:balanced-silver-coefficients}. After the $p$ \texttt{silver} steps $S_p=(s_{p,0},\ldots,s_{p,p-1})$, $q_p$ counts the following steps equal to $\bar\alpha$, and $\widehat z_p$ is the normalized error bound after these $p+q_p$ steps. The remaining $N-p-q_p$ steps take decreasing stepsizes. We define the whole $N$-step schedule $\mathcal A_N(\gamma)=(\alpha_0,\ldots,\alpha_{N-1})$ by 
\begin{equation}
\alpha_i=
\begin{cases}
s_{p,i},&\text{if }0\le i<p,\\[3pt]\bar\alpha,&\text{if }p\le i<p+q_p,\\[3pt]\displaystyle\frac{\widehat z_p}{1+(i-p-q_p+1)\widehat z_p},&\text{if }p+q_p\le i<N.
\end{cases}
\label{eq:arbitrary-horizon-steps}
\end{equation}
\end{definition}

If present, the capped and decreasing \texttt{serial} parts begin at indices $p$ and $p+q_p$, with corresponding stepsizes $\alpha_p$ and $\alpha_{p+q_p}$, respectively. At the cap-transition boundary, the first stepsize prescribed by the decreasing formula may also equal $\bar\alpha$. The schedule is determined entirely by the initial parameters; no information about realized iterates or oracle outputs is required to specify the stepsizes in~\eqref{eq:arbitrary-horizon-steps}.

\begin{lemma}[Well-definedness of the schedule]\label{lem:arbitrary-horizon-well-defined}
For every predetermined parameters $M>1$, $\gamma\ge0$, and natural number $N\ge1$, Definition~\ref{def:arbitrary-horizon-schedule} specifies unique $p\in[N]$ and $q_p\in[N-p]$. 
All prescribed stepsizes are finite, nonnegative, and determined before the oracle calls. Every step after the $p$-step \texttt{silver} prefix lies in $[0,\bar\alpha]$.
\end{lemma}
\begin{proof}
Definition~\ref{def:silver-join} gives finite \texttt{silver} steps and \eqref{eq:balanced-silver-coefficients} defines coefficients $\tau_k>0$ and $V_k\ge0$, resulting in 
$z_k=\tau_k\gamma+V_k$ is always finite and nonnegative. By the finite and nonnegativity properties of the map $\mathcal T$ in~\eqref{eq:one-step-optimal-value}, 
all $W_k(\gamma)$ in~\eqref{eq:prefix-budget-values} are finite and nonnegative, then the minimum of the finite list $\{W_k(\gamma)\}_{k\in[N]}$ and a unique $p\in[N]$ satisfying~\eqref{eq:prefix-length-selection} are attainable. 

Using $\rho$ from~\eqref{eq:arbitrary-horizon-rho}, $z_k$ from~\eqref{eq:balanced-silver-coefficients}, and $\gamma_1$ from~\eqref{eq:one-two-notation}, we have $0<\rho<1$, $z_k\ge0$, and $\gamma_1-\frac{1}{M}>0$. Hence the logarithmic expression in~\eqref{eq:explicit-serial-cap-count} is finite and nonnegative. Its ceiling and truncation give an integer $q_p\in [N-p]$. Also, because of $\widehat z_p=\rho^{q_p}z_p+\frac{1-\rho^{q_p}}{M}\ge0$, the denominator in~\eqref{eq:arbitrary-horizon-steps} is positive. If decreasing steps remain, then $q_p<N-p$, so $q_p$ is given by the second argument of the minimum in~\eqref{eq:explicit-serial-cap-count}. Since $1/\rho>1$, ~\eqref{eq:explicit-serial-cap-count} gives $\rho^{-q_p} \ge \max\left\{1,\frac{z_p-1/M}{\gamma_1-1/M}\right\} \ge \frac{z_p-1/M}{\gamma_1-1/M}$.
Multiplying by $\rho^{q_p}(\gamma_1-1/M)>0$ and adding $1/M$ gives $\widehat z_p=1/M+\rho^{q_p}(z_p-1/M)\le\gamma_1$.
Hence each such step is at most $\frac{\gamma_1}{1+\gamma_1}=\bar\alpha$. All quantities depend only on the initial parameters.
\end{proof}

\begin{rem}[Explicit formula for $p$]
The schedule in Definition~\ref{def:arbitrary-horizon-schedule} has up to three phases, of lengths $p$, $q_p$, and $N-p-q_p$, respectively. Given $p$, formula~\eqref{eq:explicit-serial-cap-count} determines $q_p$ explicitly. Appendix~\ref{app:explicit-p} gives an explicit formula for $p$ and proves that it satisfies the selection conditions in~\eqref{eq:prefix-length-selection}.
\end{rem}

\subsection{Convergence Analysis}

For each possible \texttt{silver}-prefix length $k\in[N]$, we consider the complete schedule in~\eqref{eq:arbitrary-horizon-steps} with $p$ replaced by $k$: it uses $k$ \texttt{silver} steps, $q_k$ steps equal to $\bar\alpha$, and $N-k-q_k$ decreasing steps. The following theorem first proves the bound for every choice of $k$ and then gives its explicit form for the selected length $p$.

\begin{theorem}[Performance of the predetermined schedule]\label{thm:arbitrary-horizon-convergence}
Under Assumption~\ref{ass:oracle}, with parameters $M>1$, $R\geq 0$, $\sigma>0$, positive integer $N$, and 
$\gamma$ from~\eqref{eq:one-two-notation}, by choosing any $k\in[N]$, we let $\boldsymbol{\alpha}$ be the complete $N$-step schedule in~\eqref{eq:arbitrary-horizon-steps} with $p$ replaced by $k$. Its \texttt{silver} block $S_k$ for $k\ge1$ is defined in~\eqref{eq:balanced-silver-blocks}; for $k=0$, no \texttt{silver} prefix is used. The remaining $N-k$ steps use $q_k$ and $\widehat z_k$ from~\eqref{eq:explicit-serial-cap-count}. With $W_k$ defined in~\eqref{eq:prefix-budget-values} and $\D_N$ specified in~\eqref{eq:worst-case-error}, we always have
\begin{equation}
\D_N(\boldsymbol{\alpha})\le\sigma^2W_k(\gamma).\label{eq:prefix-candidate-performance}
\end{equation}

In particular, by taking $k=p$ in~\eqref{eq:prefix-length-selection}, 
using $q_p$ and $\widehat z_p$ in~\eqref{eq:explicit-serial-cap-count}, $\tau_p$ and $V_p$ in~\eqref{eq:balanced-silver-coefficients}, and $\rho$ from~\eqref{eq:arbitrary-horizon-rho}, and defining
\begin{equation}
\begin{aligned}
\widehat\tau_N=\frac{\rho^{q_p}\tau_p}{(1+(N-p-q_p)\widehat z_p)^2}, \quad \text{and} \quad \widehat V_N=\frac{\rho^{q_p}V_p+(1-\rho^{q_p})/M+(N-p-q_p)\widehat z_p^2}{(1+(N-p-q_p)\widehat z_p)^2},
\end{aligned}
\label{eq:arbitrary-horizon-coefficients}
\end{equation}
the selected schedule has the explicit worst-case error bound as
\begin{equation}
\begin{aligned}
\D_N(\mathcal A_N(\gamma))\le\widehat\tau_NR^2+\widehat V_N\sigma^2 =\sigma^2W_p(\gamma)=\frac{\sigma^2\widehat z_p}{1+(N-p-q_p)\widehat z_p}.
\end{aligned}
\label{eq:arbitrary-horizon-bound}
\end{equation}
\end{theorem}
\begin{proof}
This proof is deferred to Appendix~\ref{app:arbitrary-horizon-bound}.
\end{proof}
The proof applies Theorem~\ref{thm:silver-bound} to the first $k$ steps and Theorem~\ref{thm:serial-bound} to the remaining short-step schedule. The normalized bound starts from $z_k$ in~\eqref{eq:balanced-silver-coefficients}, with $p$ replaced by $k$, evolves into $\widehat z_k$ in~\eqref{eq:explicit-serial-cap-count} after the $q_k$ capped steps, and becomes $W_k(\gamma)$ in~\eqref{eq:prefix-budget-values} after the decreasing steps. Applying the same coefficient updates then yields $\widehat\tau_N$ and $\widehat V_N$ in~\eqref{eq:arbitrary-horizon-coefficients}.


\noindent \textbf{Comparison with short-step schedules.}
Lemma~\ref{lem:serial-short-optimal} shows that the purely \texttt{serial} schedule attains the smallest exact error in Proposition~\ref{prop:exact} among all short-step schedules with stepsizes in $[0,\bar\alpha]$. Any further improvement therefore requires a stepsize larger than $\bar\alpha$. The proposed construction considers \texttt{silver} prefixes, using larger stepsizes to reduce the initial-error faster. By~\eqref{eq:prefix-length-selection} in Definition~\ref{def:arbitrary-horizon-schedule}, the selected prefix minimizes the candidate bounds $W_k$ in~\eqref{eq:prefix-budget-values}, including the purely \texttt{serial} candidate $k=0$. Its error bound is therefore strictly better than the optimal short-step error. The following corollary establishes improvement by large stepsizes. 

\begin{cor}[Improvement by large stepsizes]\label{lem:large-step-improvement}
Under Assumption~\ref{ass:oracle}, let $M>1$, $N\ge2$, and let $\mathcal A_N(\gamma)$ be the schedule in Definition~\ref{def:arbitrary-horizon-schedule} (see~\eqref{eq:arbitrary-horizon-steps}), with $W_k$ defined in~\eqref{eq:prefix-budget-values}. For every $\gamma\ge0$,
\begin{equation*}
\D_N(\mathcal A_N(\gamma)) \le\sigma^2W_p(\gamma)\le\sigma^2W_0(\gamma) \quad \text{and} \quad \sigma^2W_0(\gamma)=\min_{\boldsymbol{\alpha}\in[0,\bar\alpha]^N}\D_N(\boldsymbol{\alpha}) =\sigma^2\mathcal T^{\circ N}(\gamma).
\end{equation*}
The comparison $W_p(\gamma)\le W_0(\gamma)$ is strict if and only if $p>0$. Thus $p=0$ recovers the optimal short-step schedule, whereas $p>0$ gives a schedule containing a stepsize larger than $\bar\alpha$ and having strictly smaller worst-case error than every predetermined short-step schedule.
\end{cor}
\begin{proof}
Theorem~\ref{thm:arbitrary-horizon-convergence} gives $\D_N(\mathcal A_N(\gamma))\le\sigma^2W_p(\gamma)$ for any $\gamma \geq 0$. By the selection rule~\eqref{eq:prefix-length-selection}, $W_p(\gamma)\le W_0(\gamma)$ holds strictly if and only if $p>0$. 
Substituting $z_0=\gamma$,~\eqref{eq:prefix-budget-values} guarantees that $W_0(\gamma)=\mathcal T^{\circ N}(\gamma)$. Lemma~\ref{lem:serial-short-optimal} identifies $\sigma^2\mathcal T^{\circ N}(\gamma)$ as the minimum worst-case error over all predetermined $N$-step short-step schedules. When $p=0$, the construction applies the optimal one-step recursively starting from $\gamma$, giving exactly the schedule in~\eqref{eq:pure-serial-steps}.
If $p>0$, then $\D_N(\mathcal A_N(\gamma)) \le \sigma^2W_p(\gamma) < \sigma^2W_0(\gamma) = \min_{\boldsymbol{\alpha}\in[0,\bar\alpha]^N} \D_N(\boldsymbol{\alpha})$.
Hence, $\mathcal A_N(\gamma)$ has strictly smaller worst-case error than every predetermined short-step schedule.
\end{proof}
Corollary~\ref{lem:large-step-improvement} shows that our proposed schedule has strictly smaller worst-case error than every short-step schedule whenever $p>0$. The next lemma shows that $p\ge2$ for every finite horizon $N\ge2$ when $\gamma$ exceeds the two-step threshold $\gamma_\star(M)$ defined in~\eqref{eq:two-step-crossing}. In this regime, every globally optimal schedule must therefore contain at least one stepsize larger than $\bar\alpha$.

\begin{lemma}[Prefix selection above the two-step threshold]
\label{lem:prefix-above-gamma-star}
Let $M>1$ and $N\ge2$, and let $p$ be the prefix length selected by Definition~\ref{def:arbitrary-horizon-schedule}. If $\gamma>\gamma_\star(M)$, then $p\ge2$.
\end{lemma}
\begin{proof}
By~\eqref{eq:balanced-silver-coefficients}, $z_0=\gamma$ and $z_1=(1-\bar\alpha)^2\gamma+\bar\alpha^2$. The one-step optimum in~\eqref{eq:one-step-optimal-value} gives $\mathcal T(\gamma)\le z_1$. The map $\mathcal T$ is strictly increasing, so~\eqref{eq:prefix-budget-values} yields $W_0(\gamma)=\mathcal T^{\circ(N-1)}(\mathcal T(\gamma))\le\mathcal T^{\circ(N-1)}(z_1)=W_1(\gamma)$.

The block $S_2$ in~\eqref{eq:balanced-silver-blocks} is the two-step \texttt{silver} pair in~\eqref{eq:two-step-optimal-pair}. Its coefficients in~\eqref{eq:balanced-silver-coefficients} are $\tau_2=T(\hat a)$ and $V_2=U(\hat a)$, with $(\hat a,T(\cdot),U(\cdot))$ from~\eqref{eq:two-step-parameters}--\eqref{eq:two-step-TU}. By Theorem~\ref{thm:two-step}, $z_2=\gamma T(\hat a)+U(\hat a)$ is the exact normalized worst-case error of $S_2$ in the stated regime.

Since $\gamma>\gamma_\star(M)>\gamma_2$ and $\mathcal T(\gamma_2)=\gamma_1$ by~\eqref{eq:one-step-optimal-value}--\eqref{eq:two-step-parameters}, the first two optimal short steps are both $\bar\alpha$. Their exact normalized error is $\mathcal T^{\circ2}(\gamma)$ by~\eqref{eq:exact}. Theorem~\ref{thm:two-step} gives $z_2<\mathcal T^{\circ2}(\gamma)$. Applying $\mathcal T^{\circ(N-2)}$ to this strict inequality yields $W_2(\gamma)<W_0(\gamma)\le W_1(\gamma)$. Lastly, the selection rule~\eqref{eq:prefix-length-selection} minimizes $W_k(\gamma)$ over $0\le k\le N$, so $p\ge2$.
\end{proof}

The following figures illustrate the stepsizes and the improvement established in Lemma~\ref{lem:large-step-improvement}, using commonly used short-step schedules as references. In particular, Figure~\ref{fig:arbitrary-horizon-stepsize-comparison} compares the stepsizes with a total of eleven variants of nine reference schedules from Table~\ref{tab:common-short-step-exact}, using the same parameters $M=10$, $N=100$, and $\gamma_0=10^6$. For the constant and inverse families, we use the instances $\alpha_i=\frac{1}{M}$, $\alpha_i=\bar\alpha$, and $\alpha_i=\frac{1}{M+i}$. For the polynomial family, we consider $\alpha_i=\frac{1}{M(i+1)^{\frac{1}{2}}}$ and $\alpha_i=\frac{1}{M(i+1)^{\frac{3}{4}}}$. The halving schedule has five epochs of length $20$ with initial stepsize $\frac{1}{M}$. The remaining schedules use the formulas in Table~\ref{tab:common-short-step-exact}. The proposed schedule exceeds the capped stepsize only inside its \texttt{silver} prefix, while the subsequent \texttt{serial} parts remain short.
\begin{figure}[!htbp]
\centering
\includegraphics[width=0.6\linewidth]{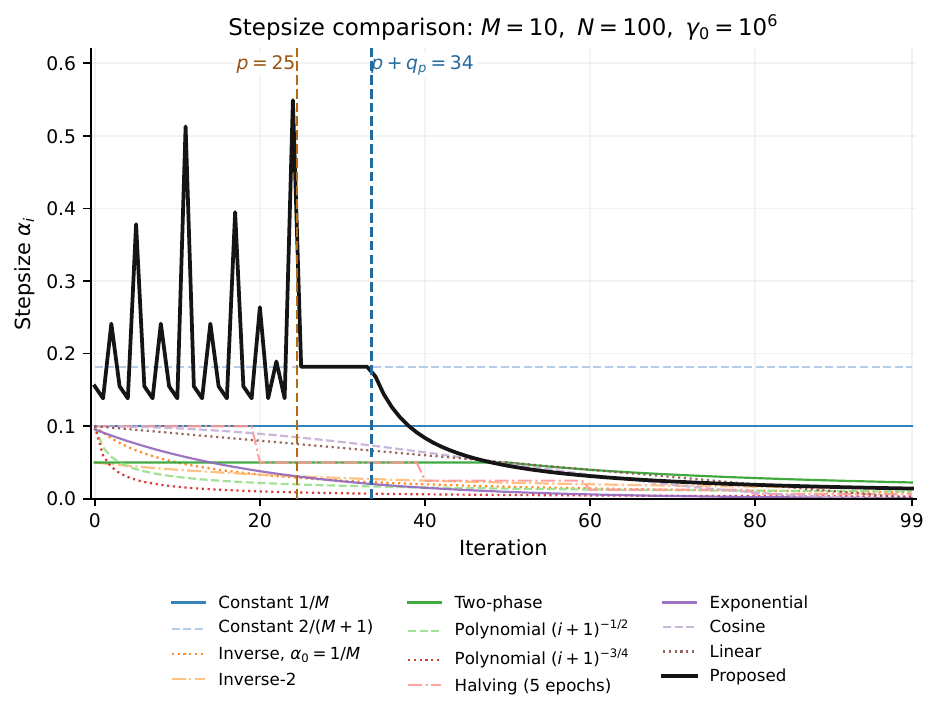}
\caption{Stepsize comparison when $M=10$, $N=100$, and $\gamma_0=10^6$. The horizontal axis is the iteration index $i\in[99]$, and the vertical axis is the stepsize length $\alpha_i$. The vertical lines separate \texttt{silver} steps, capped steps, and decreasing steps with $p=25$ and $q_p=9$.}\label{fig:arbitrary-horizon-stepsize-comparison}
\end{figure}

Figure~\ref{fig:arbitrary-horizon-bound-comparison} compares the normalized upper bound $W_{p}(\gamma)$ in~\eqref{eq:arbitrary-horizon-bound} 
with the normalized exact errors of thirteen reference short-step schedules, where $p$ denotes the prefix length in~\eqref{eq:prefix-length-selection}. In addition to the eleven reference schedules in Figure~\ref{fig:arbitrary-horizon-stepsize-comparison}, we include additional constant and inverse schedules from Table~\ref{tab:common-short-step-exact} whose first stepsizes depend on the initial $\gamma$. Using $\alpha_*(\gamma)$ defined in~\eqref{eq:one-step-optimal}, their stepsizes are, respectively,
\begin{equation}
\alpha_i=\alpha_*(\gamma)\quad\text{(constant)}\qquad \text{and} \qquad
\alpha_i=\frac{\alpha_*(\gamma)}{1+i\alpha_*(\gamma)}\quad\text{(inverse)},\qquad \forall i\in[99]. 
\label{eq:gamma-dependent-reference-schedules}
\end{equation}
Let $\mathcal B(\gamma)$ denote these thirteen reference stepsize selections, 
excluding our proposed schedule, and define the best performance among reference schedules as
\begin{equation}
B_{\mathrm{ref}}(\gamma)=\min_{\boldsymbol{\alpha}\in\mathcal B(\gamma)}\left\{\tau_{100}(\boldsymbol{\alpha})\gamma+V_{100}(\boldsymbol{\alpha})\right\},\label{eq:reference-catalogue-minimum}
\end{equation}
where $\tau_{100}(\boldsymbol{\alpha})$ and $V_{100}(\boldsymbol{\alpha})$ are defined in~\eqref{eq:coefficients}. By Proposition~\ref{prop:exact}, $\frac{\D_{100}(\boldsymbol{\alpha})}{\sigma^2} = \tau_{100}(\boldsymbol{\alpha})\gamma+V_{100}(\boldsymbol{\alpha})$ for each $\boldsymbol{\alpha}\in\mathcal B(\gamma)$. 
For each initial $\gamma$, the minimum operation in~\eqref{eq:reference-catalogue-minimum} selects one complete reference schedule, with all its stepsizes determined prior to stochastic oracle calls. 
As $\gamma$ varies, among all thirteen reference schedules in $\mathcal B(\gamma)$, eleven schedules remain unchanged, while the two schedules in~\eqref{eq:gamma-dependent-reference-schedules} vary with $\gamma$. 
The vertical axis in Figure~\ref{fig:arbitrary-horizon-bound-comparison}(a) shows terminal squared-error bounds divided by $\sigma^2$: we present $W_{p}(\gamma)$ for our proposed schedule, $\frac{\D_{100}(\boldsymbol{\alpha})}{\sigma^2}$ for each reference schedule, and $B_{\mathrm{ref}}(\gamma)$ in~\eqref{eq:reference-catalogue-minimum}. The vertical axis in Figure~\ref{fig:arbitrary-horizon-bound-comparison}(b) is the ratio $\frac{W_{p}(\gamma)}{B_{\mathrm{ref}}(\gamma)}$. Such ratio smaller than one certifies our proposed schedule has an improved worst-case error compared to all reference schedules. 

\begin{figure}[ht]
\centering
\includegraphics[width=0.8\linewidth]{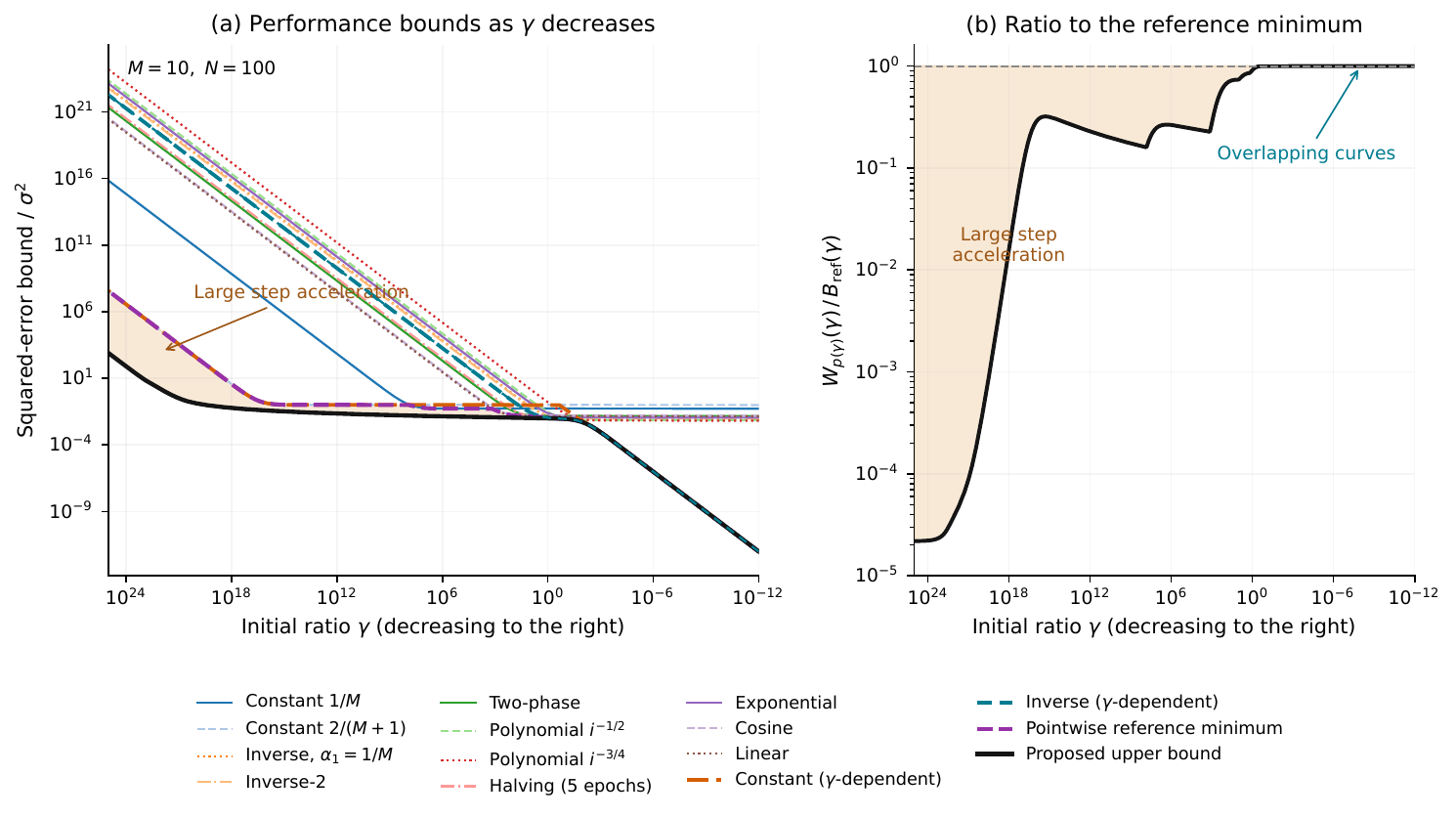}
\caption{Performance comparison over different stepsize schedules when $M=10$ and $N=100$. Subfigure~(a)'s vertical axis represents the terminal squared-error bounds divided by $\sigma^2$. Our proposed curve is $W_{p}(\gamma)$ from~\eqref{eq:arbitrary-horizon-bound}; each reference curve shows $\frac{\D_{100}(\boldsymbol{\alpha})}{\sigma^2}$ in~\eqref{eq:coefficients}--\eqref{eq:exact}, and the reference minimum is $B_{\mathrm{ref}}$ in~\eqref{eq:reference-catalogue-minimum}, over eleven reference schedules in Figure~\ref{fig:arbitrary-horizon-stepsize-comparison} and two additional schedules in~\eqref{eq:gamma-dependent-reference-schedules}. Subfigure~(b)'s vertical axis demonstrates the performance ratio $\frac{W_{p}(\gamma)}{B_{\mathrm{ref}}(\gamma)}$.}\label{fig:arbitrary-horizon-bound-comparison}
\end{figure}
From Figure~\ref{fig:arbitrary-horizon-bound-comparison}, we observe our proposed schedule outperforms all reference short-step schedules when the initial ratio $\gamma$ is reasonably large. In fact, when $\gamma$ is close to zero, the optimal short-step schedule in Lemma~\ref{lem:serial-short-optimal} is also globally optimal: allowing arbitrary nonnegative stepsizes cannot improve its worst-case error. Recently,~\cite{grimmer2026exact} establishes a decreasing short-step schedule globally minimizes the worst-case expected final-iterate suboptimality for smooth convex problems when the variance bound and problem dimension are both sufficiently large. The following theorem gives a related result for strongly convex setting.  It provides a sufficient interval on which our proposed schedule attains such optimal value and describes the lower-bound construction.

\begin{theorem}[Global optimality of the predetermined schedule]
\label{thm:arbitrary-horizon-global-optimality}
Under Assumption~\ref{ass:oracle}, let $M>1$ and $N\in\mathbb{N}/\{0\}$. Define
\begin{equation}
\Gamma(M)=\left(\frac{M+1}{M-1}\right)^3-1. \label{eq:global-optimality-sufficient-threshold}
\end{equation}
For any $\gamma\in[0,\Gamma(M)]$, the schedule $\mathcal A_N(\gamma)$ in Definition~\ref{def:arbitrary-horizon-schedule} is globally optimal among all nonnegative stepsize schedules. In particular, $\D_N(\mathcal A_N(\gamma)) =\inf_{\boldsymbol\alpha\in[0,\infty)^N}\D_N(\boldsymbol\alpha) =\sigma^2\mathcal T^{\circ N}(\gamma)$.
\end{theorem}
\begin{proof}
This proof is deferred to Appendix~\ref{app:arbitrary-horizon-global}.
\end{proof}

\begin{rem}
Corollary~\ref{lem:large-step-improvement} shows that whenever $p>0$, our proposed schedule has at least one stepsize larger than $\bar\alpha$ and achieves strictly smaller worst-case expected error than the optimal short-step schedule. Lemma~\ref{lem:prefix-above-gamma-star} guarantees $p\ge2$ whenever $\gamma>\gamma_\star(M)$ and $N\ge2$. Thus, for sufficiently large $\gamma$, large stepsizes accelerate the SG method in the sense of achieving a smaller worst-case expected error at the same finite horizon. For $\gamma\in[0,\Gamma(M)]$, Theorem~\ref{thm:arbitrary-horizon-global-optimality} shows that our proposed schedule is globally optimal among all nonnegative stepsize schedules.
\end{rem}

\section{Conclusion}
In this paper, we characterize the optimal one- and two-step schedules, which motivate two complementary strategies for composing stepsize schedules for SG methods. Repeated \texttt{serial} joins of optimally chosen one-step schedules yield the optimal schedule within the short-step class, as established by our tight worst-case analysis. Consequently, improving upon this bound requires larger stepsizes. \texttt{Silver} joins offer such an improvement by jointly modifying the terminal steps of two \texttt{silver} blocks, thereby introducing larger stepsizes that accelerate the reduction of the initial error.
Combining these two composition strategies leads to an explicit predetermined schedule for any prescribed horizon. Specifically, a \texttt{silver} prefix is followed by constant and decreasing short-step phases connected through \texttt{serial} joins. Both the stepsizes and the phase lengths are determined entirely by the initial problem parameters, without requiring any oracle calls. By appropriately selecting the prefix length, we obtain a worst-case error bound that improves upon the exact optimum within the short-step class, with a strict improvement whenever the \texttt{silver} prefix is nonempty.
Moreover, our global optimality result shows that, when the initial error bound is sufficiently small relative to the noise variance, the purely \texttt{serial} schedule minimizes the worst-case error over all predetermined nonnegative stepsize schedules. These results highlight the complementary roles of the two phases: coordinated large steps rapidly reduce the initial error, while short steps control variance accumulation and achieve optimal worst-case performance in the regime.

\section*{Acknowledgments}

We solved performance estimation problems (PEPs) numerically using the PEP framework of~\cite{rubbens2025computer}. The numerical solutions guided our search for stepsize schedules, worst-case instances, and dual certificates. Drawing on the AI-assisted workflow described in~\cite{kim2026domain}, we used ChatGPT 6.1 Sol to identify patterns in these solutions, propose analytical formulas, and verify algebraic derivations. The authors verified the final mathematical claims and proofs and take responsibility for the manuscript.

\bibliographystyle{abbrv}
\bibliography{reference}

\appendix
\section{Proofs of the Optimal One- and Two-Step Results}\label{app:one-two-step-proofs}
We prove the optimality of the one- and two-step schedules~\eqref{eq:one-step-optimal} and~\eqref{eq:two-step-optimal-pair} by matching the upper bound at each proposed schedule with the global minimum of the corresponding lower bound.

\subsection{A Family of Lower-Bound Instances}\label{app:two-step-lower}
\begin{lemma}\label{lem:signal-noise}
Let $\varphi:\mathbb R\to\mathbb R$ be differentiable, $1$-strongly convex, and $M$-smooth, with $\nabla\varphi(0)=0$, and let $\lambda\in[1,M]$. For predetermined nonnegative stepsizes $(\alpha_0,\alpha_1)$, let's define
\begin{equation}
u_1=R-\alpha_0\nabla\varphi(R)\quad \text{and} \quad u_2=u_1-\alpha_1\nabla\varphi(u_1).
\label{eq:signal-trajectory}
\end{equation}
Choose $c>\max\{R,|u_1|\}$ and define
\begin{equation}
\widetilde\varphi(t)=
\begin{cases}
\varphi(-c)+\nabla \varphi(-c)(t+c)+\frac{\lambda}{2}(t+c)^2&\text{if }t<-c\\
\varphi(t)&\text{if }-c\le t\le c\\
\varphi(c)+\nabla\varphi(c)(t-c)+\frac{\lambda}{2}(t-c)^2&\text{if }t>c.
\end{cases}
\label{eq:two-step-scalar-objective}
\end{equation}
For each integer $n\ge2$, we consider the function, initial point, stochastic oracle, and SG updates such as
\begin{equation*}
\begin{aligned}
f(x)&=\widetilde\varphi(x),\quad  x_\star=0,\quad  x_0=R, \quad g_0^{(n)} = \nabla\widetilde\varphi(x_0)+\xi_0^{(n)} \\
x_1^{(n)} &= x_0 - \alpha_0g_0^{(n)}, \quad g_1^{(n)}=\nabla\widetilde\varphi(x_1^{(n)})+\xi_1^{(n)},\quad \text{and} \quad x_2^{(n)}=x_1^{(n)}-\alpha_1g_1^{(n)},
\end{aligned}
\end{equation*}
where $\xi_0^{(n)}$ and $\xi_1^{(n)}$ are independent random variables satisfying
\begin{equation}
\xi_k^{(n)}=
\begin{cases}
-\sigma\sqrt n,&\text{with probability }\frac{1}{2n},\\
\sigma\sqrt n,&\text{with probability }\frac{1}{2n},\\
0,&\text{with probability }1-\frac{1}{n},
\end{cases}
\qquad \text{for }k\in\{0,1\}.
\label{eq:two-step-scalar-noise}
\end{equation}
Then $f\in\mathcal F_{1,M}$, and the stochastic oracle satisfies Assumption~\ref{ass:oracle}. The iterates satisfy
\begin{equation}
x_1^{(n)}=u_1-\alpha_0\xi_0^{(n)} \quad \text{and} \quad x_2^{(n)}=x_1^{(n)}-\alpha_1\nabla\widetilde\varphi(x_1^{(n)})-\alpha_1\xi_1^{(n)}.
\label{eq:noise-coordinate}
\end{equation}
Consequently,
\begin{equation}
\begin{aligned}
\E[(x_1^{(n)})^2]=u_1^2+\alpha_0^2\sigma^2 \quad \text{and} \quad
\lim_{n\to+\infty}\E[(x_2^{(n)})^2]=u_2^2+(\alpha_0^2(1-\lambda\alpha_1)^2+\alpha_1^2)\sigma^2.
\end{aligned}
\label{eq:signal-noise-errors}
\end{equation}
For $d>1$, we consider $f(x)=\widetilde\varphi([x]_1)+\frac{1}{2}\sum_{j\in[d]\backslash[1]}([x]_j)^2$, $x_0=R\mathbf e_1$, and stochastic gradient noise $\xi_k\mathbf e_1$, where $\mathbf e_1$ is the first column of the $d$-dimensional identity matrix, so~\eqref{eq:signal-noise-errors} holds for the squared distance to the minimizer.
\end{lemma}
\begin{proof}
By~\eqref{eq:two-step-scalar-objective}, $f\in\mathcal F_{1,M}$ and has the unique minimizer zero. By~\eqref{eq:two-step-scalar-noise}, $\xi_0^{(n)}$ and $\xi_1^{(n)}$ are independent random variables with $\E[\xi_k^{(n)}]=0$ and $\E[(\xi_k^{(n)})^2]=\sigma^2$, so the stochastic oracle satisfies Assumption~\ref{ass:oracle}. Substituting $g_0^{(n)}$ and $g_1^{(n)}$ into~\eqref{eq:sg-update} gives~\eqref{eq:noise-coordinate} and $\E[(x_1^{(n)})^2]=u_1^2+\alpha_0^2\sigma^2$. Furthermore,~\eqref{eq:two-step-scalar-noise} results in
\begin{equation}
\begin{aligned}
\E[(x_2^{(n)})^2]&=\left(1-\tfrac{1}{n}\right)u_2^2+\alpha_1^2\sigma^2+\tfrac{1}{2n}\sum_{\varepsilon\in\{-1,1\}}\left(u_1-\varepsilon\alpha_0\sigma\sqrt n-\alpha_1\nabla\widetilde\varphi(u_1-\varepsilon\alpha_0\sigma\sqrt n)\right)^2.
\end{aligned}
\label{eq:two-step-scalar-second-moment}
\end{equation}
By~\eqref{eq:two-step-scalar-objective}, $\sup_{t\in\mathbb R}|\nabla\widetilde\varphi(t)-\lambda t|<\infty$, so
\begin{equation}
\lim_{n\to\infty}\tfrac{u_1-\varepsilon\alpha_0\sigma\sqrt n-\alpha_1\nabla\widetilde\varphi(u_1-\varepsilon\alpha_0\sigma\sqrt n)}{\sqrt n}=-\varepsilon\alpha_0\sigma(1-\lambda\alpha_1),\qquad \forall\varepsilon\in\{-1,1\}. \label{eq:two-step-scalar-tail-limit}
\end{equation}
By~\eqref{eq:two-step-scalar-second-moment} and~\eqref{eq:two-step-scalar-tail-limit},~\eqref{eq:signal-noise-errors} has been proved. For general natural number $d > 1$, by~\eqref{eq:sg-update}, all coordinates except the first remain zero, so~\eqref{eq:signal-noise-errors} still holds, which concludes the statement.
\end{proof}

We apply Lemma~\ref{lem:signal-noise} to the following four functions. The first three functions are
\begin{equation}
\label{eq:two-step-first-three-functions}
\varphi_1(u)=\tfrac{u^2}{2},\quad \varphi_2(u)=\tfrac{M u^2}{2},\quad \text{and} \quad \varphi_3(u)=
\begin{cases}
u^2/2&\text{if }u\leq0\\
M u^2/2&\text{if }u>0. 
\end{cases}
\end{equation}
The fourth function joins two quadratic pieces at a breakpoint chosen to coincide with the first iterate. For $0\leq\alpha_0\leq1$, define
\begin{equation}
\label{eq:two-step-fourth-function}
\begin{aligned}
b_{\alpha_0}:=\tfrac{R(1-\alpha_0)}{1+\alpha_0(M-1)} \ \ \text{and} \ \ \varphi_{4,\alpha_0}(u):=
\begin{cases}
\tfrac{M u^2}{2}&\text{if }u\leq b_{\alpha_0}\\
\frac{u^2}{2}+(M-1)b_{\alpha_0}u-\frac{(M-1)b_{\alpha_0}^2}{2}&\text{if }u>b_{\alpha_0}.
\end{cases}
\end{aligned}
\end{equation}
For $(\alpha_0,\alpha_1)\in\R_{\geq0}\times\R_{\geq0}$, we define
\begin{equation}
\begin{aligned}
r_1&=|(1-\alpha_0)(1-\alpha_1)|,\qquad r_2=|(M\alpha_0-1)(M\alpha_1-1)|,\\ r_3&=|(M\alpha_0-1)(1-\alpha_1)|,\qquad \text{and} \qquad r_4=\tfrac{|(1-\alpha_0)(M\alpha_1-1)|}{1+\alpha_0(M-1)},
\end{aligned}
\label{eq:two-step-four-factors}
\end{equation}
and set
\begin{equation}
\begin{aligned}
\hat\rho(\alpha_0,\alpha_1)&=
\begin{cases}
\max\{r_1,r_2,r_3,r_4\}&\text{if }0\le\alpha_0\le1\\ \max\{r_1,r_2,r_3\}&\text{if }\alpha_0>1
\end{cases}
\\ \text{and} \quad L_\gamma(\alpha_0,\alpha_1)&=\gamma\hat\rho(\alpha_0,\alpha_1)^2+\alpha_0^2\max\{(1-\alpha_1)^2,(1-M\alpha_1)^2\}+\alpha_1^2.
\end{aligned}
\label{eq:two-step-lower-envelope}
\end{equation}
Based on Lemma~\ref{lem:signal-noise} and \eqref{eq:two-step-first-three-functions}--\eqref{eq:two-step-lower-envelope}, the following lemma establishes a lower bound valid for every nonnegative two-step schedule.

\begin{lemma}[Lower bounding function]\label{lem:two-step-lower}
Under Assumption~\ref{ass:oracle}, consider parameters $(M,R,\sigma)\in\R_{>1}\times\R_{\geq0}\times\R_{>0}$. For every stepsize pair $(\alpha_0,\alpha_1)\in\R_{\geq0}\times\R_{\geq0}$, it holds that $\D_2(\boldsymbol{\alpha})\ge\sigma^2L_\gamma(\alpha_0,\alpha_1)$, where $L_\gamma(\alpha_0,\alpha_1)$ is defined in~\eqref{eq:two-step-lower-envelope}.
\end{lemma}
\begin{proof}
For any stepsize pair $(\alpha_0,\alpha_1)\in\R_{\geq0}\times\R_{\geq0}$, choose $\varphi\in\{\varphi_1,\varphi_2,\varphi_3,\varphi_{4,\alpha_0}\}$ as defined in~\eqref{eq:two-step-first-three-functions}--\eqref{eq:two-step-fourth-function}, and $\lambda\in\{1,M\}$. The function $\varphi_{4,\alpha_0}$ is omitted when $\alpha_0>1$.
For each choice of $\varphi$ and $\lambda$, let $u_k^{\varphi,\lambda}$, with $k\in\{1,2\}$, defined in~\eqref{eq:signal-trajectory}. By Lemma~\ref{lem:signal-noise} and~\eqref{eq:worst-case-error}, $\D_2(\boldsymbol\alpha)\ge\E[(x_2^{(n)})^2]$ for every integer $n\ge2$ and every fixed dimension $d\ge1$. Taking the limit by~\eqref{eq:signal-noise-errors} gives
\begin{equation}
\D_2(\boldsymbol\alpha)\ge\lim_{n\to\infty}\E[(x_2^{(n)})^2]=(u_2^{\varphi,\lambda})^2+\sigma^2\alpha_0^2(1-\lambda\alpha_1)^2+\sigma^2\alpha_1^2.
\label{eq:two-step-lower-lemma-application}
\end{equation}
To compute $u_2^{\varphi,\lambda}$, substitute each of these functions into~\eqref{eq:signal-trajectory}, obtaining
\begin{equation}
u_2^{\varphi,\lambda}=
\begin{cases}
(1-\alpha_0)(1-\alpha_1)R,&\text{if }\varphi=\varphi_1,\\
(1-M\alpha_0)(1-M\alpha_1)R,&\text{if }\varphi=\varphi_2,\\
(1-M\alpha_0)(1-\alpha_1)R,&\text{if }\varphi=\varphi_3\text{ and }\alpha_0\ge1/M,\\
(1-M\alpha_0)(1-M\alpha_1)R,&\text{if }\varphi=\varphi_3\text{ and }0\le\alpha_0<1/M,\\
\tfrac{(1-\alpha_0)(1-M\alpha_1)}{1+\alpha_0(M-1)}R,&\text{if }\varphi=\varphi_{4,\alpha_0}\text{ and }0\le\alpha_0\le1.
\end{cases}
\label{eq:two-step-lower-signal-trajectories}
\end{equation}
By~\eqref{eq:two-step-lower-lemma-application} and~\eqref{eq:two-step-lower-signal-trajectories}, we have
\begin{equation*}
\begin{aligned}
\D_2(\boldsymbol\alpha)
&\ge\sup_{\substack{\varphi\in\{\varphi_1,\varphi_2,\varphi_3,\varphi_{4,\alpha_0}\}\\ \lambda\in\{1,M\}}}\lim_{n\to\infty}\E[(x_2^{(n)})^2]=\sup_{\lambda\in\{1,M\}}\left(R^2\hat\rho(\alpha_0,\alpha_1)^2+\sigma^2\alpha_0^2(1-\lambda\alpha_1)^2+\sigma^2\alpha_1^2\right)\\
&=R^2\hat\rho(\alpha_0,\alpha_1)^2+\sigma^2\alpha_0^2\max\{(1-\alpha_1)^2,(1-M\alpha_1)^2\}+\sigma^2\alpha_1^2.
\end{aligned}
\end{equation*}
where the first equality follows from~\eqref{eq:two-step-lower-envelope}--\eqref{eq:two-step-lower-signal-trajectories} and the inequality $r_3\le r_1$ when $\alpha_0<\frac{1}{M}$, and the last equality follows by taking the supremum over $\lambda\in\{1,M\}$.
Finally, the definitions of $\gamma$ in~\eqref{eq:one-two-notation} and $L_\gamma(\alpha_0,\alpha_1)$ in~\eqref{eq:two-step-lower-envelope} yield the claimed bound in every fixed dimension $d\ge1$.
\end{proof}

\subsection{Proof of Theorem~\ref{thm:one-step}}\label{app:one-step-optimal-proof}

\begin{proof}[Proof of Theorem~\ref{thm:one-step}]
Applying Proposition~\ref{prop:exact}, with $N=1$, gives
\begin{equation}
\D_1(\alpha)=R^2(1-\alpha)^2+\sigma^2\alpha^2,\qquad \forall \alpha\in[0,\bar\alpha].\label{eq:one-step-exact}
\end{equation}
To exclude all larger steps, let $\varphi(u)=\frac{Mu^2}{2}$ and $\alpha_0=\alpha$ in Lemma~\ref{lem:signal-noise}. For every $\alpha\ge0$, this gives $\D_1(\alpha)\ge R^2(1-M\alpha)^2+\sigma^2\alpha^2$. At $\alpha=\bar\alpha$, this lower bound of $\D_1(\alpha)$ agrees with \eqref{eq:one-step-exact}, and, for any $\alpha\ge\bar\alpha$, its derivative satisfies $2(M^2R^2+\sigma^2)\alpha-2MR^2\ge\frac{2M(M-1)R^2+4\sigma^2}{M+1}>0$. Consequently, the optimal one-step stepsize must lie in the short-step interval $[0,\bar\alpha]$. Moreover, by~\eqref{eq:one-step-exact}, we can rewrite $\frac{\D_1(\alpha)}{\sigma^2}=(1+\gamma)(\alpha-\frac{\gamma}{1+\gamma})^2+\frac{\gamma}{1+\gamma}$. It follows that the optimal one-step stepsize is given by~\eqref{eq:one-step-optimal}, with the corresponding optimal value given by~\eqref{eq:one-step-optimal-value}.
\end{proof}

\subsection{Upper bounding function}\label{app:two-step-upper}
For $f\in\mathcal F_{1,M}$, the interpolation inequality in~\cite[Theorem~4]{taylor2017smooth} gives $Q(u,v)\ge0$ for every $\{u,v\}\subset\mathbb R^d$, where
\begin{equation}
Q(u,v)=f(u)-f(v)+\tfrac{\langle M\nabla f(v)-\nabla f(u),v-u\rangle}{M-1}-\tfrac{\|\nabla f(u)-\nabla f(v)\|^2}{2(M-1)}-\tfrac{M\|u-v\|^2}{2(M-1)}.\label{eq:two-step-interpolation}
\end{equation}
Let $x_u$ and $x_v$ be SG iterates viewed as random variables on the same probability space. Since $Q(u,v)\ge0$ for every pair of points $\{u,v\}\subset\mathbb R^d$, we have $\E [Q(x_u,x_v)]\ge0$. In the following lemma, we establish upper bounds for the function $\D_2(\cdot)$.
\begin{lemma}[Upper bounding function]\label{lem:two-step-upper}
Under Assumption~\ref{ass:oracle}, the following results hold.
\begin{enumerate}
\item[$(i)$] If $\{\alpha_0,\alpha_1\}\subset[0,\bar\alpha]$, then $\D_2(\boldsymbol{\alpha})\le R^2(1-\alpha_0)^2(1-\alpha_1)^2+\sigma^2\alpha_0^2(1-\alpha_1)^2+\sigma^2\alpha_1^2$.
\item[$(ii)$] For $1\le\zeta\le\hat a$, define
\begin{equation}
\alpha_0(\zeta)=\tfrac{\zeta+1}{M\zeta+1}\quad \text{and} \quad \alpha_1(\zeta)=\tfrac{\zeta+1}{\zeta+M}.\label{eq:two-step-boundary-pair}
\end{equation}
With functions $T(\cdot)$ and $U(\cdot)$ defined in~\eqref{eq:two-step-TU}, the stepsize pair $\boldsymbol{\alpha}=(\alpha_0(\zeta),\alpha_1(\zeta))$ satisfies $\D_2(\boldsymbol{\alpha})\le R^2T(\zeta)+\sigma^2U(\zeta)$.
\end{enumerate}
\end{lemma}
\begin{proof}
Part~\textup{(i)} directly follows from Proposition~\ref{prop:exact} with $N=2$.
For Part~\textup{(ii)}, we fix $1\le\zeta\le\hat a$ and use the stepsizes defined in~\eqref{eq:two-step-boundary-pair}. Direct substitution gives
\begin{equation}
T(\zeta)=(1-\alpha_0(\zeta))^2(1-\alpha_1(\zeta))^2 \quad \text{and} \quad U(\zeta)=\alpha_0(\zeta)^2(M\alpha_1(\zeta)-1)^2+\alpha_1(\zeta)^2.
\label{eq:two-step-coefficient-identities}
\end{equation}
At the two endpoints $\zeta = 1$ and $\zeta = \hat{a}$, we have
\begin{equation}
(\alpha_0(1),\alpha_1(1))=(\bar\alpha,\bar\alpha) \quad \text{and} \quad (\alpha_0(\hat a),\alpha_1(\hat a))=\left(\tfrac2{1+\eta},\tfrac2{1+2M-\eta}\right).
\label{eq:two-step-endpoints}
\end{equation}
Without loss of generality, we assume $x_\star=0$ and $f(0)=0$ for simplicity. The SG iterates remain $x_1=x_0-\alpha_0(\zeta)g_0$ and $x_2=x_1-\alpha_1(\zeta)g_1$. Writing the first stochastic oracle error and the exact-gradient point after one step as $\epsilon_0=g_0-\nabla f(x_0)$ and $\bar x_1=x_0-\alpha_0(\zeta)\nabla f(x_0)$, respectively, we then have $x_1-\bar x_1=-\alpha_0(\zeta)\epsilon_0$.
Here, $\E[\epsilon_0]=0$ and $\bar x_1$ does not rely on the stochastic oracle. Let's also define the scalar coefficients that would be helpful for our analysis:
\begin{equation}
\begin{aligned}
&\Delta=\alpha_1(\zeta)((M+1)\alpha_1(\zeta)-2),\quad \lambda_{0\star}=\tfrac{(M-1)^2\zeta(\zeta+1)(M+2(M-1)\zeta-M\zeta^2)}{(M\zeta+1)^2(\zeta+M)^2},\\
&\lambda_{01}=(1-\alpha_0(\zeta))\Delta,\quad \lambda_{10}=(M\alpha_0(\zeta)-1)\Delta,\quad \lambda_{\star0}=\lambda_{0\star}+(2-(M+1)\alpha_0(\zeta))\Delta,\\
&\lambda_{\star1}=\alpha_0(\zeta)(M-1)\alpha_1(\zeta)^2+2(1-\alpha_0(\zeta))\alpha_1(\zeta)(1-\alpha_1(\zeta)),\\ 
\text{and} \quad &\lambda_{1\star}=(M-1)\alpha_1(\zeta)^2-(M\alpha_0(\zeta)-1)\Delta.
\end{aligned}
\label{eq:two-step-certificate-coefficients}
\end{equation}
We next verify that all coefficients in~\eqref{eq:two-step-certificate-coefficients} are nonnegative, so that their weighted interpolation inequalities remain nonnegative. \eqref{eq:two-step-boundary-pair} gives $\frac1M<\alpha_0(\zeta)\le\bar\alpha\le\alpha_1(\zeta)<1$, $0<M\alpha_0(\zeta)-1<1$, and $0\le\Delta<(M-1)\alpha_1(\zeta)^2$.
The bound on $\Delta$ follows from $\alpha_1(\zeta)\ge\bar\alpha$ and $\Delta-(M-1)\alpha_1(\zeta)^2=2\alpha_1(\zeta)(\alpha_1(\zeta)-1)<0$. Because $1-\alpha_0(\zeta)>0$, $M\alpha_0(\zeta)-1>0$, and $\Delta\ge0$, we then have $\{\lambda_{01},\lambda_{10}\}\subset\mathbb{R}_{\geq 0}$. Since $\{\alpha_0(\zeta),\alpha_1(\zeta)\}\subset(0,1)$ and $M>1$, both terms defining $\lambda_{\star1}$ are positive, so $\lambda_{\star1}>0$. Moreover, $0<M\alpha_0(\zeta)-1<1$ ensures $\lambda_{1\star}=(M-1)\alpha_1(\zeta)^2-(M\alpha_0(\zeta)-1)\Delta\ge(M-1)\alpha_1(\zeta)^2-\Delta>0$.
Finally, $M\zeta^2-2(M-1)\zeta-M\le0$ for $1\le\zeta\le\hat a$, because $\hat a$ is its positive root and its other root is negative. Substituting this inequality into the formula for $\lambda_{0\star}$ results in $\lambda_{0\star}\ge0$. Since $\alpha_0(\zeta)\le\bar\alpha$, we also have $2-(M+1)\alpha_0(\zeta)\ge0$, and hence $\lambda_{\star0}=\lambda_{0\star}+(2-(M+1)\alpha_0(\zeta))\Delta\ge0$. At $\zeta=\hat a$, we have $\lambda_{0\star}=0$.

Let $\epsilon_1=g_1-\nabla f(x_1)$. Using~\eqref{eq:unbiasedness} and substituting~\eqref{eq:two-step-certificate-coefficients} into~\eqref{eq:two-step-interpolation}, we obtain
\begin{equation*}
\begin{aligned}
&\E[\|x_2-x_\star\|^2]=\E[\|x_1-x_\star-\alpha_1(\zeta)\nabla f(x_1)\|^2]+\alpha_1(\zeta)^2\E[\|\epsilon_1\|^2]\\
= \ &T(\zeta)\|x_0-x_\star\|^2+\alpha_0(\zeta)^2(M\alpha_1(\zeta)-1)^2\E[\|\epsilon_0\|^2]+\alpha_1(\zeta)^2\E[\|\epsilon_1\|^2]-\lambda_{0\star}Q(x_0,x_\star)\\
&\quad-\lambda_{\star0}Q(x_\star,x_0)-\lambda_{01}\E [Q(x_0,x_1)]-\lambda_{10}\E[Q(x_1,x_0)] -\lambda_{\star1}\E [Q(x_\star,x_1)]-\lambda_{1\star}\E [Q(x_1,x_\star)]\\
&\quad-\Delta\E[Q(x_1,\bar x_1)+Q(\bar x_1,x_1)]-\tfrac{\Delta}{M-1}\E[\|\nabla f(x_1)-\nabla f(\bar x_1)+M\alpha_0(\zeta)\epsilon_0\|^2]\\
\le \ &T(\zeta)\|x_0-x_\star\|^2+\alpha_0(\zeta)^2(M\alpha_1(\zeta)-1)^2\E[\|\epsilon_0\|^2]+\alpha_1(\zeta)^2\E[\|\epsilon_1\|^2]\le T(\zeta)R^2+U(\zeta)\sigma^2.
\end{aligned}
\end{equation*}
All expectations are finite by the conditional variance bounds and $M$-smoothness. Since the preceding bound holds for every instance allowed in~\eqref{eq:worst-case-error}, the definition of $\D_2$ gives $\D_2(\boldsymbol{\alpha})\le T(\zeta)R^2+U(\zeta)\sigma^2$.
\end{proof}

\subsection{The Endpoint Comparison}\label{app:two-step-endpoints}

We next compare the two coefficient functions in Part~\textup{(ii)} of Lemma~\ref{lem:two-step-upper}. The comparison will be used to minimize the lower bound in Lemma~\ref{lem:two-step-lower} and identify the minimum lower bound is matched by the short-step upper bound in Part~\textup{(i)} or Part~\textup{(ii)} of Lemma~\ref{lem:two-step-upper}.

\begin{lemma}[The endpoint comparison]\label{lem:two-step-curve}
Let $\gamma_1$ be defined in~\eqref{eq:one-two-notation}, $\hat a$ and $\gamma_2$ in~\eqref{eq:two-step-parameters}, $\gamma_\star$ in~\eqref{eq:two-step-crossing}, and functions $T(\cdot)$ and $U(\cdot)$ in~\eqref{eq:two-step-TU}. Then $T(\cdot)$ is strictly decreasing on $(1,\hat a]$, and we have
\begin{equation}
0<\gamma_1<\gamma_2<\gamma_\star\qquad \text{and} \qquad \gamma_\star(M-1)>\hat a(\hat a+1).\label{eq:two-step-threshold-order}
\end{equation}
Moreover, it holds that
\begin{equation}
\gamma_\star T(\zeta)+U(\zeta)>\gamma_\star T(1)+U(1)=\gamma_\star T(\hat a)+U(\hat a),\qquad \forall \zeta\in(1,\hat{a}).\label{eq:two-step-strict-chord}
\end{equation}
\end{lemma}
\begin{proof}
Differentiating $T(\cdot)$ in~\eqref{eq:two-step-TU} gives $\nabla T(\zeta)=-\frac{2M(M-1)^4\zeta(\zeta-1)(\zeta+1)}{(M\zeta+1)^3(\zeta+M)^3}<0$ for $\zeta>1$.
Since $\hat a>1$ by~\eqref{eq:two-step-parameters}, $T(\cdot)$ is strictly decreasing on $(1,\hat a]$. This proves the first assertion and gives $T(1)-T(\hat a)>0$, so the denominator defining $\gamma_\star$ is positive.

The definitions in~\eqref{eq:one-two-notation} and~\eqref{eq:two-step-parameters} ensure $\gamma_1>0$ and $\gamma_2-\gamma_1=\frac{4(M+1)}{(M-1)^3}>0$.
Thus, $0<\gamma_1<\gamma_2$. To prove $\gamma_2<\gamma_\star$, let's define $G(\zeta)=\frac{U(\zeta)-U(1)}{T(1)-T(\zeta)}$ for $\zeta\in(1,3]$, whose denominator is positive by the monotonically decrease of $T(\cdot)$ over $(1,+\infty)$. $G(\zeta)$ represents the value of $\gamma$ satsifying $\gamma T(\zeta)+U(\zeta)=\gamma T(1)+U(1)$; in particular, $G(\hat a)=\gamma_\star$ by~\eqref{eq:two-step-crossing}. We could use $G(3)$ as a lower bound of $G(\hat a)$, because~\eqref{eq:two-step-parameters} implying $\hat a=\frac{M-1}{M}+\sqrt{1+\left(\frac{M-1}{M}\right)^2}$ and $1<\hat a<1+\sqrt2<\frac52<3$. We further have that for any $\zeta\in(1,3]$, $\nabla G(\zeta)=-\frac{2(M+1)^4(\zeta+1)((\zeta-1)^3+2M(1+M^2)\zeta^2(3-\zeta)+2M^2\zeta(3+\zeta^2))}{M(M-1)^3(\zeta-1)^2[M(\zeta^2+1)+2(M^2+M+1)\zeta]^2}<0$.
Hence, $G(\cdot)$ is strictly decreasing on $(1,3]$, resulting in $\gamma_\star=G(\hat a)>G(3)$. Direct substitution also yields $G(3)-\gamma_2=\frac{(M+1)^3(21M^2+8M-11)}{M(M-1)^3(3+8M+3M^2)}>0$.
Therefore, $\gamma_\star>G(3)>\gamma_2$, which implies $0<\gamma_1<\gamma_2<\gamma_\star$.

We next prove $\gamma_\star(M-1)>\hat a(\hat a+1)$. By $\gamma_\star>G(3)$, direct substitution gives $G(3)-\frac9{M-1}=\frac{69M^4+190M^3+42M^2-34M-11}{M(M-1)^3(3+8M+3M^2)}>0$,
implying that $\gamma_\star(M-1)>9$. The bound $\hat a<\frac{5}{2}$ established above also guarantees $\hat a(\hat a+1)<\frac{35}{4}<9$. Consequently, $\gamma_\star(M-1)>9>\hat a(\hat a+1)$, which completes the proof of~\eqref{eq:two-step-threshold-order}.

For $1<\zeta<\hat a$, the strict decrease of $T(\cdot)$ guarantees $T(1)-T(\zeta)>0$, and the strict decrease of $G(\cdot)$ ensures $G(\zeta)>G(\hat a)=\gamma_\star$. Therefore, $\gamma_\star T(\zeta)+U(\zeta)-(\gamma_\star T(1)+U(1))=(T(1)-T(\zeta))(G(\zeta)-G(\hat a)) >0$. Finally, the definition of $\gamma_\star$ in~\eqref{eq:two-step-crossing} implies $\gamma_\star T(1)+U(1)=\gamma_\star T(\hat a)+U(\hat a)$.
Together, these two statements prove~\eqref{eq:two-step-strict-chord}.
\end{proof}

\subsection{Proof of Theorem~\ref{thm:two-step}}\label{app:two-step-optimal-proof}

\begin{proof}[Proof of Theorem~\ref{thm:two-step}]
For $\boldsymbol{\alpha}=(\alpha_0,\alpha_1)\in\mathbb R_{\geq0}^2$, let $\D_2(\boldsymbol{\alpha})$ be the worst-case error defined in~\eqref{eq:worst-case-error}, with parameters and admissible instances specified therein. 
Lemma~\ref{lem:two-step-lower} shows $\D_2(\boldsymbol{\alpha})\ge\sigma^2L_\gamma(\alpha_0,\alpha_1)$, where $L_\gamma(\cdot,\cdot)$ is defined in~\eqref{eq:two-step-lower-envelope}. Lemma~\ref{lem:two-step-upper} gives upper bounds of $\D_2(\cdot)$ for the two families of stepsizes containing the proposed schedules. We first minimize $L_\gamma(\cdot,\cdot)$ over all nonnegative stepsize pairs, then show that an appropriate upper bound attains this minimum.

Parameters $(\gamma,\bar\alpha,\gamma_1)$ are defined in~\eqref{eq:one-two-notation}, functions $T(\cdot)$ and $U(\cdot)$ are defined in~\eqref{eq:two-step-TU}, parameters $(\hat a,\gamma_2)$ are from~\eqref{eq:two-step-parameters}, and the threshold $\gamma_\star$ is in~\eqref{eq:two-step-crossing}. We use Lemma~\ref{lem:two-step-curve} for the strict decrease of function $T(\cdot)$ over $(1,+\infty)$, the threshold inequalities~\eqref{eq:two-step-threshold-order}, and the comparison inequality~\eqref{eq:two-step-strict-chord}.

\noindent \textbf{Restricting the minimization to $[0,1]^2$.}
Using factors $(r_1,r_2,r_3,r_4)$ defined in~\eqref{eq:two-step-four-factors} and the function $\hat\rho(\alpha_0,\alpha_1)$ defined in~\eqref{eq:two-step-lower-envelope}, we show that every stepsize pair outside $[0,1]^2$ has a strictly larger value of $L_\gamma(\cdot,\cdot)$ than some stepsize pair belonging to $[0,1]^2$.

When $\alpha_0>1$, replacing $\alpha_0$ by $1$ does not increase values of $(r_1,r_2,r_3)$ and results in $r_4=0$. Thus $\hat\rho(1,\alpha_1)\le\hat\rho(\alpha_0,\alpha_1)$. Meanwhile, replacing $\alpha_0$ by $1$ strictly decreases $\alpha_0^2\max\{(1-\alpha_1)^2,(1-M\alpha_1)^2\}$ because $\max\{(1-\alpha_1)^2,(1-M\alpha_1)^2\} > 0$, $\forall\alpha_1 \geq 0$. 
Since $\gamma\ge0$, function $L_\gamma(\cdot,\cdot)$ satisfies
\begin{equation}
\hat\rho(1,\alpha_1)\le\hat\rho(\alpha_0,\alpha_1) \quad \text{and} \quad L_\gamma(1,\alpha_1)<L_\gamma(\alpha_0,\alpha_1),\qquad \forall (\alpha_0,\alpha_1)\in\R_{>1}\times\R_{\geq 0}.\label{eq:two-step-unit-square-first}
\end{equation}
If $\alpha_1>1$, replacing $\alpha_1$ by $1$ reduces $|1-\alpha_1|$ to zero and $|M\alpha_1-1|$ to $M-1$.  Consequently, after replacing $\alpha_1$ by 1, every factor entering the maximum operation defining $\rho$ is nonincreasing and $\alpha_0^2\max\{(1-\alpha_1)^2,(1-M\alpha_1)^2\}$ is nonincreasing. Since $\alpha_1^2$ strictly decreases, we obtain
\begin{equation}
\hat\rho(\alpha_0,1)\le\hat\rho(\alpha_0,\alpha_1)\quad \text{and} \quad L_\gamma(\alpha_0,1)<L_\gamma(\alpha_0,\alpha_1),\qquad \forall(\alpha_0,\alpha_1)\in\R_{\geq 0}\times\R_{>1}.\label{eq:two-step-unit-square-second}
\end{equation}
By~\eqref{eq:two-step-unit-square-first}--\eqref{eq:two-step-unit-square-second} and
$L_\gamma(\cdot,\cdot)$ being continuous over $[0,1]^2$, we know $\min_{\{\alpha_0,\alpha_1\}\subset\mathbb{R}_{\ge0}}L_\gamma(\alpha_0,\alpha_1)=\min_{\{\alpha_0,\alpha_1\}\subset[0,1]}L_\gamma(\alpha_0,\alpha_1)$ and every minimum solution pair belongs to $[0,1]^2$. We restrict the following minimization to this area.

\noindent \textbf{The case $0\le\alpha_1\le\bar\alpha$.}
When $(\alpha_0,\alpha_1)\in[0,1]\times[0,\bar\alpha]$, we have $|M\alpha_1-1|\le1-\alpha_1$, which results in $r_2\le r_3$ and $r_4\le r_1$ in~\eqref{eq:two-step-four-factors} and $\max\{(1-\alpha_1)^2,(1-M\alpha_1)^2\}=(1-\alpha_1)^2$. Hence,
\begin{equation}
L_\gamma(\alpha_0,\alpha_1)=(1-\alpha_1)^2\left(\gamma\max\{(1-\alpha_0)^2,(M\alpha_0-1)^2\}+\alpha_0^2\right)+\alpha_1^2.\label{eq:two-step-short-strip}
\end{equation}
For $0\le\alpha_0\le\bar\alpha$, the factor $\gamma\max\{(1-\alpha_0)^2,(M\alpha_0-1)^2\}+\alpha_0^2$ equals $\D_1(\alpha_0)/\sigma^2$ by~\eqref{eq:one-step-exact}. For $\alpha_0\ge\bar\alpha$, the factor further equals $\gamma(M\alpha_0-1)^2+\alpha_0^2$ and is strictly increasing as $\alpha_0$ increases. Its unique minimizer over $[0,1]$ is therefore $\alpha_*(\gamma)$ in~\eqref{eq:one-step-optimal}, with minimum value $\mathcal T(\gamma)$ in~\eqref{eq:one-step-optimal-value}. Because $(1-\alpha_1)^2>0$, the choice $\alpha_0=\alpha_*(\gamma)$ minimizes~\eqref{eq:two-step-short-strip} for each fixed $\alpha_1\in[0,\bar\alpha]$. Minimizing the resulting function $(1-\alpha_1)^2\mathcal T(\gamma)+\alpha_1^2$ over $[0,\bar\alpha]$ gives
\begin{equation}
(\alpha_0,\alpha_1)=\bigl(\alpha_*(\gamma),\alpha_*(\mathcal T(\gamma))\bigr) \quad \text{and} \quad \min_{(\alpha_0,\alpha_1)\in[0,1]\times[0,\bar\alpha]}L_\gamma(\alpha_0,\alpha_1)=\mathcal T(\mathcal T(\gamma)),\quad \forall \gamma\ge0.
\label{eq:two-step-short-minimum}
\end{equation}
The minimizer in~\eqref{eq:two-step-short-minimum} is unique by the uniqueness of the one-step minimizer in Theorem~\ref{thm:one-step}.

\noindent \textbf{The case $\bar\alpha\le\alpha_1\le1$.}
First let's suppose $\alpha_0\in[\bar\alpha,1]$. Because $M\alpha_i-1\ge1-\alpha_i\ge0$ for $i\in\{0,1\}$, we have $\hat\rho(\alpha_0,\alpha_1)=r_2=(M\alpha_0-1)(M\alpha_1-1)$. Consequently, $L_\gamma(\alpha_0,\alpha_1)=\bigl(\gamma(M\alpha_0-1)^2+\alpha_0^2\bigr)(M\alpha_1-1)^2+\alpha_1^2$.
When $(\alpha_0,\alpha_1)\in[\bar\alpha,1]^2$, $(M\alpha_1-1)^2$ is positive and increasing, and $\gamma(M\alpha_0-1)^2+\alpha_0^2$ is strictly increasing in $\alpha_0$. Since $\alpha_1^2$ is strictly increasing in $\alpha_1$, the unique minimizer of $L_\gamma(\alpha_0,\alpha_1)$ for $(\alpha_0,\alpha_1)\in[\bar\alpha,1]^2$ is $(\bar\alpha,\bar\alpha)$. 

Then it remains to consider $\alpha_0\in[0,\bar\alpha]$. Let's first define $p(\alpha_0)=\max\left\{|M\alpha_0-1|,\frac{1-\alpha_0}{1+\alpha_0(M-1)}\right\}$.
Because $M\alpha_1-1\ge1-\alpha_1$, we have $r_3\le r_2$. Taking the maximum of $r_2$ and $r_4$ gives $p(\alpha_0)(M\alpha_1-1)$. Therefore, we have $\hat\rho(\alpha_0,\alpha_1)=\max\{(1-\alpha_0)(1-\alpha_1),\,p(\alpha_0)(M\alpha_1-1)\}$.
For any fixed $\alpha_0\in[0,\bar\alpha]$, $r_1=(1-\alpha_0)(1-\alpha_1)$ is decreasing in $\alpha_1$, and $p(\alpha_0)(M\alpha_1-1)$ is increasing in $\alpha_1$. The two terms in $\hat\rho(\alpha_0, \alpha_1)$ are equal when
\begin{equation}
\alpha_1=\tfrac{1-\alpha_0+p(\alpha_0)}{1-\alpha_0+Mp(\alpha_0)}\in[\bar\alpha,1).\label{eq:two-step-beta-projection}
\end{equation}
When $\alpha_1$ is no smaller than the threshold in~\eqref{eq:two-step-beta-projection}, $\hat\rho(\alpha_0,\alpha_1)=p(\alpha_0)(M\alpha_1-1)$. In such case, $L_\gamma(\alpha_0,\alpha_1)=(\gamma p(\alpha_0)^2+\alpha_0^2)(M\alpha_1-1)^2+\alpha_1^2$ is strictly increasing in $\alpha_1$, since $M\alpha_1-1>0$. Thus, when $\alpha_1 > \frac{1-\alpha_0+p(\alpha_0)}{1-\alpha_0+Mp(\alpha_0)}$, holding $\alpha_0$ fixed and decreasing $\alpha_1$ to the threshold in~\eqref{eq:two-step-beta-projection} strictly decreases $L_\gamma(\alpha_0,\alpha_1)$ and decreases $\hat\rho(\alpha_0,\alpha_1)$ as well. When we consider minimizing $L_\gamma(\alpha_0,\alpha_1)$ over $(\alpha_0,\alpha_1)\in[0,\bar\alpha]\times[\bar\alpha,1]$, it suffices to consider the stepsize pairs with $\bar\alpha\le\alpha_1\le\frac{(1-\alpha_0+p(\alpha_0))}{(1-\alpha_0+Mp(\alpha_0))}$. On this region $r_1\ge p(\alpha_0)(M\alpha_1-1)$ and hence $\hat\rho(\alpha_0,\alpha_1)=r_1$.
When $\alpha_1$ is no larger than the threshold in~\eqref{eq:two-step-beta-projection}, we have $\alpha_1<1$ and $\max\{r_2,r_4\}=p(\alpha_0)(M\alpha_1-1)\le r_1$. Since $r_3 \le r_2$, we know $\hat\rho(\alpha_0, \alpha_1) = r_1$. From $\alpha_1\ge\bar\alpha=\frac{2}{M+1}$, we further have $M\alpha_1-1\ge1-\alpha_1>0$ and $\hat\rho(\alpha_0, \alpha_1) = r_1$. Substituting into~\eqref{eq:two-step-lower-envelope} yields that the lower bound function for $(\alpha_0, \alpha_1) \in [0, \bar\alpha] \times [\bar\alpha, \frac{1-\alpha_0+p(\alpha_0)}{1-\alpha_0+Mp(\alpha_0)}]$ becomes $L_\gamma(\alpha_0,\alpha_1)=\gamma(1-\alpha_0)^2(1-\alpha_1)^2+\alpha_0^2(M\alpha_1-1)^2+\alpha_1^2$.

Next, for every $\alpha_1 \in \left[\bar\alpha,\frac{(1-\alpha_0+p(\alpha_0))}{(1-\alpha_0+Mp(\alpha_0))}\right]$, we denote $\alpha_1$ by $\alpha_1(\zeta)$ and parametrize it by setting $\zeta=\frac{M\alpha_1(\zeta)-1}{1-\alpha_1(\zeta)}$. Then $\alpha_1 = \alpha_1(\zeta)=\frac{\zeta+1}{\zeta+M}$. From $\alpha_1(\zeta) \ge \bar\alpha$, we have $\zeta \geq 1$. It also follows from $\alpha_1(\zeta) \le \frac{(1-\alpha_0+p(\alpha_0))}{(1-\alpha_0+Mp(\alpha_0))}$ that $\zeta \le \frac{1-\alpha_0}{p(\alpha_0)}$. Since $p(\alpha_0) \geq \frac{1-\alpha_0}{1+ (M-1)\alpha_0}$ and $p(\alpha_0) \geq |M\alpha_0 - 1| \geq M\alpha_0 - 1$, we have $\zeta\le1+(M-1)\alpha_0$ and $\zeta(M\alpha_0-1)\le1-\alpha_0$. Rearranging gives $\frac{\zeta-1}{M-1}\leq \alpha_0\leq \frac{\zeta+1}{M\zeta+1}$, which further implies $M\zeta^2-2(M-1)\zeta-M\le0$ whose positive root is $\hat a$, so $\zeta\le\hat a$. In particular, when $\alpha_1(\zeta)$ equals the threshold in~\eqref{eq:two-step-beta-projection}, we have $p(\alpha_0)(M\alpha_1(\zeta)-1)=(1-\alpha_0)(1-\alpha_1(\zeta))$, which leads to $\zeta=\frac{1-\alpha_0}{p(\alpha_0)}$ and further yields $\frac{1-\alpha_0}{p(\alpha_0)}\le\hat a$ by $\zeta\le\hat a$. Hence the entire interval $\bigl[1,\frac{1-\alpha_0}{p(\alpha_0)}\bigr]$ lies in the domain $[1,\hat a]$ of $\alpha_1(\zeta)$ in~\eqref{eq:two-step-boundary-pair}. 
Thus, minimizing $L_\gamma(\alpha_0,\alpha_1)$ over $[0,\bar\alpha]\times[\bar\alpha,1]$ is equivalent to minimizing it over the following reduced set of stepsize pairs parameterized by $\zeta$,
\begin{equation}
\alpha_1=\alpha_1(\zeta)\quad \text{and} \quad \tfrac{\zeta-1}{M-1}\le\alpha_0\le\tfrac{\zeta+1}{M\zeta+1},\qquad \forall \zeta\in[1, \hat a].\label{eq:two-step-lens}
\end{equation}
For each stepsize pair in~\eqref{eq:two-step-lens},
\begin{equation}
L_\gamma(\alpha_0,\alpha_1(\zeta))=\gamma(1-\alpha_0)^2(1-\alpha_1(\zeta))^2+\alpha_0^2(M\alpha_1(\zeta)-1)^2+\alpha_1(\zeta)^2.\label{eq:two-step-lens-objective}
\end{equation}
Having identified the representative lower bounding functions in~\eqref{eq:two-step-short-strip} and~\eqref{eq:two-step-lens-objective}, we next analyze them to determine the minimum of $L_{\gamma}(\alpha_0,\alpha_1)$ over $[0,1]^2$.

\noindent \textbf{Minimizing $L_{\gamma}(\alpha_0,\alpha_1)$ over $[0,1]^2$.}
The minimum value of $L_{\gamma}(\alpha_0,\alpha_1)$ over $[0,1]\times[0,\bar\alpha]$ is $\mathcal{T}(\mathcal{T}(\gamma))$, given by~\eqref{eq:two-step-short-minimum}.  We now compare that minimum in~\eqref{eq:two-step-short-minimum} with the values of $L_\gamma$ in~\eqref{eq:two-step-lens-objective} for $\alpha_1\in[\bar\alpha, 1]$.
We first prove the following inequality for use in the subsequent comparison:
\begin{equation}
\hat\rho(\alpha_0,\alpha_1)^2\ge T(\hat a),\qquad \forall (\alpha_0, \alpha_1)\in [0,1]\times[0,1].\label{eq:two-step-global-rate-lower}
\end{equation}
where function $T(\cdot)$ is defined in~\eqref{eq:two-step-TU}. When $\alpha_1\le\bar\alpha$, we have $r_2\le r_3$ and $r_4\le r_1$, which further imply $\hat\rho(\alpha_0,\alpha_1)=(1-\alpha_1)\max\{|1-\alpha_0|,|1-M\alpha_0|\}\ge(1-\bar\alpha)^2$.
Here, $1-\alpha_1\ge1-\bar\alpha$, and the quantity $\max\{|1-\alpha_0|,|1-M\alpha_0|\}$ is minimized at $\alpha_0 = \bar\alpha$ with the value $1-\bar\alpha$. This gives $\hat\rho(\alpha_0, \alpha_1)^2\ge (1-\bar\alpha)^4 = T(1)>T(\hat a)$. When $\min\{\alpha_0,\alpha_1\}\ge\bar\alpha$, $(M\alpha_0-1)(M\alpha_1-1)$ is at least $(1-\bar\alpha)^2$, so $r_2^2\ge T(1)>T(\hat a)$. In the remaining setting of $(\alpha_0,\alpha_1)\in[0,\bar\alpha]\times[\bar\alpha,1]$, if $\alpha_1$ exceeds the threshold in~\eqref{eq:two-step-beta-projection}, lowering $\alpha_1$ to the threshold could decrease $\hat\rho(\alpha_0,\alpha_1)$ and yields stepsize pairs in~\eqref{eq:two-step-lens}. For each $(\alpha_0,\alpha_1(\zeta))$ described in~\eqref{eq:two-step-lens}, we have $\hat\rho(\alpha_0,\alpha_1(\zeta))^2=(1-\alpha_0)^2(1-\alpha_1(\zeta))^2\ge T(\zeta)\ge T(\hat a)$, where the first inequality uses $\alpha_0\le\frac{\zeta+1}{M\zeta+1}$ and~\eqref{eq:two-step-coefficient-identities}, and the second inequality applies $1\le\zeta\le\hat a$ and Lemma~\ref{lem:two-step-curve}. Therefore,~\eqref{eq:two-step-global-rate-lower} has been proved. 

Now we can discuss the solution and value of $\min_{(\alpha_0,\alpha_1)\in[0,1]^2}L_{\gamma}(\alpha_0, \alpha_1)$ in each region of $\gamma$.

\emph{(i) When $0\le\gamma\le\gamma_1$.}
It follows~\eqref{eq:two-step-four-factors}--\eqref{eq:two-step-lower-envelope} that
\begin{equation}
L_\gamma(\alpha_0,\alpha_1)\ge\left(\gamma(1-\alpha_0)^2+\alpha_0^2\right)(1-\alpha_1)^2+\alpha_1^2\ge\tfrac{\gamma}{1+2\gamma},\label{eq:two-step-universal-small-lower}
\end{equation}
where the last inequality holds as an equality if and only if $(\alpha_0,\alpha_1) = \left(\frac{\gamma}{1+\gamma},\frac{\gamma}{1+2\gamma}\right)$ which also coincides with the unique minimizer of $L_{\gamma}(\alpha_0,\alpha_1)$ in the case of $\gamma\in[0,\gamma_1]$. Moreover, these stepsizes also agree with those presented in~\eqref{eq:two-step-short-minimum}.


\emph{(ii) When $\gamma_1<\gamma<\gamma_\star$.}
We first compare $L_\gamma(\alpha_0, \alpha_1)$ with the value $\gamma T(1)+U(1)$ at $\gamma=\gamma_1$ and $\gamma=\gamma_\star$ over $(\alpha_0, \alpha_1)\in[0,1]\times[\bar\alpha,1]$.

When $\gamma=\gamma_1$, minimizing $\gamma_1(1-\alpha_0)^2+\alpha_0^2$ in~\eqref{eq:two-step-universal-small-lower} results in the minimum solution $\alpha_0 = \bar\alpha$ together with the minimum value $\bar\alpha$. Therefore, by~\eqref{eq:two-step-universal-small-lower}, 
$L_{\gamma_1}(\alpha_0,\alpha_1)\ge\bar\alpha(1-\alpha_1)^2+\alpha_1^2$, where the right-hand-side quantity further achieves the minimum value $\gamma_1T(1)+U(1)$ when $\alpha_0=\alpha_1=\bar\alpha$.

When $\gamma=\gamma_\star$, we only need to consider the stepsize pairs $(\alpha_0,\alpha_1(\zeta))$ defined in~\eqref{eq:two-step-lens} for minimizing $L_{\gamma_\star}(\alpha_0,\alpha_1)$ over $(\alpha_0,\alpha_1)\in[0,1]\times[\bar\alpha,1]$. For fixed $\zeta$, it follows~\eqref{eq:two-step-lens-objective} that $\frac{\partial L_{\gamma_\star}(\alpha_0,\alpha_1(\zeta))}{\partial\alpha_0}=2(1-\alpha_1(\zeta))^2((\gamma_\star+\zeta^2)\alpha_0-\gamma_\star)$. When setting $\alpha_0$ to its maximum value $\frac{\zeta+1}{M\zeta+1}$ in~\eqref{eq:two-step-lens}, we have $(\gamma_\star+\zeta^2)\alpha_0-\gamma_\star = (\gamma_\star+\zeta^2)\cdot\frac{\zeta+1}{M\zeta+1}-\gamma_\star = \frac{\zeta(\zeta(\zeta+1)-\gamma_\star(M-1))}{M\zeta+1} < 0$, where the inequality follows~\eqref{eq:two-step-threshold-order} and $\zeta(\zeta+1)\le\hat a(\hat a+1)$.
Hence, $L_{\gamma_\star}(\alpha_0,\alpha_1(\zeta))$ is strictly decreasing in $\alpha_0$ throughout the interval in~\eqref{eq:two-step-lens}. For any fixed $\zeta\in[1,\hat{a}]$ and $\alpha_1(\zeta)$, the minimum solution of $L_{\gamma_\star}(\alpha_0,\alpha_1(\zeta))$ would be $\alpha_0=\frac{\zeta+1}{M\zeta+1}$ with the minimum value $\gamma_\star T(\zeta)+U(\zeta)$. Equation~\eqref{eq:two-step-strict-chord} then shows that the capped pair $(\bar\alpha,\bar\alpha)$ at $\zeta=1$ and the \texttt{silver} stepsize pair at $\zeta=\hat a$ in~\eqref{eq:two-step-endpoints} are the only minimizers of $L_{\gamma_\star}(\alpha_0,\alpha_1)$ over $(\alpha_0,\alpha_1)\in[0,1]\times[\bar\alpha,1]$, and both attain the minimum value $\gamma_\star T(1)+U(1)=\gamma_\star T(\hat a)+U(\hat a)$. Therefore, combining the cases of $\gamma = \gamma_1$ and $\gamma = \gamma_*$, we have that for any $(\alpha_0, \alpha_1)\in[0,1]\times[\bar\alpha, 1]$
\begin{equation}\label{eq:two-step-endpoint-comparison}
    L_{\gamma_1}(\alpha_0,\alpha_1) \ge \gamma_1T(1)+U(1),\quad \text{and} \quad L_{\gamma_\star}(\alpha_0,\alpha_1) \ge \gamma_\star T(1)+U(1).
\end{equation}

For each fixed $(\alpha_0,\alpha_1)\in[0,1]\times[\bar\alpha,1]$, by~\eqref{eq:two-step-lower-envelope}, the difference function $L_\gamma(\alpha_0,\alpha_1)-(\gamma T(1)+U(1))$ is affine in $\gamma$. Interpolating between $\gamma_1$ and $\gamma_\star$ gives that for any $\gamma\in(\gamma_1,\gamma_\star)$,
\begin{equation}
\begin{aligned}
L_\gamma(\alpha_0,\alpha_1)-(\gamma T(1)+U(1))=&\ \tfrac{\gamma_\star-\gamma}{\gamma_\star-\gamma_1}\cdot\bigl(L_{\gamma_1}(\alpha_0,\alpha_1)-(\gamma_1T(1)+U(1))\bigr) \\
&+\tfrac{\gamma-\gamma_1}{\gamma_\star-\gamma_1}\cdot\bigl(L_{\gamma_\star}(\alpha_0,\alpha_1)-(\gamma_\star T(1)+U(1))\bigr)\ge0,
\end{aligned}
\label{eq:two-step-affine-comparison}
\end{equation}
where the last inequality follows~\eqref{eq:two-step-endpoint-comparison}. Since the interpolation weights $\frac{\gamma_\star-\gamma}{\gamma_\star-\gamma_1}$ and $\frac{\gamma-\gamma_1}{\gamma_\star-\gamma_1}$ are both strictly positive, the last inequality in~\eqref{eq:two-step-affine-comparison} holding as an equality requires $L_{\gamma_1}(\alpha_0,\alpha_1) = \gamma_1T(1)+U(1)$ and $L_{\gamma_\star}(\alpha_0,\alpha_1) = \gamma_\star T(1)+U(1)$. It follows~\eqref{eq:two-step-four-factors}--\eqref{eq:two-step-lower-envelope} that, $L_{\gamma_1}(\alpha_0,\alpha_1) \ge (1-\alpha_1)^2 (\bar\alpha+(1+\gamma_1)(\alpha_0-\bar\alpha)^2) +\alpha_1^2$.
The function $\bar\alpha(1-\alpha_1)^2+\alpha_1^2$ is strictly increasing in $\alpha_1\in[\bar\alpha,1]$ and equals $\gamma_1T(1)+U(1)$ at $\alpha_1=\bar\alpha$. Thus, $L_{\gamma_1}(\alpha_0,\alpha_1) = \gamma_1 T(1)+U(1)$ directly implies $\alpha_0 = \alpha_1=\bar\alpha$. 
To check $L_{\gamma_\star}(\alpha_0,\alpha_1) = \gamma_\star T(1)+U(1)$, we substitute $(\alpha_0,\alpha_1)=(\bar\alpha,\bar\alpha)$ into~\eqref{eq:two-step-four-factors}--\eqref{eq:two-step-lower-envelope}, leading to $L_{\gamma_\star}(\bar\alpha,\bar\alpha) =\gamma_\star(1-\bar\alpha)^4 +\bar\alpha^2(1-\bar\alpha)^2+\bar\alpha^2 =\gamma_\star T(1)+U(1)$,
where the last equality uses~\eqref{eq:two-step-coefficient-identities} at $\zeta=1$. Hence, $(\alpha_0,\alpha_1) = (\bar\alpha,\bar\alpha)$ is the only stepsize pair that attains $L_{\gamma}(\alpha_0,\alpha_1) = \gamma T(1)+U(1)$ for $\gamma = \gamma_1$ and $\gamma = \gamma_*$, and thus the only stepsize pair that makes the inequality in~\eqref{eq:two-step-affine-comparison} hold with equality. 

Therefore, $(\alpha_0,\alpha_1) = (\bar\alpha,\bar\alpha)$ is the unique solution for $\min_{(\alpha_0,\alpha_1)\in[0,1]\times[\bar\alpha,1]}L_\gamma(\alpha_0,\alpha_1)$ when $\gamma\in(\gamma_1,\gamma_\star)$, which also belongs to $[0,1]\times[0,\bar\alpha]$. Hence, by~\eqref{eq:two-step-short-minimum}, $\mathcal T(\mathcal T(\gamma))$ and $(\alpha_*(\gamma),\alpha_*(\mathcal T(\gamma)))$ also serve as the global minimum value and minimum solution for $\min_{(\alpha_0,\alpha_1)\in[0,1]^2}L_\gamma(\alpha_0,\alpha_1)$ when $\gamma\in(\gamma_1,\gamma_\star)$. 

\emph{(iii) When $\gamma\ge\gamma_\star$.}
For $(\alpha_0,\alpha_1)\in[0,1]\times[\bar\alpha,1]$,~\eqref{eq:two-step-strict-chord} and~\eqref{eq:two-step-endpoint-comparison} imply $L_{\gamma_\star}(\alpha_0,\alpha_1)\ge\gamma_\star T(\hat a)+U(\hat a)$. In the case of $\gamma > \gamma_*$, by~\eqref{eq:two-step-global-rate-lower}, we further have
\begin{equation}
\begin{aligned}
L_\gamma(\alpha_0,\alpha_1)-(\gamma T(\hat a)+U(\hat a))=\ &L_{\gamma_\star}(\alpha_0,\alpha_1)-(\gamma_\star T(\hat a)+U(\hat a))\\
&+(\gamma-\gamma_\star)(\hat\rho(\alpha_0,\alpha_1)^2-T(\hat a)) \ge0.
\end{aligned}
\label{eq:two-step-large-budget-comparison}
\end{equation}
By~\eqref{eq:two-step-endpoint-comparison} and~\eqref{eq:two-step-crossing}, $L_{\gamma_\star}(\alpha_0,\alpha_1) \ge \gamma_\star T(1)+U(1) = \gamma_\star T(\hat a)+U(\hat a)$.
Also, $\hat\rho(\alpha_0,\alpha_1)^2\ge T(\hat a)$ by~\eqref{eq:two-step-global-rate-lower}. 
Therefore, the inequality in~\eqref{eq:two-step-large-budget-comparison} holds as an equality if and only if $L_{\gamma_\star}(\alpha_0,\alpha_1)=\gamma_\star T(\hat a)+U(\hat a)$ and $\hat\rho(\alpha_0,\alpha_1)^2=T(\hat a)$, which, by the help of~\eqref{eq:two-step-coefficient-identities}--\eqref{eq:two-step-endpoints} and~\eqref{eq:two-step-lens}--\eqref{eq:two-step-lens-objective}, implies that the \texttt{silver} stepsize pair is the unique optimal solution for $\min_{(\alpha_0,\alpha_1)\in[0,1]\times[\bar\alpha,1]}L_{\gamma}(\alpha_0,\alpha_1)$ when $\gamma>\gamma_\star$.


At $\gamma=\gamma_\star$, we only need to consider the stepsize pairs $(\alpha_0,\alpha_1(\zeta))$ defined in~\eqref{eq:two-step-lens} for minimizing $L_{\gamma_\star}(\alpha_0,\alpha_1)$ over $(\alpha_0,\alpha_1)\in[0,1]\times[\bar\alpha,1]$. By~\eqref{eq:two-step-strict-chord} in Lemma~\ref{lem:two-step-curve}, for any $\zeta\in(1,\hat a)$, $\gamma_\star T(\zeta)+U(\zeta) > \gamma_\star T(1)+U(1) = \gamma_\star T(\hat a)+U(\hat a)$. Thus $\gamma_\star T(\zeta)+U(\zeta)$ attains its minimum over $[1,\hat a]$ exactly at $\zeta=1$ and $\zeta=\hat a$, which result in the capped and \texttt{silver} stepsize pairs in~\eqref{eq:two-step-endpoints}, respectively.

When $(\alpha_0, \alpha_1)\in[0,1]\times[0,\bar\alpha]$, by~\eqref{eq:two-step-short-minimum}, the unique minimizer for $\min_{(\alpha_0, \alpha_1)\in[0,1]\times[0,\bar\alpha]}L_\gamma(\alpha_0,\alpha_1)$ is $(\alpha_*(\gamma),\alpha_*(\mathcal T(\gamma)))$. Since $\gamma\ge\gamma_\star>\gamma_2>\gamma_1$, by~\eqref{eq:one-step-optimal}--\eqref{eq:two-step-parameters}, we directly have $(\alpha_*(\gamma),\alpha_*(\mathcal T(\gamma))) = (\bar\alpha,\bar\alpha)$ with $L_\gamma(\bar\alpha,\bar\alpha) = \gamma T(1)+U(1)$. 

Thus, for every $\gamma\ge\gamma_\star$, it follows previous results that
\begin{align}
&\min_{(\alpha_0,\alpha_1)\in[0,1]\times[\bar\alpha,1]}L_\gamma(\alpha_0,\alpha_1) - \min_{(\alpha_0,\alpha_1)\in[0,1]\times[0,\bar\alpha]}L_\gamma(\alpha_0,\alpha_1) \label{eq.Lgamma_diff} \\
= \ &(\gamma T(\hat a)+U(\hat a)) - (\gamma T(1)+U(1)) = \gamma(T(\hat a)-T(1))+U(\hat a) - U(1) = (\gamma-\gamma_\star)(T(\hat a)-T(1)), \notag
\end{align}
where the last equality uses~\eqref{eq:two-step-crossing}. Since $T(1)>T(\hat a)$, the difference is strictly negative when $\gamma>\gamma_\star$. Hence the \texttt{silver} pair is the unique minimizer over $[0,1]^2$ when $\gamma>\gamma_\star$. When $\gamma=\gamma_\star$, $\min_{(\alpha_0,\alpha_1)\in[0,1]\times[\bar\alpha,1]}L_\gamma(\alpha_0,\alpha_1) = \min_{(\alpha_0,\alpha_1)\in[0,1]\times[0,\bar\alpha]}L_\gamma(\alpha_0,\alpha_1)$, implying the capped and \texttt{silver} stepsize pairs are both optimal solutions.

\noindent \textbf{Matching the upper bound in each $\gamma$ region.}
First, let's summarize results regarding optimal lower bounds for $\gamma$ taking different values.
When $0\le\gamma<\gamma_\star$, by~\eqref{eq:two-step-short-minimum}, $(\alpha_*(\gamma),\alpha_*(\mathcal T(\gamma)))$ and $\mathcal T(\mathcal T(\gamma))$ are the optimal solution and the optimal value of $\min_{(\alpha_0,\alpha_1)\in[0,1]^2} L_\gamma(\alpha_0,\alpha_1)$. On the other hand, when $\gamma \geq \gamma_\star$, from~\eqref{eq:two-step-crossing},~\eqref{eq:two-step-large-budget-comparison}, and~\eqref{eq.Lgamma_diff}, we know that the \texttt{silver} stepsize pair is the unique minimizer for $\min_{(\alpha_0,\alpha_1)\in[0,1]^2} L_\gamma(\alpha_0,\alpha_1)$ when $\gamma>\gamma_\star$, and both the capped and \texttt{silver} stepsize pairs are optimal solutions when $\gamma = \gamma_\star$.




Next, we turn to identify upper bounds that match lower bounding functions.
Part~\textup{(i)} of Lemma~\ref{lem:two-step-upper} matches $\sigma^2\mathcal T(\mathcal T(\gamma))$ when $\gamma\in[0,\gamma_\star)$. Part~\textup{(ii)} of Lemma~\ref{lem:two-step-upper} at $\zeta=\hat a$ matches $\sigma^2(\gamma T(\hat a)+U(\hat a))$ when $\gamma \geq \gamma_\star$.

Finally, the schedule in~\eqref{eq:two-step-optimal-pair} is well defined for every $\gamma\ge0$. Its four regions are ordered by~\eqref{eq:two-step-threshold-order}. For the \texttt{silver} stepsize pair, $\eta$ is defined in~\eqref{eq:two-step-parameters} and satisfies $(2M-1)^2-\eta^2=2M(M-1)>0$, so $\eta<2M-1$, $1+\eta>0$, and $1+2M-\eta>2$. Every prescribed stepsize is therefore finite and nonnegative. Direct substitution gives $\frac{\gamma_1}{1+\gamma_1}=\bar\alpha$, $\frac{\gamma_1}{1+2\gamma_1}=\frac2{M+3}=\frac{\gamma_1(M-1)^2+4}{\gamma_1(M-1)^2+4+(M+1)^2}$, $\frac{\gamma_2(M-1)^2+4}{\gamma_2(M-1)^2+4+(M+1)^2}=\bar\alpha=\gamma_2T(1)+U(1)$, and $\gamma_\star T(1)+U(1)=\gamma_\star T(\hat a)+U(\hat a)$.
Thus the adjacent stepsize formulas agree at $\gamma = \gamma_1$ and $\gamma = \gamma_2$, and the adjacent value formulas agree at all three thresholds, as claimed.
\end{proof}

\section{Proof of Theorem~\ref{thm:silver-bound}}
\label{app:silver-weights}

\begin{proof}[Proof of Theorem~\ref{thm:silver-bound}]
We prove the upper bound by constructing a certificate for each \texttt{silver} block, starting from the seed certificate and specifying how the certificates for $A$ and $B$ produce a certificate for $A\bowtie B$. Consider an arbitrary admissible instance in~\eqref{eq:worst-case-error}, with objective $f\in\mathcal{F}_{1,M}$, initial point $x_0\in\R^d$, and stochastic oracle satisfying Assumption~\ref{ass:oracle}. For a block $A=(a_0,\ldots,a_{p-1})\in\mathcal S_p$ with terminal parameter $u$, use functions $e(\cdot)$ and $\tau(\cdot)$ from~\eqref{eq:composition-terminal-parameters}. Let $x_k$ be the iterate in~\eqref{eq:sg-update} with stepsizes $a_k$, and let $g_k$ be the stochastic gradient estimate evaluated at $x_k$. Translating the minimizer and subtracting the minimum function value allows us to suppose $x_\star=0$ and $f(0)=0$. We also denote stochastic gradient error as $\epsilon_k=g_k-\nabla f(x_k)$ for $k\in[p-1]$.
We prove by induction over the construction of the \texttt{silver} blocks that nonnegative quantities $(\mathcal W_A,\mathcal R_A)$ can be constructed so that
\begin{equation}
\begin{aligned}
&\tau(u)\|x_0\|^2+\left(\sum_{k=0}^{p-1}\mu_k^A\E[\|\epsilon_k\|^2]\right)-\E[\|x_p\|^2]=\mathcal W_A+\mathcal R_A,\\
&\mu_{p-1}^A=a_{p-1}^2,\quad \text{and} \quad \mu_k^A=a_k^2\prod_{r=k+1}^{p-1}(1-Ma_r)^2\quad \forall k\in[p-2].
\end{aligned}
\label{eq:silver-expected-certificate}
\end{equation}
Once~\eqref{eq:silver-expected-certificate}, $\mathcal W_A\ge0$, and $\mathcal R_A\ge0$ have been proved, Assumption~\ref{ass:oracle} naturally leads to $\E[\|x_p\|^2]\le\tau(u)\|x_0\|^2+\sigma^2\sum_{k=0}^{p-1}\mu_k^A$.

\noindent \textbf{The expected interpolation inequalities.}
Fix $f\in\mathcal{F}_{1,M}$ and consider the $p$-step SG method~\eqref{eq:sg-update} with stepsize block $A=(a_0,\ldots,a_{p-1})$. For $\ell\in [p-1]$, define $Z_\ell=(x_0,g_0,\ldots,g_{\ell-1})$ and $G_\ell=(g_\ell,\ldots,g_{p-1})$. Here, $Z_\ell$ consists of the initial point and the stochastic oracle outputs prior to iteration $\ell$, and $G_\ell$ includes stochastic gradients generated from oracle calls from iteration $\ell$ through $p-1$. By Assumption~\ref{ass:oracle}, $\mathcal H_\ell=\boldsymbol{\sigma}(Z_\ell)$. We will compare the queries $x_{i;l}$ and $x_{j;l}$ at indices $i$ and $j$ in two auxiliary trajectories that are both conditioned on $Z_\ell$, where $(i,j)\in([p-1]\backslash[\ell-1])^2$.

\emph{(i) Conditional distribution of the remaining outputs.}
Since $Z_\ell$ and $G_\ell$ are finite tuples of Borel-measurable random vectors, they are themselves Borel measurable. Their Euclidean value spaces, equipped with the respective Borel $\sigma$-algebras, are standard Borel spaces. Thus, by~\cite[Theorem~8.5]{kallenberg2021foundations}, there exists a regular conditional distribution $K_\ell$ of $G_\ell$ given $Z_\ell$, satisfying that
\begin{equation}
\text{for every Borel set $C$,}\qquad\mathbb P(G_\ell\in C\mid\mathcal H_\ell)=K_\ell(Z_\ell,C)
\qquad\text{almost surely}.
\label{eq:silver-continuation-kernel}
\end{equation}

\emph{(ii) Conditional independence of the two trajectories.}
By the conditional product construction based on~\cite[Lemma~3.3]{kallenberg2021foundations}, the two auxiliary trajectories are conditionally independent given $\mathcal H_\ell$. Since the query points $x_{i;l}$ and $x_{j;l}'$ are measurable functions of the shared history and their respective trajectories through~\eqref{eq:sg-update}, they are also conditionally independent given $\mathcal H_\ell$. Consequently, for any Borel maps $\phi,\psi:\mathbb R^d\to\mathbb R^m$ such that $\phi(x_{i;l})$ and $\psi(x_{j;l}')$ have finite second moments, conditional independence yields
\begin{equation}
\E[\langle\phi(x_{i;l}),\psi(x_{j;l}')\rangle\mid\mathcal H_\ell] = \left\langle \E[\phi(x_{i;l})\mid\mathcal H_\ell], \E[\psi(x_{j;l}')\mid\mathcal H_\ell] \right\rangle.
\label{eq:silver-conditional-factorization}
\end{equation}

\emph{(iii) Distribution of each auxiliary trajectory.}
By~\cite[Theorem~8.5]{kallenberg2021foundations}, each auxiliary trajectory has the same joint distribution of the shared history and remaining oracle outputs as $(Z_\ell,G_\ell)$. Since the query points are determined by the same update~\eqref{eq:sg-update}, each auxiliary trajectory has the same joint distribution of query points and oracle outputs as the original stochastic process. Consequently, for any $i\in[p-1]\backslash[\ell-1]$ and any Borel function $\phi$ such that $\phi(x_{i;l})$ is integrable, the two trajectories have the same conditional distribution given $\mathcal H_\ell$, yielding
\begin{equation}
\E[\phi(x_{i;l}')\mid\mathcal H_\ell] =\E[\phi(x_{i;l})\mid\mathcal H_\ell] \qquad \text{and} \qquad \E[\phi(x_{i;l}')]=\E[\phi(x_{i;l})],
\label{eq:silver-same-run-moments}
\end{equation}
where the second equality follows from the tower property, since $\mathcal H_0\subseteq\mathcal H_\ell$. We further define
\begin{equation}
\mathcal Q_{ij;\ell}=\E[Q(x_{i;l},x_{j;l}')],\qquad
\mathcal Q_{\star i;0}=\E[Q(0,x_{i;0})],\qquad \text{and} \qquad
\mathcal Q_{i\star;0}=\E[Q(x_{i;0},0)],
\label{eq:composition-coupled-Q}
\end{equation}
where function $Q(\cdot,\cdot)$ is defined in~\eqref{eq:two-step-interpolation}.
The expectation in the definition of $\mathcal Q_{ij;\ell}$ averages over the shared history and both trajectories constructed in (ii). In $\mathcal Q_{\star i;0}$ and $\mathcal Q_{i\star;0}$, the minimizer $x_\star=0$ is fixed, so the expectation only averages over the query $x_i$ from a single SG trajectory.~\eqref{eq:two-step-interpolation} gives $Q(x,y)\ge0$ for every pair of points of the same objective $f$. Applying it to the realized pairs in~\eqref{eq:composition-coupled-Q} proves that all three types of expected interpolation terms are nonnegative.

We construct the nonnegative term $\mathcal W_A$ in~\eqref{eq:silver-expected-certificate} as a weighted sum of the expectation terms in~\eqref{eq:composition-coupled-Q} with non-negative weights $\lambda^A_{ij;\ell}$:
\begin{equation}
\mathcal W_A=\sum_{i,j,\ell}\lambda^A_{ij;\ell}\mathcal Q_{ij;\ell}.
\label{eq:silver-interpolation-sum}
\end{equation}
Since each $\mathcal Q_{ij;\ell}\geq 0$, nonnegative weights $\lambda^A_{ij;\ell}$ ensure that $\mathcal W_A\ge0$. We specify these weight quantities separately for the one-step seed and the \texttt{silver} join of two blocks, setting all unspecified weights to zero.

\noindent \textbf{The seed certificate.}
For $A=(\bar\alpha)$, we set
\begin{equation}
\lambda^{(\bar\alpha)}_{0\star;0}=\lambda^{(\bar\alpha)}_{\star0;0}=\frac{4(M-1)}{(M+1)^2}\qquad \text{and} \qquad \mathcal R_{(\bar\alpha)}=0.
\label{eq:silver-certificate-seed}
\end{equation}
Applying $x_\star=0$, $f(0)=0$, and $\nabla f(0)=0$ to~\eqref{eq:two-step-interpolation}, we obtain
\begin{equation}
\begin{aligned}
(M-1)\bigl(Q(x,0)+Q(0,x)\bigr)
&=(M+1)\langle\nabla f(x),x\rangle-\|\nabla f(x)\|^2-M\|x\|^2\\
&=-\langle\nabla f(x)-Mx,\nabla f(x)-x\rangle.
\end{aligned}
\label{eq:silver-star-sum}
\end{equation}
For the seed block $(\bar\alpha)$, we have $\tau(1/M)=(1-\bar\alpha)^2$ by~\eqref{eq:one-two-notation} and~\eqref{eq:composition-terminal-parameters}. From~\eqref{eq:sg-update} and the conditional unbiasedness in~\eqref{eq:unbiasedness}, the first equation in~\eqref{eq:silver-expected-certificate} with $\mu_0^A=\bar\alpha^2$ becomes $(1-\bar\alpha)^2\|x_0\|^2-\|x_0-\bar\alpha\nabla f(x_0)\|^2
=\bar\alpha^2\bigl((M+1)\langle x_0,\nabla f(x_0)\rangle-\|\nabla f(x_0)\|^2-M\|x_0\|^2\bigr)= \frac{4(M-1)}{(M+1)^2}\bigl(Q(x_0,0)+Q(0,x_0)\bigr) = \mathcal W_{(\bar\alpha)} + \mathcal R_{(\bar\alpha)} = \mathcal W_{(\bar\alpha)} \geq 0$,
where the first two equalities use $2\bar\alpha=(M+1)\bar\alpha^2$ and~\eqref{eq:silver-star-sum}, the last equality follows~\eqref{eq:silver-certificate-seed}, and the nonnegativity inequality is from~\eqref{eq:two-step-interpolation} and~\cite[Theorem~4]{taylor2017smooth}.

Consider a block $A$ with length $p\ge1$ and terminal parameter $u\in(0,1)$, for every $i\in[p-1]$, we set
\begin{equation}
\begin{aligned}
\lambda^A_{i\star;0}
&=\mathbf{1}_{\{i=p-1\}}\cdot\tfrac{4(M-1)u}{M(1+u)^2}\\
\text{and}\quad \lambda^A_{\star i;0}
&=\mathbf{1}_{\{i=p-1\}}\cdot\tfrac{2(M-1)[1+(M-1)u+Mu^2]}{M^2(1+u)^2}\\
&\qquad+
\begin{cases}
0,&\text{if }p=1\text{ or }i=p-1,\\[3pt]
\frac{2(M-1)^2(1-u)}{M^2(1+u)^2}a_i\prod_{r=i+1}^{p-2}(1-a_r),&\text{otherwise},
\end{cases}
\end{aligned}
\label{eq:silver-star-interface}
\end{equation}
where $\mathbf{1}_{\{i=p-1\}} = 1$ if $i=p-1$ and $\mathbf{1}_{\{i=p-1\}} = 0$ if $i\neq p-1$. By convention, we consider empty products being $1$ and empty sums being $0$. In the seed block $(\bar\alpha)$, we have $p=1$, $i=0$, and $u=1/M$, so~\eqref{eq:silver-star-interface} recovers the two coefficients in~\eqref{eq:silver-certificate-seed}. 

\noindent \textbf{Constructing the certificate of the joined block.}
Let $A=(a_0,\ldots,a_{p-1})$ and $B=(b_0,\ldots,b_{q-1})$ be \texttt{silver} blocks with terminal parameters $\{u,v\}\subset(0,1)$. Suppose~\eqref{eq:silver-expected-certificate} and~\eqref{eq:silver-star-interface} hold for blocks $A$ and $B$. We set $N=p+q$, let $A\bowtie B=(\alpha_0,\ldots,\alpha_{N-1})$ have the corresponding parameters $(y,z)$ defined in~\eqref{eq:silver-join-parameters}, and consider $(s_A,s_B)$ defined in~\eqref{eq:silver-inheritance-scales}. We will analyze the performance of the SG method taking stepsizes $A\bowtie B=(\alpha_0,\ldots,\alpha_{N-1})$.
Let's define $K$, $H$, and the auxiliary weights $\{\lambda_j\}_{j=p}^{N-2}$ by
\begin{equation}
K=\tfrac{4(M-1)uz(1-z)^2}{M^2(z-u)^2(1+z)^2},
\ \ H=K\cdot\tfrac{1-u}{1-z},\ \ \text{and} \ \ \lambda_j=\alpha_{j}\prod_{r=j+1}^{N-2}(1-\alpha_r)
\ \forall j\in[N-2]\backslash[p-1].
\label{eq:silver-interface-coefficients}
\end{equation}
In fact, $\{\lambda_j\}_{j=p}^{N-2}$ is meaningful only when $N-2=p+q-2\geq p$, or equivalently, $q\geq 2$. When $q=1$, the quantities $\lambda_j$ are not needed in~\eqref{eq:silver-interface-coefficients}. For $q\geq 2$, we adopt the convention that $\lambda_{N-2}=\alpha_{N-2}$.
The $\mathcal{Q}$ terms for $A\bowtie B$ are defined by~\eqref{eq:composition-coupled-Q} with iteration indices $[N-1]$. We define their coefficients as follows.

\emph{Coefficients inherited from blocks A and B.}
Set
\begin{equation}
\begin{aligned}
\lambda^{A\bowtie B}_{ij;\ell} &=s_A\lambda^A_{ij;\ell}, &&\forall (i,j)\in[p-1]^2\ \ \text{and} \ \ \forall\ell\in[\min\{i,j\}],\\
\lambda^{A\bowtie B}_{p+i,p+j;p+\ell} &=s_B\lambda^B_{ij;\ell}, &&\forall (i,j)\in[q-1]^2\ \ \text{and} \ \ \forall\ell\in[\min\{i,j\}], \\
\text{and} \quad \lambda^{A\bowtie B}_{\star i;0}&=s_A\lambda^A_{\star i;0}, &&\forall i\in[p-2].
\end{aligned} \label{eq:silver-inherited-coefficients}
\end{equation}
\emph{Coefficients assigned by the \texttt{silver} join.}
For every $j\in[N-2]\backslash[p-1]$, we set
\begin{equation}
\lambda^{A\bowtie B}_{p-1,j;p-1}=(M-1)K\lambda_j,\ \  \lambda^{A\bowtie B}_{N-1,j;p-1}=(M-1)H\lambda_j, \ \  \text{and} \ \  \lambda^{A\bowtie B}_{\star j;0}=\tfrac{2(M-1)^2(1-z)\lambda_j}{M^2(1+z)^2}.\label{eq:silver-interior-update}
\end{equation}
In addition, we take
\begin{align}
&\lambda^{A\bowtie B}_{p-1,N-1;p-1}=K,\qquad \lambda^{A\bowtie B}_{N-1,p-1;p-1}=MyK,\qquad\lambda^{A\bowtie B}_{N-1,N-1;p-1}=H, \notag \\
&\lambda^{A\bowtie B}_{\star,p-1;0} =\tfrac{2(M-1)e(y)(1-v)(1-z)}{M(1+z)^2},\qquad 
\lambda^{A\bowtie B}_{\star,N-1;0} =\tfrac{2(M-1)[1+(M-1)z+Mz^2]}{M^2(1+z)^2}, \label{eq:silver-boundary-update} \\
&\text{and} \quad\lambda^{A\bowtie B}_{i\star;0} =
\begin{cases}
0,&\text{if }i\in[N-2],\\[3pt]
\frac{4(M-1)z}{M(1+z)^2},&\text{if }i=N-1.
\end{cases} \notag
\end{align}
All remaining $\lambda$ coefficients are considered to be zero. The assignments~\eqref{eq:silver-inherited-coefficients}--\eqref{eq:silver-boundary-update} specify every coefficient in~\eqref{eq:silver-interpolation-sum} and hence define $\mathcal W_{A\bowtie B}$. For $\mathcal R_{A\bowtie B}$ and $\{\mu_k^{A\bowtie B}\}_{k=0}^{N-1}$ in~\eqref{eq:silver-expected-certificate} regarding the joint block $A\bowtie B$, we also consider
\begin{align}
&\mathcal R_{A\bowtie B} =s_A\mathcal R_A+s_B\mathcal R_B+\tfrac{H}{M-1}\E[\|\nabla f(x_{N-1})-Mx_{N-1}-\E[\nabla f(x_{N-1})-Mx_{N-1}\mid\mathcal H_{p-1}]\|^2] \notag \\ 
&\text{and} \ \ \mu_k^{A\bowtie B}=
\begin{cases}
s_A\mu_k^A,&\text{if }0\le k<p-1,\\
s_B\tau(v)e(y)^2,&\text{if }k=p-1,\\
s_B\mu_{k-p}^B,&\text{if }p\le k<N-1,\\
e(z)^2,&\text{if }k=N-1.
\end{cases} \label{eq:silver-joined-noise-coefficients}
\end{align}
We now verify that these quantities satisfy the required certificate, including nonnegativity constraints,~\eqref{eq:silver-star-interface}, and~\eqref{eq:silver-expected-certificate}.

\noindent \textbf{Nonnegativity.}
Definition~\ref{def:silver-join} guarantees $0<y<\min\{u,v\}\le\max\{u,v\}<z<1$. Then~\eqref{eq:silver-inheritance-scales} and~\eqref{eq:silver-interface-coefficients} further imply $\{s_A,s_B,K,H\}\subset\R_{>0}$.
From the seed block $(\bar\alpha)$, each join introduces only $e(y)$ and $e(z)$ by~\eqref{eq:silver-join-definition}. Since $e(t)\in(0,1)$ for $t\in(0,1)$ by~\eqref{eq:composition-terminal-parameters} and~\eqref{eq:composition-terminal-factors}, every step satisfies $\alpha_k\in(0,1)$. Hence~\eqref{eq:silver-interface-coefficients} gives $\lambda_j>0$ for $p\in[N-2]\backslash[p-1]$.
Moreover, coefficients in~\eqref{eq:silver-inherited-coefficients} are nonnegative by induction, and each coefficient introduced in~\eqref{eq:silver-interior-update} and~\eqref{eq:silver-boundary-update} is also nonnegative. Since every $\mathcal Q_{ij;\ell}$ in~\eqref{eq:composition-coupled-Q} is nonnegative, \eqref{eq:silver-interpolation-sum} ensures $\mathcal W_{A\bowtie B}\ge0$. In~\eqref{eq:silver-joined-noise-coefficients}, $(s_A\mathcal R_A,s_B\mathcal R_B,H)$ are all nonnegative, leading to $\mathcal R_{A\bowtie B}\ge0$, and each $\mu_k^{A\bowtie B}$ is also nonnegative.

\noindent \textbf{Verifying the coefficients in~\eqref{eq:silver-star-interface} used in the recursion.}
For the seed block $(\bar\alpha)$, $N=1$, $t=\frac{1}{M}$, and $1-\bar\alpha=-(1-M\bar\alpha)=\frac{M-1}{M+1}$ by~\eqref{eq:one-two-notation}. Consider two silver blocks A and B with terminal parameters $u\in(0,1)$ and $v\in(0,1)$. Then the joint block $A\bowtie B$ is with the terminal parameter $z\in(0,1)$ and intermediate parameter $y\in(0,1)$ defined by~\eqref{eq:silver-join-parameters}. It follows~\eqref{eq:composition-terminal-factors} that
\begin{align*}
\prod_{k=0}^{N-1}(1-\alpha_k) &=\tfrac{1-u}{1+u}\cdot\tfrac{1-v}{1+v}\cdot\tfrac{1-e(y)}{1-e(u)}\cdot\tfrac{1-e(z)}{1-e(v)}=\tfrac{(1-u)(1-v)}{(1+y)(1+z)} =\tfrac{1-z}{1+z}\\
\text{and}\quad\prod_{k=0}^{N-1}(1-M\alpha_k) &=(-1)^{p+q}\cdot\tfrac{1-u}{1+u}\cdot\tfrac{1-v}{1+v}\cdot\tfrac{1-Me(y)}{1-Me(u)}\cdot\tfrac{1-Me(z)}{1-Me(v)}=(-1)^N\cdot\tfrac{yz}{uv}\cdot\tfrac{(1-u)(1-v)}{(1+y)(1+z)} =(-1)^N\cdot\tfrac{1-z}{1+z},
\end{align*}
where the last equalities use $(1-u)(1-v)=(1-z)(1+y)$ and $yz=uv$ from~\eqref{eq:composition-kernel-identities}. Therefore, by induction, stepsizes satisfy
\begin{equation}
\prod_{k=0}^{N-1}(1-\alpha_k)=\tfrac{1-t}{1+t} \quad \text{and} \quad \prod_{k=0}^{N-1}(1-M\alpha_k)=(-1)^N\cdot\tfrac{1-t}{1+t}
\label{eq:silver-endpoint-invariants}
\end{equation} 
for any \texttt{silver} block of length $N$ and terminal parameter $t\in(0,1)$. 

We then verify~\eqref{eq:silver-star-interface} for the joined block $A\bowtie B$ of length $N$ and terminal parameter $z\in(0,1)$. A direct comparison shows that $\{\lambda^{A\bowtie B}_{i\star;0}\}_{i\in [N-1]}$ and $\lambda^{A\bowtie B}_{\star,N-1;0}$ in~\eqref{eq:silver-boundary-update} agree with their corresponding terms in~\eqref{eq:silver-star-interface}, but potentially with updated terminal parameters. For $j\in[N-2]\backslash[p-1]$, the coefficient $\lambda^{A\bowtie B}_{\star j;0}$ in~\eqref{eq:silver-interior-update} also agrees with~\eqref{eq:silver-star-interface} following $\alpha_j\prod_{r=j+1}^{N-2}(1-\alpha_r)=\lambda_j$ in~\eqref{eq:silver-interface-coefficients}.
Moreover, the coefficient $\lambda^{A\bowtie B}_{\star,p-1;0}$ in~\eqref{eq:silver-boundary-update} agrees with~\eqref{eq:silver-star-interface} because $\alpha_{p-1}=e(y)$ and $ \prod_{r=p}^{N-2}(1-\alpha_r) =\frac{1-v}{(1+v)(1-e(v))} =\frac{M(1-v)}{M-1}$ by~\eqref{eq:silver-endpoint-invariants} and~\eqref{eq:composition-terminal-factors}.
Finally, for $i\in[p-2]$, we have from~\eqref{eq:silver-inherited-coefficients} that $\lambda^{A\bowtie B}_{\star i;0}=s_A\lambda^A_{\star i;0}$. By~\eqref{eq:silver-inheritance-scales} and~\eqref{eq:composition-kernel-identities},
\begin{equation}
\begin{aligned}
s_A\cdot\tfrac{1-u}{(1+u)^2} &=\tfrac{(1-z)^2}{(1+z)^2(1-u)} =\tfrac{1-z}{(1+z)^2}\cdot\tfrac{1-v}{1+y}=\tfrac{1-z}{(1+z)^2}\cdot(1-e(y))\cdot\prod_{r=p}^{N-2}(1-\alpha_r) =\tfrac{1-z}{(1+z)^2}\cdot\prod_{r=p-1}^{N-2}(1-\alpha_r),
\end{aligned} \label{eq:silver-left-star-identity}
\end{equation}
where the third equality uses $1-e(y)=\frac{M-1}{M(1+y)}$ from~\eqref{eq:composition-terminal-factors} and $\prod_{r=p}^{N-2}(1-\alpha_r)=\frac{M(1-v)}{M-1}$ from the preceding calculation, and the last equality is from $\alpha_{p-1}=e(y)$.
Substituting the formula for $\lambda^A_{\star i;0}$ from~\eqref{eq:silver-star-interface} and using~\eqref{eq:silver-left-star-identity} results in $\lambda^{A\bowtie B}_{\star i;0} =s_A\lambda^A_{\star i;0}=\frac{2(M-1)^2}{M^2}\cdot\left(s_A\cdot\frac{1-u}{(1+u)^2}\right) \cdot a_i\cdot\prod_{r=i+1}^{p-2}(1-a_r)
= \frac{2(M-1)^2(1-z)a_i}{M^2(1+z)^2}\cdot\prod_{r=i+1}^{p-2}(1-a_r)\cdot \prod_{r=p-1}^{N-2}(1-\alpha_r)=\frac{2(M-1)^2(1-z)\alpha_i}{M^2(1+z)^2} \cdot\prod_{r=i+1}^{N-2}(1-\alpha_r)$,
where the last equality uses $\alpha_r=a_r$ for $r\in [p-2]$. This is exactly the value of $\lambda^{A\bowtie B}_{\star i;0}$ in~\eqref{eq:silver-star-interface} for a block of length $N$ and terminal parameter $z$.

\noindent \textbf{Verifying the joined certificate identity.}
By~\eqref{eq:composition-terminal-factors},~\eqref{eq:composition-kernel-identities}, and~\eqref{eq:silver-inheritance-scales}, 
\begin{equation}
s_B(1-Me(v))^2=(1-Me(z))^2\quad \text{and} \quad
s_A(1-Me(u))^2=s_B\tau(v)(1-Me(y))^2.
\label{eq:silver-noise-scaling}
\end{equation}
By~\eqref{eq:silver-endpoint-invariants} and~\eqref{eq:silver-noise-scaling}, $\prod_{r=p}^{N-1}(1-M\alpha_r)^2=s_B\tau(v)$. For the joined block $A\bowtie B$, by~\eqref{eq:silver-expected-certificate}, $\mu_{N-1}^{A\bowtie B} = \alpha_{N-1}^2$ and $\mu_k^{A\bowtie B} = \alpha_k^2\prod_{r=k+1}^{N-1}(1-M\alpha_r)^2$ for every $k\in[N-2]$, and then we further compare it directly with~\eqref{eq:silver-joined-noise-coefficients}.
When $k\in[p-2]$, we have $\alpha_r=a_r$ for $r\in[p-2]\backslash[k-1]$ and $\alpha_{p-1}=e(y)$. Using $\prod_{r=p}^{N-1}(1-M\alpha_r)^2=s_B\tau(v)$ and $a_{p-1}=e(u)$, we have $\alpha_k^2\cdot\prod_{r=k+1}^{N-1}(1-M\alpha_r)^2=a_k^2\cdot\left(\prod_{r=k+1}^{p-2}(1-Ma_r)^2\right)\cdot(1-Me(y))^2s_B\tau(v)=\frac{(1-Me(y))^2s_B\tau(v)}{(1-Me(u))^2}\cdot\mu_k^A=s_A\mu_k^A$.
When $k\in[N-2]\backslash[p-1]$, we have $\alpha_r=b_{r-p}$ for $r\in[N-2]\backslash[k-1]$ and $\alpha_{N-1}=e(z)$. Using $b_{q-1}=e(v)$ gives $\alpha_k^2\cdot\prod_{r=k+1}^{N-1}(1-M\alpha_r)^2
=b_{k-p}^2\cdot\left(\prod_{r=k-p+1}^{q-2}(1-Mb_r)^2\right)\cdot(1-Me(z))^2=\frac{(1-Me(z))^2}{(1-Me(v))^2}\cdot\mu_{k-p}^B=s_B\mu_{k-p}^B$, where the last equality in each calculation follows from~\eqref{eq:silver-noise-scaling}.
When $k=p-1$ or $k=N-1$,~\eqref{eq:silver-expected-certificate} gives $\mu_{p-1}^{A\bowtie B} = e(y)^2s_B\tau(v)$ and $\mu_{N-1}^{A\bowtie B} = e(z)^2$. 
Combining cases discussed above, we know that the coefficients assigned in~\eqref{eq:silver-joined-noise-coefficients} agree with $\{\mu_k^{A\bowtie B}\}_{k\in[N-1]}$ defined in~\eqref{eq:silver-expected-certificate}.

Given $\mathcal H_p$, let's apply~\eqref{eq:silver-expected-certificate} to $B$ under the conditional law  in~\eqref{eq:silver-continuation-kernel}, then $x_p$ is fixed and Assumption~\ref{ass:oracle} still holds. The last iterate for block $B$ is $x_{N-1}-e(v)g_{N-1}$ by~\eqref{eq:sg-update}. We define the interpolation terms for $B$ given $\mathcal H_p=\boldsymbol{\sigma}(Z_p)$ using~\eqref{eq:composition-coupled-Q} such that for every $(i,j)\in[q-1]$ and $\ell\in[\min\{i,j\}]$, 
\begin{equation}\label{eq.QB_definition}
\begin{aligned}
\mathcal Q^B_{ij;\ell}(Z_p)&=\E[Q(x_{p+i;p+\ell},x'_{p+j;p+\ell})\mid\mathcal H_p], \\
\mathcal Q^B_{i\star;0}(Z_p)&=\E[Q(x_{p+i},0)\mid\mathcal H_p], \\
\text{and} \quad \mathcal Q^B_{\star i;0}(Z_p)&=\E[Q(0,x_{p+i})\mid\mathcal H_p].
\end{aligned}
\end{equation}
By~\eqref{eq:silver-interpolation-sum}, 
\begin{equation}\label{eq:silver-WforB}
   \mathcal W_B(Z_p) =\left(\sum_{(i,j)\in[q-1]^2}\sum_{\ell=0}^{\min\{i,j\}}\lambda^B_{ij;\ell}\mathcal Q^B_{ij;\ell}(Z_p)\right)+\sum_{i=0}^{q-1}\bigl(\lambda^B_{i\star;0}\mathcal Q^B_{i\star;0}(Z_p)+\lambda^B_{\star i;0}\mathcal Q^B_{\star i;0}(Z_p)\bigr) 
\end{equation}
With these definitions, by~\eqref{eq:silver-expected-certificate}, the certificate for $B$ conditional on $\mathcal H_p$ is $\tau(v)\E[\|x_p\|^2\mid\mathcal H_p]+\sum_{k=0}^{q-1}\mu_k^B\E[\|\epsilon_{p+k}\|^2\mid\mathcal H_p]-\E[\|x_{N-1}-e(v)g_{N-1}\|^2\mid\mathcal H_p]
=\mathcal W_B(Z_p)+\mathcal R_B(Z_p)$.
After taking $\E[\cdot]=\E[\cdot\mid\mathcal H_0]$ on both sides, by the tower property, we have
\begin{align}
\tau(v)\E[\|x_p\|^2]+\left(\sum_{k=0}^{q-1}\mu_k^B\E[\|\epsilon_{p+k}\|^2]\right)-\E[\|x_{N-1}-e(v)g_{N-1}\|^2] &= \E[\mathcal W_B(Z_p)]+\E[\mathcal R_B(Z_p)] =\mathcal W_B+\mathcal R_B, \label{eq.WBRB}
\end{align}
where $\mathcal W_B:=\E[\mathcal W_B(Z_p)]$ and $\mathcal R_B:=\E[\mathcal R_B(Z_p)]$. Taking expectation $\mathbb{E}[\cdot]$ on both sides of~\eqref{eq:silver-WforB}, it follows~\eqref{eq:silver-continuation-kernel},~\eqref{eq:composition-coupled-Q}, and~\eqref{eq.QB_definition} that $\E[\mathcal Q^B_{ij;\ell}(Z_p)]=\mathcal Q_{p+i,p+j;p+\ell}$, $\E[\mathcal Q^B_{i\star;0}(Z_p)]=\mathcal Q_{p+i,\star;0}$, and $\E[\mathcal Q^B_{\star i;0}(Z_p)]=\mathcal Q_{\star,p+i;0}$.
It further follows from~\eqref{eq:silver-interpolation-sum} that $\mathcal W_B =(\sum_{(i,j)\in[q-1]^2}\sum_{\ell=0}^{\min\{i,j\}}\lambda^B_{ij;\ell}\mathcal Q_{p+i,p+j;p+\ell})+\sum_{i=0}^{q-1}\bigl(\lambda^B_{i\star;0}\mathcal Q_{p+i,\star;0}+\lambda^B_{\star i;0}\mathcal Q_{\star,p+i;0}\bigr)$.

Applying~\eqref{eq:silver-expected-certificate} to $A$ with oracle outputs $\{g_k\}_{k\in[p-1]}$ and the terminal iterate $x_{p-1}-e(u)g_{p-1}$, by~\eqref{eq.WBRB}, we know that
\begin{equation}
\begin{aligned}
&s_A\left(\tau(u)\|x_0\|^2+\left(\sum_{k\in[p-1]}\mu_k^A\E[\|\epsilon_k\|^2]\right)-\E[\|x_{p-1}-e(u)g_{p-1}\|^2]\right) \\
&\quad+s_B\left(\tau(v)\E[\|x_p\|^2]+\left(\sum_{k\in[q-1]}\mu_k^B\E[\|\epsilon_{p+k}\|^2]\right)-\E[\|x_{N-1}-e(v)g_{N-1}\|^2]\right) \\
&=s_A\mathcal W_A+s_B\mathcal W_B+s_A\mathcal R_A+s_B\mathcal R_B.
\end{aligned}
\label{eq:silver-sum-child-certificates}
\end{equation}
Applying~\eqref{eq:silver-expected-certificate} to $A\bowtie B$, it remains to verify the following certificate identity:
\begin{equation}
\tau(z)\|x_0\|^2+\left(\sum_{k\in[N-1]}\mu_k^{A\bowtie B}\E[\|\epsilon_k\|^2]\right)-\E[\|x_N\|^2]
=\mathcal W_{A\bowtie B}+\mathcal R_{A\bowtie B}.
\label{eq:silver-joined-certificate-target}
\end{equation}
Subtracting~\eqref{eq:silver-sum-child-certificates} from~\eqref{eq:silver-joined-certificate-target} side by side, we obtain
\begin{equation}
\begin{aligned}
\mathcal E =\mathcal W_{A\bowtie B}-s_A\mathcal W_A-s_B\mathcal W_B+\mathcal R_{A\bowtie B}-s_A\mathcal R_A-s_B\mathcal R_B,
\end{aligned}
\label{eq:silver-local-certificate-identity}
\end{equation}
where
\begin{equation}\label{eq.Error_E}
\begin{aligned}
\mathcal E &:=\tau(z)\|x_0\|^2+\left(\sum_{k\in[N-1]}\mu_k^{A\bowtie B}\E[\|\epsilon_k\|^2]\right)-\E[\|x_N\|^2] \\
&\qquad-s_A\left[\tau(u)\|x_0\|^2+\left(\sum_{k\in[p-1]}\mu_k^A\E[\|\epsilon_k\|^2]\right)-\E[\|x_{p-1}-e(u)g_{p-1}\|^2]\right]\\
&\qquad-s_B\left[\tau(v)\E[\|x_p\|^2]+\left(\sum_{k\in[q-1]}\mu_k^B\E[\|\epsilon_{p+k}\|^2]\right)-\E[\|x_{N-1}-e(v)g_{N-1}\|^2]\right].
\end{aligned}
\end{equation}
Since~\eqref{eq:silver-sum-child-certificates} holds, establishing~\eqref{eq:silver-joined-certificate-target} is equivalent to proving~\eqref{eq:silver-local-certificate-identity}. We then verify~\eqref{eq:silver-local-certificate-identity} by comparing its two sides. 

On the left-hand side of~\eqref{eq:silver-local-certificate-identity} (see~\eqref{eq.Error_E}), the terms in $\|x_0\|^2$ cancel because $\tau(z)=s_A\tau(u)$ by~\eqref{eq:silver-inheritance-scales}. By~\eqref{eq:silver-joined-noise-coefficients}, the coefficients of $\E[\|\epsilon_k\|^2]$ in the noise sums in~\eqref{eq.Error_E} agree except for $k\in\{p-1,N-1\}$. Using $x_p=x_{p-1}-e(y)g_{p-1}$ and $x_N=x_{N-1}-e(z)g_{N-1}$ from~\eqref{eq:sg-update}, the squared-norm terms in~\eqref{eq.Error_E} at $k=p-1$ and $k=N-1$ are, respectively,
\begin{align*}
&s_A\E[\|x_{p-1}-e(u)g_{p-1}\|^2]-s_B\tau(v)\E[\|x_{p-1}-e(y)g_{p-1}\|^2],\\
\text{and}\quad &s_B\E[\|x_{N-1}-e(v)g_{N-1}\|^2]-\E[\|x_{N-1}-e(z)g_{N-1}\|^2],
\end{align*}
where the term with coefficient $-s_B\tau(v)$ comes from subtracting $s_B\tau(v)\E[\|x_p\|^2]$ in~\eqref{eq:silver-sum-child-certificates}. In addition, by the conditional unbiasedness in~\eqref{eq:unbiasedness}, for $k\in\{p-1,N-1\}$, $\E[\|x_k-e(t)g_k\|^2]=\E[\|x_k-e(t)\nabla f(x_k)\|^2]+e(t)^2\E[\|\epsilon_k\|^2]$.
The noise terms from these expansions cancel the differences for $k\in\{p-1,N-1\}$. Set $r(x)=\nabla f(x)-Mx$. Using $x-e(t)\nabla f(x)=(1-Me(t))x-e(t)r(x)$ and the cancellation of the $\E[\|x_{p-1}\|^2]$ and $\E[\|x_{N-1}\|^2]$ terms by~\eqref{eq:silver-noise-scaling}, we obtain
\begin{align}
\mathcal E&=\E\bigl[s_A\|x_{p-1}-e(u)\nabla f(x_{p-1})\|^2-s_B\tau(v)\|x_{p-1}-e(y)\nabla f(x_{p-1})\|^2 \notag \\
&\qquad\qquad+s_B\|x_{N-1}-e(v)\nabla f(x_{N-1})\|^2-\|x_{N-1}-e(z)\nabla f(x_{N-1})\|^2\bigr] \notag \\
&=\bigl(s_Ae(u)^2-s_B\tau(v)e(y)^2\bigr)\E[\|r(x_{p-1})\|^2] \label{eq:silver-local-boundary-difference} \\
&\qquad+\bigl(2s_Ae(u)(Me(u)-1)-2s_B\tau(v)e(y)(Me(y)-1)\bigr)\E[\langle r(x_{p-1}),x_{p-1}\rangle] \notag \\
&\qquad+\bigl(s_Be(v)^2-e(z)^2\bigr)\E[\|r(x_{N-1})\|^2] \notag \\
&\qquad+\bigl(2s_Be(v)(Me(v)-1)-2e(z)(Me(z)-1)\bigr)\E[\langle r(x_{N-1}),x_{N-1}\rangle]. \notag
\end{align}
On the right-hand side of~\eqref{eq:silver-local-certificate-identity}, \eqref{eq:silver-joined-noise-coefficients} gives $\mathcal R_{A\bowtie B}-s_A\mathcal R_A-s_B\mathcal R_B
=\frac{H}{M-1}\E[\bigl\|r(x_{N-1})-\E[r(x_{N-1})\mid\mathcal H_{p-1}]\bigr\|^2]$.
By~\eqref{eq:silver-interpolation-sum}, each $\mathcal W$ is a sum of terms $\lambda_{ij;\ell}\mathcal Q_{ij;\ell}$. To compute $\mathcal W_{A\bowtie B}-s_A\mathcal W_A-s_B\mathcal W_B$, we compare the coefficients of the same $\mathcal Q_{ij;\ell}$ terms. Comparing $\{\mathcal Q_{ij;\ell}\}_{(i,j)\in[p-1]^2}$, $\{\mathcal Q_{p+i,p+j;p+\ell}\}_{(i,j)\in[q-1]^2}$, and $\{\mathcal Q_{\star i;0}\}_{i\in[p-1]\backslash\{p-1\}}$, it follows~\eqref{eq:silver-inherited-coefficients} that
\begin{align*}
\lambda^{A\bowtie B}_{ij;\ell}-s_A\lambda^A_{ij;\ell}&=0,
&\forall (i,j)\in[p-1]^2 \quad \text{and} \quad \forall\ell\in[\min\{i,j\}],\\
\lambda^{A\bowtie B}_{p+i,p+j;p+\ell}-s_B\lambda^B_{ij;\ell}&=0,
&\forall (i,j)\in[q-1]^2 \quad \text{and} \quad \forall\ell\in[\min\{i,j\}],\\
\text{and}\qquad \lambda^{A\bowtie B}_{\star i;0}-s_A\lambda^A_{\star i;0}&=0,
&\forall i\in[p-1]\backslash\{p-1\}.
\end{align*}
These terms cancel out in $\mathcal W_{A\bowtie B}-s_A\mathcal W_A-s_B\mathcal W_B$. Expanding the definition of $Q(\cdot,\cdot)$ in~\eqref{eq:two-step-interpolation} and using $\nabla f(0)=0$ leads to, for every $\{x,y\}\subset\mathbb R^d$,
\begin{equation}
Q(x,y)-Q(0,y)-Q(x,0)=\tfrac{\langle r(x),\nabla f(y)-y\rangle}{M-1}. \label{eq:silver-Q-difference}
\end{equation}
Now we move to focus on the terms $\lambda^{A\bowtie B}_{ij;p-1}\mathcal Q_{ij;p-1}$ with $(i,j)\in \{p-1,N-1\}\times([N-2]\backslash[p-1]) \cup \{(p-1,N-1),(N-1,p-1),(N-1,N-1)\}$. By substituting $x=x_{i;p-1}$ and $y=x'_{j;p-1}$ into~\eqref{eq:silver-Q-difference}, from the shared history $\mathcal H_{p-1}$, $x_{i;p-1}$ and $x'_{j;p-1}$ are conditionally independent, so~\eqref{eq:silver-conditional-factorization} implies that $\E\bigl[\langle r(x_{i;p-1}),\nabla f(x'_{j;p-1})-x'_{j;p-1}\rangle\mid\mathcal H_{p-1}\bigr]=\bigl\langle\E[r(x_{i;p-1})\mid\mathcal H_{p-1}],\E[\nabla f(x'_{j;p-1})-x'_{j;p-1}\mid\mathcal H_{p-1}]\bigr\rangle$.
Taking the outer expectation by the tower property and using~\eqref{eq:silver-same-run-moments}--\eqref{eq:composition-coupled-Q} further indicates that
\begin{equation}
\mathcal Q_{ij;p-1}-\mathcal Q_{i\star;0}-\mathcal Q_{\star j;0}=\tfrac{1}{M-1}\E[\bigl\langle\E[r(x_{i;p-1})\mid\mathcal H_{p-1}],\E[\nabla f(x_{j;p-1})-x_{j;p-1}\mid\mathcal H_{p-1}]\bigr\rangle].
\label{eq:silver-coupled-Q-expansion}
\end{equation}
For every $j\in[N-2]\backslash[p-1]$, \eqref{eq:silver-star-interface} and~\eqref{eq:silver-interior-update} implies
\begin{equation}
\lambda^{A\bowtie B}_{\star j;0}-s_B\lambda^B_{\star,j-p;0} =\lambda_j\left(\tfrac{2(M-1)^2(1-z)}{M^2(1+z)^2}-s_B\tfrac{2(M-1)^2(1-v)}{M^2(1+v)^2}\right) =-(M-1)(K+H)\lambda_j, \label{eq:silver-positive-star-identity}
\end{equation}
where the last equality uses~\eqref{eq:silver-inheritance-scales}, \eqref{eq:silver-interface-coefficients}, and $(v,y)=(\frac{z(z-u)}{u+z-2uz},\frac{u(z-u)}{u+z-2uz})$ from~\eqref{eq:composition-kernel-identities}. Combining~\eqref{eq:silver-interior-update} and~\eqref{eq:silver-positive-star-identity}, we know that the quantity $(M-1)\sum_{j\in[N-2]\backslash[p-1]}\lambda_j(K(\mathcal Q_{p-1,j;p-1}-\mathcal Q_{\star j;0})+H(\mathcal Q_{N-1,j;p-1}-\mathcal Q_{\star j;0}))$ exactly appears in $\mathcal W_{A\bowtie B}-s_A\mathcal W_A-s_B\mathcal W_B$.
For $j\in[N-2]\backslash[p-1]$, \eqref{eq:silver-interface-coefficients} gives $\lambda_j=\left(\prod_{r\in[N-2]\backslash[j]}(1-\alpha_r)\right)-\left(\prod_{r\in[N-2]\backslash[j-1]}(1-\alpha_r)\right)$. Moreover, by~\eqref{eq:unbiasedness}--\eqref{eq:sg-update}, we have
\begin{align}
&\sum_{j\in[N-2]\backslash[p-1]}\lambda_j=1-\prod_{r\in[N-2]\backslash[p-1]}(1-\alpha_r)=\tfrac{Mv-1}{M-1} \qquad \text{and} \label{eq:silver-conditional-telescoping} \\
&\E\left[\sum_{j\in[N-2]\backslash[p-1]}\lambda_j(\nabla f(x_j)-x_j)\,\middle|\,\mathcal H_{p-1}\right]=\tfrac{M(1-v)(x_{p-1}-e(y)\nabla f(x_{p-1}))}{M-1}-\E[x_{N-1}\mid\mathcal H_{p-1}]. \notag
\end{align}
Here, $1-\prod_{r\in[N-2]\backslash[p-1]}(1-\alpha_r)=\frac{Mv-1}{M-1}$ because of~\eqref{eq:composition-terminal-factors} and~\eqref{eq:silver-endpoint-invariants}. When $q=1$,~\eqref{eq:silver-conditional-telescoping} vanishes because $v=\frac1M$ and $x_{N-1} = x_p$, then $\E[x_{N-1}\mid\mathcal H_{p-1}]=x_{p-1}-e(y)\nabla f(x_{p-1})$ by~\eqref{eq:unbiasedness}--\eqref{eq:sg-update}. Furthermore, it follows~\eqref{eq:silver-conditional-factorization},~\eqref{eq:silver-same-run-moments},~\eqref{eq:silver-coupled-Q-expansion} and~\eqref{eq:silver-conditional-telescoping} that $(M-1)\sum_{j\in[N-2]\backslash[p-1]}\lambda_j(K(\mathcal Q_{p-1,j;p-1}-\mathcal Q_{\star j;0})+H(\mathcal Q_{N-1,j;p-1}-\mathcal Q_{\star j;0}))= (Mv-1)\bigl(K\mathcal Q_{p-1,\star;0}+H\mathcal Q_{N-1,\star;0}\bigr)+\frac{M(1-v)}{M-1}\E[\bigl\langle Kr(x_{p-1})+H\E[r(x_{N-1})\mid\mathcal H_{p-1}],x_{p-1}-e(y)\nabla f(x_{p-1})\bigr\rangle]-\E[\bigl\langle Kr(x_{p-1})+H\E[r(x_{N-1})\mid\mathcal H_{p-1}],\E[x_{N-1}\mid\mathcal H_{p-1}]\bigr\rangle]$.
By~\eqref{eq:silver-conditional-factorization}, \eqref{eq:silver-same-run-moments} and~\eqref{eq:silver-coupled-Q-expansion},
\begin{align*}
\mathcal Q_{p-1,N-1;p-1}&=\mathcal Q_{p-1,\star;0}+\mathcal Q_{\star,N-1;0}+\tfrac{1}{M-1}\E[\bigl\langle r(x_{p-1}),\E[\nabla f(x_{N-1})-x_{N-1}\mid\mathcal H_{p-1}]\bigr\rangle],\\
\mathcal Q_{N-1,p-1;p-1}&=\mathcal Q_{N-1,\star;0}+\mathcal Q_{\star,p-1;0}+\tfrac{1}{M-1}\E[\bigl\langle\E[r(x_{N-1})\mid\mathcal H_{p-1}],r(x_{p-1})+(M-1)x_{p-1}\bigr\rangle],\\
\text{and} \ \ \mathcal Q_{N-1,N-1;p-1}&=\mathcal Q_{N-1,\star;0}+\mathcal Q_{\star,N-1;0}+\tfrac{1}{M-1}\E[\bigl\langle\E[r(x_{N-1})\mid\mathcal H_{p-1}],\E[\nabla f(x_{N-1})-x_{N-1}\mid\mathcal H_{p-1}]\bigr\rangle].
\end{align*}
Applying~\eqref{eq:silver-star-sum} to $\mathcal Q_{N-1,N-1;p-1}$, by $\nabla f(x_{N-1})-x_{N-1}=r(x_{N-1})+(M-1)x_{N-1}$ and the conditional second-moment identity $\E[\bigl\|r(x_{N-1})-\E[r(x_{N-1})\mid\mathcal H_{p-1}]\bigr\|^2]=\E[\|r(x_{N-1})\|^2]-\E[\bigl\|\E[r(x_{N-1})\mid\mathcal H_{p-1}]\bigr\|^2]$,
we obtain
\begin{equation}
\begin{aligned}
&H\mathcal Q_{N-1,N-1;p-1}+\tfrac{H}{M-1}\E[\bigl\|r(x_{N-1})-\E[r(x_{N-1})\mid\mathcal H_{p-1}]\bigr\|^2]\\
= \ &-H\E[\langle r(x_{N-1}),x_{N-1}\rangle]+H\E[\bigl\langle\E[r(x_{N-1})\mid\mathcal H_{p-1}],\E[x_{N-1}\mid\mathcal H_{p-1}]\bigr\rangle].
\end{aligned}
\label{eq:silver-copy-covariance}
\end{equation}
Substituting~\eqref{eq:silver-boundary-update} and~\eqref{eq:silver-copy-covariance} into the right-hand side of~\eqref{eq:silver-local-certificate-identity}, we get quantities $C_X(\mathcal Q_{p-1,\star;0}+\mathcal Q_{\star,p-1;0})+C_Y(\mathcal Q_{N-1,\star;0}+\mathcal Q_{\star,N-1;0})$, where $C_X:=MvK-s_A\lambda^A_{p-1,\star;0}
=MyK+\lambda^{A\bowtie B}_{\star,p-1;0}-s_A\lambda^A_{\star,p-1;0}$ and $C_Y:=K+\lambda^{A\bowtie B}_{\star,N-1;0}-s_B\lambda^B_{\star,q-1;0}=MyK+(Mv-1)H+\lambda^{A\bowtie B}_{N-1,\star;0}-s_B\lambda^B_{q-1,\star;0}$ by~\eqref{eq:composition-kernel-identities},~\eqref{eq:silver-inheritance-scales}, and~\eqref{eq:silver-star-interface}.
Using $\nabla f(x_{N-1})-x_{N-1}=r(x_{N-1})+(M-1)x_{N-1}$ in $\mathcal Q_{p-1,N-1;p-1}$, we also have $-\E[\bigl\langle Kr(x_{p-1})+H\E[r(x_{N-1})\mid\mathcal H_{p-1}],\E[x_{N-1}\mid\mathcal H_{p-1}]\bigr\rangle] +K\E[\bigl\langle r(x_{p-1}),\E[x_{N-1}\mid\mathcal H_{p-1}]\bigr\rangle] +H\E[\bigl\langle\E[r(x_{N-1})\mid\mathcal H_{p-1}],\E[x_{N-1}\mid\mathcal H_{p-1}]\bigr\rangle]=0$.
The remaining terms in $\mathcal Q_{p-1,N-1;p-1}-\mathcal Q_{p-1,\star;0}-\mathcal Q_{\star,N-1;0}$ vanish by $\frac{M(1-v)}{M-1}H[x_{p-1}-e(y)\nabla f(x_{p-1})] =-MyKx_{p-1}-\frac{(1+My)K}{M-1}r(x_{p-1})$, following~\eqref{eq:composition-terminal-factors},~\eqref{eq:composition-kernel-identities} and~\eqref{eq:silver-interface-coefficients}.

After these cancellations, the right-hand side of~\eqref{eq:silver-local-certificate-identity} consists of the two pairs of interpolation terms $C_X(\mathcal Q_{p-1,\star;0}+\mathcal Q_{\star,p-1;0})+C_Y(\mathcal Q_{N-1,\star;0}+\mathcal Q_{\star,N-1;0})$, the term $\frac{M(1-v)}{M-1}K\E\langle r(x_{p-1}),x_{p-1}-e(y)\nabla f(x_{p-1})\rangle$, and the term $-H\E\langle r(x_{N-1}),x_{N-1}\rangle$ from~\eqref{eq:silver-copy-covariance}. Applying~\eqref{eq:silver-star-sum} to the two pairs of interpolation terms and using $x_{p-1}-e(y)\nabla f(x_{p-1})=-(Me(y)-1)x_{p-1}-e(y)r(x_{p-1})$ gives
\begin{equation}
\begin{aligned}
&\mathcal W_{A\bowtie B}-s_A\mathcal W_A-s_B\mathcal W_B
 +\mathcal R_{A\bowtie B}-s_A\mathcal R_A-s_B\mathcal R_B\\
&\qquad=\left(-\tfrac{C_X}{M-1}-\tfrac{M(1-v)}{M-1}Ke(y)\right)\E[\|r(x_{p-1})\|^2] \\
&\qquad\quad+\left(-C_X-\tfrac{M(1-v)}{M-1}K(Me(y)-1)\right)\E[\langle r(x_{p-1}),x_{p-1}\rangle] \\
&\qquad\quad-\tfrac{C_Y}{M-1}\E[\|r(x_{N-1})\|^2]+(-C_Y-H)\E[\langle r(x_{N-1}),x_{N-1}\rangle].
\end{aligned}
\label{eq:silver-local-right-expansion}
\end{equation}

The coefficients in~\eqref{eq:silver-local-right-expansion} agree with those in~\eqref{eq:silver-local-boundary-difference}. In the order of the four terms in those equations, substituting $(C_X,C_Y)$ and using~\eqref{eq:composition-kernel-identities},~\eqref{eq:silver-inheritance-scales},~\eqref{eq:silver-star-interface}, and~\eqref{eq:silver-interface-coefficients}, we have
\begin{align*}
-\tfrac{C_X}{M-1}-\tfrac{M(1-v)}{M-1}Ke(y)
&=s_Ae(u)^2-s_B\tau(v)e(y)^2,\\
-C_X-\tfrac{M(1-v)}{M-1}K(Me(y)-1)
&=2s_Ae(u)(Me(u)-1)-2s_B\tau(v)e(y)(Me(y)-1),\\
-\tfrac{C_Y}{M-1}
&=s_Be(v)^2-e(z)^2,\\
\text{and} \qquad -C_Y-H
&=2s_Be(v)(Me(v)-1)-2e(z)(Me(z)-1),
\end{align*}
which proves~\eqref{eq:silver-local-certificate-identity}. Therefore,~\eqref{eq:silver-joined-certificate-target} holds and the induction completes.

\noindent \textbf{The resulting performance bound.}
By~\eqref{eq:unbiasedness},~\eqref{eq:silver-performance-bound}, and~\eqref{eq:silver-expected-certificate}, $\E[\|x_p\|^2]\le\tau(u)\|x_0\|^2+\sigma^2\sum_{k\in[p-1]}\mu_k^A=\tau(u)\|x_0\|^2+V_A\sigma^2$,
which leads to~\eqref{eq:composition-uniform-guarantee}. Finally, by~\eqref{eq:silver-inheritance-scales}, we have $\tau(z)=s_A\tau(u)$. Following~\eqref{eq:silver-joined-noise-coefficients} and summing the coefficients $\mu_k^{A\bowtie B}$ over $k\in[N-1]$ yields
$V_{A\bowtie B}=s_A(V_A-e(u)^2)+s_B(V_B-e(v)^2)+s_B\tau(v)e(y)^2+e(z)^2$, which proves~\eqref{eq:silver-bound-transform}.
\end{proof}

\section{Proof of Theorem~\ref{thm:arbitrary-horizon-convergence}}
\label{app:arbitrary-horizon-bound}
\begin{proof}[Proof of Theorem~\ref{thm:arbitrary-horizon-convergence}]
Fix $k\in[N]$, and let $\boldsymbol{\alpha}$ be the schedule in~\eqref{eq:arbitrary-horizon-steps} with $p=k$. By using $q_k$ and $\widehat z_k$ from~\eqref{eq:explicit-serial-cap-count} with $p=k$, for every admissible instance in~\eqref{eq:worst-case-error}, Theorem~\ref{thm:silver-bound} and~\eqref{eq:balanced-silver-coefficients} guarantee 
\begin{equation}\label{eq.thm5_proof_ub}
\E[\|x_k-x_\star\|^2]\le\tau_kR^2+V_k\sigma^2=\sigma^2z_k.
\end{equation}
For $k=0$, the same inequality follows from $\tau_0=1$, $V_0=0$, and $\|x_0-x_\star\|\le R$.
By~\eqref{eq:explicit-serial-cap-count}, we have $q_k\in[N-k]$ and $\widehat z_k=\rho^{q_k}z_k+\frac{1-\rho^{q_k}}{M}\ge0$.
If $q_k<N-k$, \eqref{eq:explicit-serial-cap-count} leads to $\rho^{-q_k}\ge\max\left\{1,\frac{z_k-\frac{1}{M}}{\gamma_1-\frac{1}{M}}\right\}\ge\frac{z_k-\frac{1}{M}}{\gamma_1-\frac{1}{M}}$.
Since $\gamma_1>\frac{1}{M}$ and $\rho^{q_k}>0$, we further know that $0\le\widehat z_k=\frac{1}{M}+\rho^{q_k}(z_k-\frac{1}{M})\le\gamma_1$.
For any $j\in[N-k-q_k]\backslash\{0\}$, it follows~\eqref{eq:one-two-notation} and~\eqref{eq:arbitrary-horizon-steps} that $0\le\alpha_{k+q_k+j-1}=\frac{\widehat z_k}{1+j\widehat z_k}\le\frac{\widehat z_k}{1+\widehat z_k}\le\frac{\gamma_1}{1+\gamma_1}=\bar\alpha$. By the help of~\eqref{eq:arbitrary-horizon-steps}, we could conclude that for every $i\in[N-1]\backslash[k-1]$, $0\le\alpha_i\le\bar\alpha$ and Proposition~\ref{prop:exact} gives the one-step coefficients $(\tau_1(\alpha_i),V_1(\alpha_i)) = ((1-\alpha_i)^2,\alpha_i^2)$.

It follows~\eqref{eq:one-two-notation} and~\eqref{eq:arbitrary-horizon-rho} that $\bar\alpha^2=\frac{1-\rho}{M}$. Applying~\eqref{eq.thm5_proof_ub} and Theorem~\ref{thm:serial-bound} to the $q_k$ steps with $\alpha_i=\bar\alpha$ in~\eqref{eq:arbitrary-horizon-steps} gives the coefficients
\begin{equation}
(\tau_{k+q_k},V_{k+q_k}) = (\rho^{q_k}\tau_k,\ \rho^{q_k}V_k+\tfrac{1-\rho^{q_k}}{M}).
\label{eq:arbitrary-horizon-capped-coefficients}
\end{equation}
When $q_k=0$, \eqref{eq:arbitrary-horizon-capped-coefficients} equals $(\tau_k,V_k)$.
For any $j\in[N-k-q_k]$, it follows~\eqref{eq:arbitrary-horizon-steps} that
\begin{equation}
\begin{aligned}
\prod_{\ell\in[N-k-q_k]\backslash[j]}(1-\alpha_{k+q_k+\ell-1})^2
&=\prod_{\ell\in[N-k-q_k]\backslash[j]}\left(\tfrac{1+(\ell-1)\widehat z_k}{1+\ell\widehat z_k}\right)^2=\tfrac{(1+j\widehat z_k)^2}{\left(1+(N-k-q_k)\widehat z_k\right)^2}.
\end{aligned}
\label{eq:arbitrary-horizon-decreasing-product}
\end{equation}
For any $j\in[N-k-q_k]\backslash\{0\}$, substituting $\alpha_{k+q_k+j-1}=\frac{\widehat z_k}{1+j\widehat z_k}$ into~\eqref{eq:arbitrary-horizon-decreasing-product} yields
\begin{equation}
\begin{aligned}
\alpha_{k+q_k+j-1}^2\prod_{\ell\in[N-k-q_k]\backslash[j]}(1-\alpha_{k+q_k+\ell-1})^2=\tfrac{\widehat z_k^2}{(1+j\widehat z_k)^2}\cdot\tfrac{(1+j\widehat z_k)^2}{\left(1+(N-k-q_k)\widehat z_k\right)^2}
=\tfrac{\widehat z_k^2}{\left(1+(N-k-q_k)\widehat z_k\right)^2}.
\end{aligned}
\label{eq:arbitrary-horizon-decreasing-noise}
\end{equation}
Applying Theorem~\ref{thm:serial-bound} together with~\eqref{eq:arbitrary-horizon-capped-coefficients}, \eqref{eq:arbitrary-horizon-decreasing-product} with $j=0$, and~\eqref{eq:arbitrary-horizon-decreasing-noise} leads to
\begin{equation}
\begin{aligned}
\E[\|x_N-x_\star\|^2]
\le\tfrac{\rho^{q_k}\tau_k}{\left(1+(N-k-q_k)\widehat z_k\right)^2}\cdot R^2+\tfrac{\rho^{q_k}V_k+\frac{1-\rho^{q_k}}{M}+(N-k-q_k)\widehat z_k^2}{\left(1+(N-k-q_k)\widehat z_k\right)^2}\cdot\sigma^2.
\end{aligned}
\label{eq:arbitrary-horizon-terminal-bound}
\end{equation}

By~\eqref{eq:one-two-notation},~\eqref{eq:one-step-optimal-value} and~\eqref{eq:arbitrary-horizon-rho}, we have
\begin{equation}
\mathcal T(z)-\tfrac{1}{M}=\rho(z-\tfrac{1}{M})\quad\forall z\ge\gamma_1\quad \text{and} \quad \tfrac{1}{\mathcal T(z)}=\tfrac{1}{z}+1\quad\forall z\in(0,\gamma_1].
\label{eq:explicit-serial-budget-identities}
\end{equation}
For any $j\in[q_k-1]$, \eqref{eq:explicit-serial-cap-count} leads to $\rho^{-j}<\frac{z_k-\frac{1}{M}}{\gamma_1-\frac{1}{M}}$ and $\frac{1}{M}+\rho^j\left(z_k-\frac{1}{M}\right)>\gamma_1$.
Applying~\eqref{eq:explicit-serial-budget-identities} therefore results in $\mathcal T^{\circ q_k}(z_k)=\frac{1}{M}+\rho^{q_k}\left(z_k-\frac{1}{M}\right)=\widehat z_k$.
If $q_k\in[N-k-1]$, then $0\le\widehat z_k\le\gamma_1$, so~\eqref{eq:explicit-serial-budget-identities} and $\mathcal T(0)=0$ imply that $\mathcal T^{\circ j}(\widehat z_k)=\frac{\widehat z_k}{1+j\widehat z_k}$ for every $j\in[N-k-q_k]$.
If $q_k=N-k$, then $\mathcal T^{\circ(N-k-q_k)}(\widehat z_k)=\widehat z_k$. Consequently, by~\eqref{eq:serial-map-iteration} and~\eqref{eq:prefix-budget-values}, we have
\begin{equation}
W_k(\gamma)=\mathcal T^{\circ(N-k-q_k)}\left(\mathcal T^{\circ q_k}(z_k)\right)=\tfrac{\widehat z_k}{1+(N-k-q_k)\widehat z_k}.
\label{eq:arbitrary-horizon-budget-formula}
\end{equation}
Substituting $R^2=\sigma^2\gamma$ and $\widehat z_k=\rho^{q_k}(\tau_k\gamma+V_k)+\frac{1-\rho^{q_k}}{M}$ into~\eqref{eq:arbitrary-horizon-terminal-bound} and using~\eqref{eq:arbitrary-horizon-budget-formula}, we obtain
\begin{equation}
\tfrac{\E[\|x_N-x_\star\|^2]}{\sigma^2}\le\tfrac{\widehat z_k+(N-k-q_k)\widehat z_k^2}{\left(1+(N-k-q_k)\widehat z_k\right)^2}=\tfrac{\widehat z_k}{1+(N-k-q_k)\widehat z_k}=W_k(\gamma).
\label{eq:prefix-budget-iteration}
\end{equation}
Taking the supremum over all admissible instances in~\eqref{eq:worst-case-error} shows $\D_N(\boldsymbol{\alpha})\le\sigma^2W_k(\gamma)$, proving~\eqref{eq:prefix-candidate-performance}.

Finally, we take $k=p$ with $p$ selected by~\eqref{eq:prefix-length-selection}. Substituting $k=p$ into~\eqref{eq:arbitrary-horizon-terminal-bound} and using~\eqref{eq:arbitrary-horizon-coefficients} results in $\D_N(\mathcal A_N(\gamma))\le\widehat\tau_NR^2+\widehat V_N\sigma^2=\sigma^2W_p(\gamma)=\frac{\sigma^2\widehat z_p}{1+(N-p-q_p)\widehat z_p}$,
where the last two equalities follow from~\eqref{eq:arbitrary-horizon-coefficients}, \eqref{eq:explicit-serial-cap-count}, and~\eqref{eq:arbitrary-horizon-budget-formula}, which demonstrates~\eqref{eq:arbitrary-horizon-bound}.
\end{proof}

\section{Explicit Formula for the \texttt{silver}-Prefix Length}\label{app:explicit-p}

In this section, we provide an explicit formula for the prefix length $p$ selected by~\eqref{eq:prefix-length-selection}. With $(\gamma,\bar\alpha,\gamma_1)$ defined in~\eqref{eq:one-two-notation} and $\rho$ from~\eqref{eq:arbitrary-horizon-rho}, for every $k\in[N]\backslash\{0\}$, let $S_k$ be the $k$-step \texttt{silver} block in~\eqref{eq:balanced-silver-blocks} and let $S_0$ be an empty block. By~\eqref{eq:balanced-silver-coefficients} and~\eqref{eq:prefix-budget-values}, for any $k\in[N]$, we have $z_k=\tau_k\gamma+V_k$, $z_0=\gamma$, and $W_k(\gamma)=\mathcal T^{\circ(N-k)}(z_k)$, where $\mathcal T$ and $\mathcal T^{\circ (N-k)}$ are defined in~\eqref{eq:one-step-optimal-value} and~\eqref{eq:serial-map-iteration}.

\noindent \textbf{Thresholds for comparing prefix lengths.}
Let's fix $k\in[N]\backslash[1]$ and $i\in[k-1]\backslash\{1\}$, and set $d=k-i$. It follows~\eqref{eq:serial-map-iteration} and~\eqref{eq:prefix-budget-values} that $W_i(\gamma)=\mathcal T^{\circ(N-k)}\left(\mathcal T^{\circ d}(z_i)\right)$ and $W_k(\gamma)=\mathcal T^{\circ(N-k)}(z_k)$.
Since $\mathcal T(\cdot)$ is strictly increasing by~\eqref{eq:one-step-optimal-value}, we have
\begin{equation}
W_k(\gamma)<W_i(\gamma)\quad\Longleftrightarrow\quad z_k<\mathcal T^{\circ d}(z_i).
\label{eq:prefix-comparison-at-k}
\end{equation}
For any $j\in[d]\backslash\{0\}$, the $j$th short step following $S_i$ updates $x_{i+j-1}$ to $x_{i+j}$. Its input and output bounds are $\mathcal T^{\circ(j-1)}(z_i)$ and $\mathcal T^{\circ j}(z_i)$, respectively. To evaluate $\mathcal T^{\circ d}(z_i)$, we consider each possible number $r\in[d]$ of initial steps with stepsize $\bar\alpha$. By~\eqref{eq:one-step-optimal-value} and~\eqref{eq:explicit-serial-budget-identities}, these $r$ steps use $\mathcal T(z)=\rho z+\frac{1-\rho}{M}$, and the remaining $d-r$ steps follow $\mathcal T(z)=\frac{z}{1+z}$. Both formulas agree at $z=\gamma_1$. Let's define
\begin{equation}
\theta_{-1}=0\quad \text{and} \quad \theta_j=\frac{1}{M}+\rho^{-j}\left(\gamma_1-\tfrac{1}{M}\right),\quad \forall j\in\mathbb N.
\label{eq:arbitrary-horizon-cap-walls}
\end{equation}
Thus, $\frac{1}{M}+\rho^j\left(\theta_j-\frac{1}{M}\right)=\gamma_1$.
For every $r\in[d]$, let's further define
\begin{equation}
\begin{aligned}
A_{i,r}&=\rho^r\tau_i,\quad B_{i,r}=\rho^rV_i+\frac{1-\rho^r}{M}, \quad C_{ki,r}=V_k\left(1+(d-r)B_{i,r}\right)-B_{i,r}, \\
\text{and} \quad Q_{ki,r}&=\tau_k\left(1+(d-r)B_{i,r}\right)+(d-r)V_kA_{i,r}-A_{i,r}.
\end{aligned}
\label{eq:prefix-root-coefficients}
\end{equation}
The first equality in~\eqref{eq:balanced-silver-coefficients} guarantees $A_{i,r}\gamma+B_{i,r}=\frac{1}{M}+\rho^r(z_i-\frac{1}{M})$. Then we also define
\begin{equation}
L_{ki,r}=\max\left\{0,\frac{\theta_{r-1}-V_i}{\tau_i}\right\} \quad \text{and} \quad U_{ki,r}=
\begin{cases}
\frac{\theta_r-V_i}{\tau_i}&\text{if }r\in[d-1]\\
+\infty&\text{if }r=d.
\end{cases}
\label{eq:prefix-branch-domain}
\end{equation}
For every $r\in[d-1]$ satisfying $Q_{ki,r}^2>4(d-r)\tau_kA_{i,r}C_{ki,r}$, we set
\begin{equation}
\begin{aligned}
\ell_{ki,r}&=\max\left\{L_{ki,r},\tfrac{-Q_{ki,r}-\sqrt{Q_{ki,r}^2-4(d-r)\tau_kA_{i,r}C_{ki,r}}}{2(d-r)\tau_kA_{i,r}}\right\}\\
\text{and} \quad u_{ki,r}&=\min\left\{U_{ki,r},\tfrac{-Q_{ki,r}+\sqrt{Q_{ki,r}^2-4(d-r)\tau_kA_{i,r}C_{ki,r}}}{2(d-r)\tau_kA_{i,r}}\right\}.
\end{aligned}
\label{eq:prefix-quadratic-root-interval}
\end{equation}
Meanwhile, for every $r\in[d-1]$ satisfying $Q_{ki,r}^2\le4(d-r)\tau_kA_{i,r}C_{ki,r}$, we set $(\ell_{ki,r},u_{ki,r})=(+\infty,-\infty)$. When $r=d$, we consider
\begin{equation}
(\ell_{ki,d},u_{ki,d})=
\begin{cases}
\left(\max\left\{L_{ki,d},-\tfrac{C_{ki,d}}{Q_{ki,d}}\right\},+\infty\right)&\text{if }Q_{ki,d}<0\\
\left(L_{ki,d},-\tfrac{C_{ki,d}}{Q_{ki,d}}\right)&\text{if }Q_{ki,d}>0\\
\left(L_{ki,d},+\infty\right)&\text{if }Q_{ki,d}=0\ \text{and} \ C_{ki,d}<0\\
(+\infty,-\infty)&\text{if }Q_{ki,d}=0\ \text{and} \ C_{ki,d}\ge0.
\end{cases}
\label{eq:prefix-linear-root-interval}
\end{equation}
Let $\mathcal J_{ki}=\{r\in[d]:\ell_{ki,r}<u_{ki,r}\}$ and define 
\begin{equation}
(\ell_{ki},u_{ki})=
\begin{cases}
\left(\displaystyle\min_{r\in\mathcal J_{ki}}\ell_{ki,r},\displaystyle\max_{r\in\mathcal J_{ki}}u_{ki,r}\right)&\text{if }\mathcal J_{ki}\ne\varnothing,\\
(+\infty,-\infty)&\text{if }\mathcal J_{ki}=\varnothing.
\end{cases}
\label{eq:prefix-pair-endpoints}
\end{equation}
\noindent \textbf{Explicit prefix-length formula.}
For every $k\in[N]\backslash[1]$, by defining
\begin{equation}
\underline\gamma_k=\max_{i\in[k-1]\backslash\{1\}}\ell_{ki}\quad \text{and} \quad \overline\gamma_k=\min_{i\in[k-1]\backslash\{1\}}u_{ki},
\label{eq:prefix-record-endpoints}
\end{equation}
the prefix length in~\eqref{eq:prefix-length-selection} is given by
\begin{equation}
p=\max\left\{\{0\} \cup \{k\in[N]\backslash[1]:\underline\gamma_k<\gamma<\overline\gamma_k\}\right\}.
\label{eq:prefix-length-closed-form}
\end{equation}
We will prove~\eqref{eq:prefix-length-closed-form} in the following lemma.

\begin{lemma}[Explicit prefix selection]\label{lem:explicit-prefix-thresholds}
For $M>1$, $\gamma\ge0$, and $N\in\mathbb N\backslash\{0\}$, the prefix length given by~\eqref{eq:prefix-length-closed-form} satisfies the selection rule in~\eqref{eq:prefix-length-selection}. In particular, it is the smallest index attaining
\begin{equation}
W_p(\gamma)=\min_{k\in[N]}W_k(\gamma).
\label{eq:prefix-bound-minimum}
\end{equation}
\end{lemma}
\begin{proof}
Let's fix $k\in[N]\backslash[1]$ and $i\in[k-1]\backslash\{1\}$, and consider $d=k-i$. By~\eqref{eq:explicit-serial-budget-identities} and~\eqref{eq:prefix-root-coefficients}, after taking $i$-step silver block and $r$ constant stepsizes $\bar\alpha$, the normalized bound becomes
$\frac{1}{M}+\rho^r(z_i-\frac{1}{M})=A_{i,r}\gamma+B_{i,r}$ for any $r\in[d]$. Moreover, for every $r\in[d]\backslash\{0\}$, \eqref{eq:arbitrary-horizon-cap-walls} guarantees $\tfrac{1}{M}+\rho^{r-1}\left(z_i-\tfrac{1}{M}\right)\ge\gamma_1$ if and only if $z_i\ge\theta_{r-1}$.
Since $0<\rho<1$ and $\gamma_1>\frac{1}{M}$, we have $\theta_j\le\theta_{r-1}$ for any $j\in[r-1]$. Thus, $z_i\ge\theta_{r-1}$ ensures that the first $r$ short steps following $S_i$ should all be $\bar\alpha$. For any $r\in[d-1]$, \eqref{eq:arbitrary-horizon-cap-walls} also implies that $\tfrac{1}{M}+\rho^r(z_i-\tfrac{1}{M})\le\gamma_1$ if and only if $z_i\le\theta_r$.
Because $0\le\frac{z}{1+z}\le z\le\gamma_1$ is satisfied for any $z\in[0,\gamma_1]$, the remaining $d-r$ short steps update $z$ with $\mathcal T(z)=\frac{z}{1+z}$. Therefore, $r$ steps with stepsize $\bar\alpha$ followed by $d-r$ steps using $\mathcal T(z)=\frac{z}{1+z}$ give $\mathcal T^{\circ d}(z_i)$, where, by~\eqref{eq:arbitrary-horizon-cap-walls}, $z_i$ satisfies
\begin{equation}\label{eq:explicitp_conditiononZ}
\begin{cases}
0\le z_i\le\theta_0&\text{if }r=0\\
\theta_{r-1}\le z_i\le\theta_r&\text{if }r\in[d-1]\backslash\{0\}\\
z_i\ge\theta_{d-1}&\text{if }r=d.
\end{cases}
\end{equation}
For $j\in[d-1]$, if $z_i = \theta_j$, then the bound after $j$ step with constant stepsize $\bar\alpha$ is $\gamma_1$ by~\eqref{eq:arbitrary-horizon-cap-walls}. From~\eqref{eq:one-step-optimal-value}, $\rho\gamma_1+\frac{1-\rho}{M}=\frac{\gamma_1}{1+\gamma_1}$, so $z_i=\theta_j$ permits either adjacent choice of $r$. Following~\eqref{eq:prefix-branch-domain}, substituting $z_i=\tau_i\gamma+V_i$ results in $L_{ki,r}\le\gamma\le U_{ki,r}$. Here, we have $\tau_0=1$ and $\tau_i>0$ for any $i\in[N]\backslash\{0\}$ by~\eqref{eq:balanced-silver-coefficients} and Definition~\ref{def:silver-join}.

For any $L_{ki,r}\le\gamma\le U_{ki,r}$ satisfying $A_{i,r}\gamma+B_{i,r}>0$, \eqref{eq:explicit-serial-budget-identities} leads to $\frac{1}{\mathcal T^{\circ d}(z_i)}=\frac{1}{A_{i,r}\gamma+B_{i,r}}+d-r$.
Hence,
\begin{equation}
\mathcal T^{\circ d}(z_i)=\frac{A_{i,r}\gamma+B_{i,r}}{1+(d-r)(A_{i,r}\gamma+B_{i,r})},
\label{eq:prefix-comparison-branch-value}
\end{equation}
which also holds when $A_{i,r}\gamma+B_{i,r}=0$ or $r=d$. Moreover, by~\eqref{eq:prefix-root-coefficients}, we have $A_{i,r}\gamma+B_{i,r}=\rho^rz_i+\frac{1-\rho^r}{M}\ge0$ and $1+(d-r)(A_{i,r}\gamma+B_{i,r})\ge1$.
Substituting $z_k=\tau_k\gamma+V_k$ into~\eqref{eq:prefix-comparison-branch-value} and using~\eqref{eq:prefix-root-coefficients}, we obtain
\begin{equation}
z_k-\mathcal T^{\circ d}(z_i)
=\tfrac{(\tau_k\gamma+V_k)\left(1+(d-r)(A_{i,r}\gamma+B_{i,r})\right)-(A_{i,r}\gamma+B_{i,r})}{1+(d-r)(A_{i,r}\gamma+B_{i,r})}=\tfrac{(d-r)\tau_kA_{i,r}\gamma^2+Q_{ki,r}\gamma+C_{ki,r}}{1+(d-r)(A_{i,r}\gamma+B_{i,r})}.
\label{eq:prefix-comparison-branch-difference}
\end{equation}
Thus, for any $L_{ki,r}\le\gamma\le U_{ki,r}$, $z_k<\mathcal T^{\circ d}(z_i)$ if and only if $(d-r)\tau_kA_{i,r}\gamma^2+Q_{ki,r}\gamma+C_{ki,r}<0$.
Because $(d-r)\tau_kA_{i,r}>0$ for any $r\in[d-1]$, solving $(d-r)\tau_kA_{i,r}\gamma^2+Q_{ki,r}\gamma+C_{ki,r}<0$ together with $L_{ki,r}\le\gamma\le U_{ki,r}$ results in~\eqref{eq:prefix-quadratic-root-interval}. If $Q_{ki,r}^2\le4(d-r)\tau_kA_{i,r}C_{ki,r}$, then $(d-r)\tau_kA_{i,r}\gamma^2+Q_{ki,r}\gamma+C_{ki,r}\ge0$ for every $\gamma\geq 0$, so there is no $\gamma\in[L_{ki,r}, U_{ki,r}]$ satisfying $z_k<\mathcal T^{\circ d}(z_i)$. For $r=d$, solving $Q_{ki,d}\gamma+C_{ki,d}<0$ together with $\gamma\ge L_{ki,d}$ gives~\eqref{eq:prefix-linear-root-interval}. Consequently, for arbitrary $r\in\mathcal J_{ki}$, we always have that
$\ell_{ki,r}=\inf\{\gamma>0:L_{ki,r}\le\gamma\le U_{ki,r},\ z_k<\mathcal T^{\circ d}(z_i)\}$ and $u_{ki,r}=\sup\{\gamma>0:L_{ki,r}\le\gamma\le U_{ki,r},\ z_k<\mathcal T^{\circ d}(z_i)\}$.

Let $I_{ki}=\{\gamma>0:z_k<\mathcal T^{\circ d}(z_i)\}$. By~\eqref{eq:one-step-optimal-value}, we have that $\nabla\mathcal T(z)=
\begin{cases}
\frac{1}{(1+z)^2}&\text{if }0<z<\gamma_1\\
\rho&\text{if }z>\gamma_1.
\end{cases}$ 
Moreover, $\frac{1}{(1+\gamma_1)^2}=\rho$ follows from~\eqref{eq:one-two-notation} and~\eqref{eq:arbitrary-horizon-rho}. Hence, $\mathcal T(\cdot)$ is increasing and concave, and $\mathcal T^{\circ d}(\cdot)$ is concave as well. Since $z_i=\tau_i\gamma+V_i$ and $z_k=\tau_k\gamma+V_k$, the function $\gamma\mapsto z_k-\mathcal T^{\circ d}(z_i)$ is continuous and convex. Therefore, $I_{ki}$ is empty or an open interval. 
By~\eqref{eq:arbitrary-horizon-cap-walls}, $\theta_{-1}=0<\theta_0<\cdots<\theta_{d-1}$. Thus the conditions $\theta_{r-1}\le z_i\le\theta_r$ for $r\in[d-1]$ and $z_i\ge\theta_{d-1}$ for $r=d$ in~\eqref{eq:explicitp_conditiononZ} cover all $z_i\ge0$. 
Substituting $z_i=\tau_i\gamma+V_i$ into that condition gives $L_{ki,r}\le\gamma\le U_{ki,r}$ by~\eqref{eq:prefix-branch-domain}. Hence $\bigcup_{r\in[d]}\left\{\gamma>0:L_{ki,r}\le\gamma\le U_{ki,r}\right\}=(0,+\infty)$. This gives $I_{ki}=\bigcup_{r\in[d]}\left\{\gamma>0:L_{ki,r}\le\gamma\le U_{ki,r},\ z_k<\mathcal T^{\circ d}(z_i)\right\}$.

Let's fix $\gamma\in I_{ki}$. Every sufficiently small open interval around $\gamma$ belongs to $I_{ki}$ and contains a point $\gamma'$ distinct from all $\{L_{ki,r}\}_{r\in[d]}$ and $\{U_{ki,r}\}_{r\in[d]}$. By the branch coverage, $\gamma'$ satisfies $L_{ki,r} <\gamma' < U_{ki,r}$ for some $r\in[d]$. Thus every element in $I_{ki}$ can be approached by $\gamma'$ satisfying $L_{ki,r}<\gamma'<U_{ki,r}$ and $z_k<\mathcal T^{\circ d}(z_i)$ for some $r\in[d]$. Therefore, $I_{ki}=\varnothing$ if and only if $\mathcal J_{ki}=\varnothing$. When $I_{ki}\ne\varnothing$, we have $\inf_{\gamma\in I_{ki}} \gamma=\min_{r\in\mathcal J_{ki}}\ell_{ki,r}$ and $\sup_{\gamma\in I_{ki}} \gamma=\max_{r\in\mathcal J_{ki}}u_{ki,r}$.
By~\eqref{eq:prefix-comparison-at-k} and~\eqref{eq:prefix-pair-endpoints}, we conclude that
\begin{equation}
W_k(\gamma)<W_i(\gamma)\quad\Longleftrightarrow\quad \ell_{ki}<\gamma<u_{ki},\qquad \forall\gamma>0.
\label{eq:prefix-pair-comparison}
\end{equation}

By~\eqref{eq:balanced-silver-blocks}--\eqref{eq:balanced-silver-coefficients}, $z_1=\rho\gamma+\bar\alpha^2$. By~\eqref{eq:one-step-optimal-value}, $\mathcal T(\gamma)\le\rho\gamma+\bar\alpha^2$, then we have $W_0(\gamma)=\mathcal T^{\circ(N-1)}\left(\mathcal T(\gamma)\right)\le\mathcal T^{\circ(N-1)}(z_1)=W_1(\gamma)$.
For any $\gamma>0$ and any $k\in[N]\backslash[1]$, \eqref{eq:prefix-record-endpoints} and~\eqref{eq:prefix-pair-comparison} result in
\begin{equation*}
\begin{aligned}
\underline\gamma_k<\gamma<\overline\gamma_k
&\quad\Longleftrightarrow\quad \ell_{ki}<\gamma<u_{ki}\quad\Longleftrightarrow\quad W_k(\gamma)<W_i(\gamma),\qquad\forall i\in[k-1]\backslash\{1\}.
\end{aligned}
\end{equation*}
Substituting $i=0$, for any $k\in[N]\backslash[1]$, $W_k(\gamma)<W_0(\gamma)\le W_1(\gamma)$, which further leads to
\begin{equation}
\underline\gamma_k<\gamma<\overline\gamma_k\quad\Longleftrightarrow\quad W_k(\gamma)<W_i(\gamma),\qquad\forall i\in[k-1].
\label{eq:prefix-record-comparison}
\end{equation}
When $\gamma=0$, from~\eqref{eq:prefix-budget-values} we know that $W_0(0)=0$ and $W_k(0)=\mathcal T^{\circ(N-k)}(V_k)>0$ for any $k\in[N]\backslash\{0\}$. The strict inequality follows from $V_k>0$ in~\eqref{eq:balanced-silver-coefficients} and the strict monotonicity of $\mathcal T(\cdot)$ with $\mathcal T(0)=0$. Also, $\underline\gamma_k\ge0$ follows from~\eqref{eq:prefix-branch-domain}, \eqref{eq:prefix-pair-endpoints}, and~\eqref{eq:prefix-record-endpoints}. So when $\gamma=0$, the condition $\underline\gamma_k<0<\overline\gamma_k$ cannot hold because $\underline\gamma_k\ge0$. In the meantime, the condition $W_k(0)<W_i(0)$ for every $i\in[k-1]$ is also false, since $0\in[k-1]$ and $W_k(0)>W_0(0)=0$. Therefore,~\eqref{eq:prefix-record-comparison} holds for all $\gamma\ge0$.

Let's select $p$ following~\eqref{eq:prefix-length-closed-form}, then $p\in[N]\backslash\{1\}$. If $p\ge2$, \eqref{eq:prefix-length-closed-form} and~\eqref{eq:prefix-record-comparison} imply that $\underline\gamma_p<\gamma<\overline\gamma_p$ and $W_p(\gamma)<W_i(\gamma)$ for every $i\in[p-1]$. When $p=0$, no index $i$ satisfies $i<p$. Thus, the first condition in~\eqref{eq:prefix-length-selection} always holds. Meanwhile,~\eqref{eq:prefix-length-closed-form} leads to $\gamma\notin\left(\underline\gamma_j,\overline\gamma_j\right)$ for every $j\in([N]\backslash\{1\})\backslash[p]$.
Suppose that $W_j(\gamma)<W_p(\gamma)$ holds for some $j>p$, and we could choose the smallest such $j$. Because $W_0(\gamma)\le W_1(\gamma)$, we have $j\ne1$. The choice of $j$ guarantees that $W_i(\gamma)\ge W_p(\gamma)$ for all $i\in[j-1]\backslash[p-1]$, while the comparison established above ensures $W_i(\gamma)>W_p(\gamma)$ for $i\in[p-1]$. Consequently, we have $W_j(\gamma)<W_p(\gamma)\le W_i(\gamma)$ for all $i\in[j-1]$.
By~\eqref{eq:prefix-record-comparison}, this further implies $\underline\gamma_j<\gamma<\overline\gamma_j$, contradicting~\eqref{eq:prefix-length-closed-form}. Therefore, $W_p(\gamma)\le W_j(\gamma)$ for all $j\in[N]\backslash[p]$.
Hence, $p$ satisfies both conditions in~\eqref{eq:prefix-length-selection}. In particular, $p$ is the smallest index attaining $\min_{k\in[N]}W_k(\gamma)$, which completes the proof of~\eqref{eq:prefix-bound-minimum}.
\end{proof}

\section{Proof of Theorem~\ref{thm:arbitrary-horizon-global-optimality}}
\label{app:arbitrary-horizon-global}

In this section, we use parameters $(\bar\alpha,\gamma_1)$ from~\eqref{eq:one-two-notation}, $\gamma_2$ from~\eqref{eq:two-step-parameters}, $\rho$ from~\eqref{eq:arbitrary-horizon-rho}, and $\Gamma(M)$ from~\eqref{eq:global-optimality-sufficient-threshold}. Meanwhile, we consider the normalized optimal one-step error in Theorem~\ref{thm:one-step}, given by~\eqref{eq:one-step-optimal-value}.
We also utilize the composition notation in~\eqref{eq:serial-map-iteration}.
For any $\{a,z\}\subset\R_{\ge0}$, let's define $(F_a(z),C(a))$ by 
\begin{equation}\label{eq:global-optimal-defFaCa}
    F_a(z)=(1-a)^2z+a^2\quad \text{and} \quad C(a)=\bigl((1-Ma)^2-(1-a)^2\bigr)_+.
\end{equation}
In particular, $F_a(z)$ is the normalized right-hand side of~\eqref{eq:recurrence}. In the meantime, $C(a) = 0$ if $(1-Ma)^2 \le (1-a)^2$, and $C(a) = (1-Ma)^2-(1-a)^2$ otherwise.

\begin{lemma}[Scalar comparisons for capped stepsizes]
\label{lem:global-unified-scalar-comparisons}
Let function $\mathcal T(\cdot)$ be defined in~\eqref{eq:one-step-optimal-value}, following the definitions in~\eqref{eq:global-optimal-defFaCa}, we have
\begin{equation}
\begin{aligned}
F_b(F_a(z))+a^2C(b)&\ge F_{\bar\alpha}(F_{\bar\alpha}(z)), &&\forall (a,b,z)\in[0,\bar\alpha]\times[\bar\alpha,1]\times[\gamma_2,\Gamma(M)], \\ 
\text{and} \qquad F_b(z)+\bar\alpha^2C(b)&\ge F_{\bar\alpha}(z), &&\forall (b,z)\in[\bar\alpha,1]\times[0,\mathcal T(\Gamma(M))]. 
\end{aligned}
\label{eq:global-unified-scalar-comparisons}
\end{equation}
\end{lemma}

\begin{proof}
By~\eqref{eq:one-two-notation}, \eqref{eq:two-step-parameters}, and~\eqref{eq:global-optimality-sufficient-threshold}, we have $\gamma_2>\gamma_1$ and $\Gamma(M)-\gamma_2=\frac{4(M+1)}{(M-1)^2}>0$. Substituting $\gamma_2$ into~\eqref{eq:one-step-optimal-value} leads to $\mathcal T(\gamma_2)=\rho\gamma_2+\bar\alpha^2=\gamma_1$. Since $\Gamma(M)>\gamma_1$, substituting $\Gamma(M)=\rho^{-\frac{3}{2}}-1$ gives
\begin{equation}
\sqrt{\rho}\,\mathcal T(\Gamma(M))
=\sqrt{\rho}\,F_{\bar\alpha}(\Gamma(M))
=1-\rho^{\frac{3}{2}}+\sqrt{\rho}\,\bar\alpha^2
=\bar\alpha(1+2\sqrt{\rho}),
\label{eq:global-cap-endpoint}
\end{equation}
where the last equality uses $\sqrt{\rho}=1-\bar\alpha$ from~\eqref{eq:arbitrary-horizon-rho}.

For any $b\in[\bar\alpha,1]$, we have from~\eqref{eq:global-optimal-defFaCa} that $C(b)=(M^2-1)b(b-\bar\alpha)$. Using $\bar\alpha^2(M^2-1)=4\sqrt{\rho}$ and~\eqref{eq:global-cap-endpoint}, we further obtain $F_b(z)+\bar\alpha^2C(b)-F_{\bar\alpha}(z)=2(\bar\alpha(1+2\sqrt{\rho})-\sqrt{\rho}\,z)(b-\bar\alpha)+(z+1+\bar\alpha^2(M^2-1))(b-\bar\alpha)^2=2\sqrt{\rho}(\mathcal T(\Gamma(M))-z)(b-\bar\alpha)+(z+1+\bar\alpha^2(M^2-1))(b-\bar\alpha)^2
\ge0$
for any $z\in[0,\mathcal T(\Gamma(M))]$. This proves the second inequality in~\eqref{eq:global-unified-scalar-comparisons}.

For any fixed $(a,b)\in[0,\bar\alpha]\times[\bar\alpha,1]$, the difference function $F_b(F_a(z))+a^2C(b)-F_{\bar\alpha}(F_{\bar\alpha}(z))$ is affine in $z$. Therefore, to complete the proof, it suffices to prove that this difference function is nonnegative when $z=\gamma_2$ and $z=\Gamma(M)$. When $z=\gamma_2$, we have $F_a(\gamma_2)\ge\mathcal T(\gamma_2)=\gamma_1$ because $0\le a\le\bar\alpha$. Since function $F_b$ is nondecreasing and $C(b)\ge0$, we know that $F_b(F_a(\gamma_2))+a^2C(b)\ge F_b(\gamma_1)=\bar\alpha+(1+\gamma_1)(b-\bar\alpha)^2\ge\bar\alpha=F_{\bar\alpha}(F_{\bar\alpha}(\gamma_2))$.
When $z=\Gamma(M)$, let's set $t=\frac{a}{\bar\alpha}\in[0,1]$ and define $A(t)=\rho(\Gamma(M)(1-\bar\alpha t)^2+(1+\sqrt{\rho})^2t^2+1)$.
Substituting $a=\bar\alpha t$, $C(b)=(M^2-1)b(b-\bar\alpha)$, and $\Gamma(M)=\rho^{-\frac{3}{2}}-1$, by $A(t) > 0$, we have
\begin{equation}
\begin{aligned}
&F_b(F_a(\Gamma(M)))+a^2C(b)-F_{\bar\alpha}(F_{\bar\alpha}(\Gamma(M)))\\
&\quad=A(t)\left(\tfrac{b-\bar\alpha}{\sqrt{\rho}}-\tfrac{\bar\alpha(1-t)\left(2(1+\rho t)+\gamma_1(1-t)\right)}{A(t)}\right)^2\\
&\quad\quad\quad+\tfrac{\bar\alpha^2(1-t)}{A(t)}\left(A(t)\left(2(1+\sqrt{\rho})+\tfrac{1-t}{\sqrt{\rho}}\right)-(1-t)\left(2(1+\rho t)+\gamma_1(1-t)\right)^2\right).
\end{aligned}
\label{eq:global-two-step-completed-square}
\end{equation}
It further follows $\rho=\frac{1}{(1+\gamma_1)^2}$ that $\frac{1}{\rho^2}(A(t)(2(1+\sqrt{\rho})+\frac{1-t}{\sqrt{\rho}})-(1-t)(2(1+\rho t)+\gamma_1(1-t))^2)=24(t-\frac{5}{48})^2+\frac{71}{96}+2\gamma_1(2-t+6t^2(3-t))+2\gamma_1^2(3-t+2t^2(5-3t))+2\gamma_1^3(2-t+2t^2(1-t))+\gamma_1^4(1-t)\ge0$,
where the inequality is from $t\in[0,1]$ and $\gamma_1>0$. Together with $A(t)>0$ and $1-t\ge0$, this proves~\eqref{eq:global-two-step-completed-square} is nonnegative. The first inequality in~\eqref{eq:global-unified-scalar-comparisons} directly follows the affinity in $z$.
\end{proof}

\begin{proof}[Proof of Theorem~\ref{thm:arbitrary-horizon-global-optimality}]
Let's fix $\gamma\in[0,\Gamma(M)]$ following the definition in~\eqref{eq:one-two-notation}.

\noindent \textbf{Upper bound for the proposed schedule.}
By Theorem~\ref{thm:arbitrary-horizon-convergence}, Lemma~\ref{lem:explicit-prefix-thresholds}, \eqref{eq:prefix-budget-values}, and $z_0=\gamma$ in~\eqref{eq:balanced-silver-coefficients}, we have that
\begin{equation}
\D_N(\mathcal A_N(\gamma))\le\sigma^2\min_{k\in[N]}W_k(\gamma)\le\sigma^2W_0(\gamma)=\sigma^2\mathcal T^{\circ N}(\gamma).
\label{eq:global-proposed-upper}
\end{equation}
\noindent \textbf{Schedules with $\alpha_i>1$ for some $i\in[N-1]\backslash\{0\}$.}
It follows~\eqref{eq:global-optimal-defFaCa} that
\begin{equation}
F_a(z)=(1-a)^2z+a^2=\tfrac{z}{1+z}+(1+z)\left(a-\tfrac{z}{1+z}\right)^2\ge\tfrac{z}{1+z},\qquad \forall\{a,z\}\subset\R_{\ge0}.
\label{eq:global-unrestricted-one-step-comparison}
\end{equation}
Meanwhile,~\eqref{eq:one-step-optimal-value} implies that $\mathcal T(z)=\min_{0\le a\le\bar\alpha}F_a(z)\le F_0(z)=z$ and $\mathcal T(z)=F_{\bar\alpha}(z)$ for any $z\ge\gamma_1$.
Also, by~\eqref{eq:one-step-optimal-value} and~\eqref{eq:arbitrary-horizon-rho}, we know that function $\mathcal T(\cdot)$ is strictly increasing and $\rho\in(0,1)$.
Suppose horizon $N\ge2$ and stepsize $\alpha_i>1$ for some $i\in[N-1]\backslash\{0\}$. Let's take function $f(x)=\frac{\|x\|^2}{2}$, initial iterate $x_0=(R,0,\ldots,0)\in\R^d$, and stochastic gradient estimates $g_k=x_k+\xi_k\mathbf e_1$, where $\mathbf e_1$ is the first column vector of the $d$-dimensional identity matrix and the $\{\xi_k\}_{k\in\mathbb{N}}$ are independent random variables with $\mathbb P[\xi_k=\sigma]=\mathbb P[\xi_k=-\sigma]=\frac{1}{2}$ for all $k\in\mathbb{N}$. By~\eqref{eq:sg-update}, we have that $\frac{\E[\|x_{k+1}\|^2]}{\sigma^2}=F_{\alpha_k}\left(\frac{\E[\|x_k\|^2]}{\sigma^2}\right)$ for any $k\in[N-1]$.
In particular, $\frac{\E[\|x_{i+1}\|^2]}{\sigma^2}\ge\alpha_i^2>1$. Applying~\eqref{eq:global-unrestricted-one-step-comparison} to the subsequent steps results in $\frac{\E[\|x_{i+1+j}\|^2]}{\sigma^2}\ge\frac{1}{j+1}$ for any $j\in[N-i-1]$. Therefore,
\begin{equation}
\D_N(\boldsymbol\alpha)\ge\tfrac{\sigma^2}{N-i}\ge\tfrac{\sigma^2}{N-1}.
\label{eq:global-stepsize-above-one-lower}
\end{equation}
For any $N\in\left[2,\left\lceil\tfrac{1}{\bar\alpha}\right\rceil\right]$, let's consider the first step with stepsize $\bar\alpha$, followed by other stepsizes $\frac{\bar\alpha}{1+(j-1)\bar\alpha}$ for all $j\in[N-1]\backslash\{0\}$. By~\eqref{eq:coefficients}, we know $\prod_{j\in[N-1]\backslash\{0\}}\left(1-\frac{\bar\alpha}{1+(j-1)\bar\alpha}\right)^2
=\frac{\rho}{\left(1+(N-2)\bar\alpha\right)^2}$ and $\left(\frac{\bar\alpha}{1+(j-1)\bar\alpha}\right)^2\prod_{\ell\in[N-1]\backslash[j]}\left(1-\frac{\bar\alpha}{1+(\ell-1)\bar\alpha}\right)^2
=\frac{\bar\alpha^2}{\left(1+(N-2)\bar\alpha\right)^2}$ for any $j\in[N-1]\backslash\{0\}$.
Since $0\le\frac{\bar\alpha}{1+(j-1)\bar\alpha}\le\bar\alpha$ for $j\in[N-1]\backslash\{0\}$ and $\mathcal T(z)\le F_a(z)$ for any $0\le a\le\bar\alpha$ and $z\geq 0$, the monotonicity of $\mathcal T(\cdot)$ (see~\eqref{eq:one-step-optimal-value}) and~\eqref{eq:coefficients} result in $\mathcal T^{\circ N}(\Gamma(M))\le\frac{\rho F_{\bar\alpha}(\Gamma(M))+(N-1)\bar\alpha^2}{\left(1+(N-2)\bar\alpha\right)^2}$.
We further have from~\eqref{eq:global-cap-endpoint} that $\left(1+(N-2)\bar\alpha\right)^2-(N-1)\left(\rho F_{\bar\alpha}(\Gamma(M))+(N-1)\bar\alpha^2\right)=\rho\left(1-(N-1)\bar\alpha\right)+(N-1)\bar\alpha^2\sqrt{\rho}>0$,
where $(N-1)\bar\alpha<1$ results from $N\le\left\lceil\frac{1}{\bar\alpha}\right\rceil$. Therefore, $\mathcal T^{\circ N}(\Gamma(M))<\frac{1}{N-1}$ for any $N\in\left[2,\left\lceil\tfrac{1}{\bar\alpha}\right\rceil\right]$.

When $N\ge\left\lceil\frac{1}{\bar\alpha}\right\rceil$, we have $\frac{1}{N}\le\bar\alpha$ and $F_{\frac{1}{N}}\left(\frac{1}{N-1}\right)=\frac{1}{N}$.
Since function $F_{\frac{1}{N}}(\cdot)$ is strictly increasing, from the induction method, the implication
\begin{equation*}
\mathcal T^{\circ N}(\Gamma(M))<\tfrac{1}{N-1}\quad\Longrightarrow\quad \mathcal T^{\circ(N+1)}(\Gamma(M))\le F_{\frac{1}{N}}\left(\mathcal T^{\circ N}(\Gamma(M))\right)<\tfrac{1}{N}
\end{equation*}
indicates the inequality $\mathcal T^{\circ N}(\Gamma(M))<\frac{1}{N-1}$ holds for every $N\ge2$. It follows $\gamma\in[0,\Gamma(M)]$ that
\begin{equation}
\mathcal T^{\circ N}(\gamma)\le\mathcal T^{\circ N}(\Gamma(M))<\tfrac{1}{N-1},\qquad \forall N\ge2.
\label{eq:global-short-value-comparison}
\end{equation}
Combining~\eqref{eq:global-proposed-upper}, \eqref{eq:global-stepsize-above-one-lower} and~\eqref{eq:global-short-value-comparison}, we obtain $\D_N(\boldsymbol\alpha)\ge\frac{\sigma^2}{N-1}>\sigma^2\mathcal T^{\circ N}(\gamma)\ge\D_N(\mathcal A_N(\gamma))$.
Consequently, $\mathcal A_N(\gamma)$ satisfies $\alpha_k\in[0,1]$ for every $k\in[N-1]\backslash\{0\}$, and it remains to prove $\D_N(\boldsymbol\alpha)\ge\sigma^2\mathcal T^{\circ N}(\gamma)$ for schedules $\boldsymbol\alpha\in[0, +\infty)\times[0, 1]^{N-1}$. For $N=1$, there is no restriction on $\alpha_0$.

\noindent \textbf{Hard instances and the lower-bound function.}
Fix $\boldsymbol\alpha=(\alpha_0,\ldots,\alpha_{N-1})\in\mathbb R_{\ge0}^N$ with $\alpha_k\le1$ for every $k\in[N-1]\setminus\{0\}$. We consider function $f(\cdot)=\varphi_3(\cdot)$, where $\varphi_3(\cdot)$ is defined in~\eqref{eq:two-step-first-three-functions}, whose gradient function is
$\nabla\varphi_3(t)=
\begin{cases}
t&\text{if }t\le0\\
Mt&\text{if }t>0.
\end{cases}$
It is clear that $f\in\mathcal F_{1,M}$ and $x_\star=0$. Meanwhile, we choose the deterministic initial point $x_0=
\begin{cases}
-R&\text{if }\alpha_0\le\bar\alpha\\
R&\text{if }\alpha_0>\bar\alpha
\end{cases}$
such that $|x_0-x_\star|=R$.
For each integer $n\ge2$, let's define noise in stochastic gradient estimates by specifying $\{(s_j,\xi_j)\}$ such that for any $j\in[N-1]$,
\begin{equation}
s_j=
\begin{cases}
-1&\text{if }j<N-1\text{ and }\alpha_{j+1}\le\bar\alpha\\
1&\text{otherwise}
\end{cases}
\ \ \text{and} \ \ 
\xi_j=
\begin{cases}
-s_j\sigma\sqrt{n-1}&\text{with probability }\frac{1}{n}\\
\frac{s_j\sigma}{\sqrt{n-1}}&\text{with probability }1-\frac{1}{n}.
\end{cases}
\label{eq:global-scalar-noise}
\end{equation}
The stepsize schedule determines $\{s_j\}_{j\in[N-1]}$ before the run of SG method. For each fixed integer $n\geq 2$, draw random variables $\{\xi_j\}_{j\in[N-1]}$ independently of one another according to~\eqref{eq:global-scalar-noise}. For every $j\in[N-1]$, we set stochastic gradient estimate $g_j=\nabla\varphi_3(x_j)+\xi_j$ and generate the next iterate $x_{j+1}$ by~\eqref{eq:sg-update}. It follows~\eqref{eq:global-scalar-noise} that for every $j\in[N-1]$, $\E[\xi_j] = 0$, $\E[\xi_j^2] = \sigma^2$, $\E[g_j\mid\mathcal H_j]=\nabla\varphi_3(x_j)$, and $\E[(g_j-\nabla\varphi_3(x_j))^2\mid\mathcal H_j]=\sigma^2$.
Therefore, such stochastic oracle is admissible under Assumption~\ref{ass:oracle} for every integer $n\ge2$.

For every $k\in[N]\setminus\{0\}$, let $\mathcal E_k^{(n)}$ denote the event that none of the first $k$ stochastic gradient noise takes the first value in~\eqref{eq:global-scalar-noise}. Meanwhile, for every $j\in[k-1]$, let $\mathcal E_{k,j}^{(n)}$ denote the event that only $\xi_j$ takes the first value in~\eqref{eq:global-scalar-noise} among the first $k$ noises. In particular,
\begin{equation}
\begin{aligned}
\mathcal E_k^{(n)}=\bigcap_{\ell\in[k-1]}\left\{\xi_\ell=\tfrac{s_\ell\sigma}{\sqrt{n-1}}\right\}\ \ \text{and} \ \ 
\mathcal E_{k,j}^{(n)}=\left\{\xi_j=-s_j\sigma\sqrt{n-1}\right\}\cap\bigcap_{\ell\in[k-1]\setminus\{j\}}\left\{\xi_\ell=\tfrac{s_\ell\sigma}{\sqrt{n-1}}\right\}.
\end{aligned}
\label{eq:global-scalar-events}
\end{equation}
By the mutual independence among $\{\xi_j\}_{j\in[k-1]}$, we know that
\begin{equation}
\mathbb P[\mathcal E_k^{(n)}]=\left(1-\tfrac{1}{n}\right)^k\quad \text{and} \quad \mathbb P[\mathcal E_{k,j}^{(n)}]=\tfrac{1}{n}\left(1-\tfrac{1}{n}\right)^{k-1}.
\label{eq:global-scalar-event-probabilities}
\end{equation}
For each fixed $k\in[N]\backslash\{0\}$, the events in~\eqref{eq:global-scalar-events} are pairwise disjoint. Write $\mathbf 1_{\mathcal E}$ as the indicator function of event $\mathcal E$ and define the normalized contribution of such events to the squared error by
\begin{equation}
L_k^{(n)}(\boldsymbol\alpha)=\frac{1}{\sigma^2}\left(\E[x_k^2\mathbf 1_{\mathcal E_k^{(n)}}]+\sum_{j\in[k-1]}\E[x_k^2\mathbf 1_{\mathcal E_{k,j}^{(n)}}]\right).
\label{eq:global-finite-n-lower-values}
\end{equation}
Since $x_\star=0$,~\eqref{eq:worst-case-error} and~\eqref{eq:global-scalar-events} guarantee $\sigma^2L_N^{(n)}(\boldsymbol\alpha)\le\E[x_N^2]\le\D_N(\boldsymbol\alpha)$. 
On the event $\mathcal E_k^{(n)}$, the stochastic noise generated within the first $k$ iterations tend to zero as $n\to\infty$. By the continuity of $\nabla\varphi_3(\cdot)$ and~\eqref{eq:sg-update}, the corresponding iterates converge to those of the updates with zero noise. When $\mathcal E_k^{(n)}$ happens, the choice of $x_0$ and the formula for $\nabla\varphi_3(\cdot)$ give $x_0-\alpha_0\nabla\varphi_3(x_0)=
\begin{cases}
-(1-\alpha_0)R&\text{if }\alpha_0\le\bar\alpha\\
(1-M\alpha_0)R&\text{if }\alpha_0>\bar\alpha,
\end{cases}$
where both values are nonpositive because $\frac{1}{M}<\bar\alpha<1$. By~\eqref{eq:global-optimal-defFaCa}, $C(a)=0$ when $0\le a\le\bar\alpha$ and $C(a)=(1-Ma)^2-(1-a)^2$ when $a>\bar\alpha$. Therefore, for every $k\in[N]\backslash\{0\}$, by $\displaystyle\lim_{n\to+\infty}\mathbb P[\mathcal E_k^{(n)}]=1$ (see~\eqref{eq:global-scalar-event-probabilities}), $\alpha_{\ell}\in[0,1]$ for $\ell\in[k-1]\backslash\{0\}$, and properties of gradient function $\nabla\varphi_3(\cdot)$, we have that
\begin{equation}
\lim_{n\to+\infty}\tfrac{\E[x_k^2\mathbf 1_{\mathcal E_k^{(n)}}]}{\sigma^2}=\gamma((1-\alpha_0)^2+C(\alpha_0))\prod_{\ell\in[k-1]\setminus\{0\}}(1-\alpha_\ell)^2.
\label{eq:global-scalar-no-noise-limit}
\end{equation}
On the event $\mathcal E_{k,j}^{(n)}$, function $\nabla\varphi_3(\cdot)$ satisfies $\nabla\varphi_3(ct)=c\nabla\varphi_3(t)$ for any $c\ge0$. Dividing~\eqref{eq:sg-update} by $\sigma\sqrt{n-1}$ therefore results in $\frac{x_{\ell+1}}{\sigma\sqrt{n-1}}=\frac{x_\ell}{\sigma\sqrt{n-1}}-\alpha_\ell\nabla\varphi_3\left(\frac{x_\ell}{\sigma\sqrt{n-1}}\right)-\frac{\alpha_\ell\xi_\ell}{\sigma\sqrt{n-1}}$ for any $\ell\in[k-1]$.
For every $\ell\in[k-1]\backslash\{j\}$, $\displaystyle\lim_{n\to+\infty}\tfrac{\xi_\ell}{\sigma\sqrt{n-1}}=\displaystyle\lim_{n\to+\infty}\tfrac{s_\ell}{n-1}=0$. When $\ell=j$, $\frac{\xi_j}{\sigma\sqrt{n-1}}=-s_j$. On the event $\mathcal E_{k,j}^{(n)}$, starting from $\displaystyle\lim_{n\to+\infty}\tfrac{x_0}{\sigma\sqrt{n-1}}=0$, the continuity of gradient function $\nabla\varphi_3(\cdot)$ in the update guarantees $\displaystyle\lim_{n\to+\infty}\tfrac{x_\ell}{\sigma\sqrt{n-1}}=0$ for every $\ell\in[j]$ and $\displaystyle\lim_{n\to+\infty}\tfrac{x_{j+1}}{\sigma\sqrt{n-1}}=s_j\alpha_j$.
If $j<k-1$, $\displaystyle\lim_{n\to+\infty}\frac{x_{j+2}}{\sigma\sqrt{n-1}} = s_j\alpha_j-\alpha_{j+1}\nabla\varphi_3(s_j\alpha_j)=
\begin{cases}
-(1-\alpha_{j+1})\alpha_j&\text{if }\alpha_{j+1}\le\bar\alpha\\
(1-M\alpha_{j+1})\alpha_j&\text{if }\alpha_{j+1}>\bar\alpha
\end{cases}$
by~\eqref{eq:global-scalar-noise} and the definition of function $\nabla\varphi_3(\cdot)$, where both values are nonpositive and their squares equal $\alpha_j^2((1-\alpha_{j+1})^2+C(\alpha_{j+1}))$ in their respective cases. Since $\nabla\varphi_3(t)=t$ for $t\le0$ and $\alpha_l\in[0,1]$, we have $\displaystyle\lim_{n\rightarrow+\infty}\tfrac{x_{\ell+1}}{\sigma\sqrt{n-1}} = (1-\alpha_l)\displaystyle\lim_{n\rightarrow+\infty}\tfrac{x_\ell}{\sigma\sqrt{n-1}} \le 0$, for all $l\in[k-1]\backslash[j+1]$. If $j=k-1$, then $\displaystyle\lim_{n\rightarrow+\infty}\tfrac{x_k}{\sigma\sqrt{n-1}} = s_{k-1}\alpha_{k-1}$.
By~\eqref{eq:global-scalar-event-probabilities}, $\displaystyle\lim_{n\to+\infty}(n-1)\mathbb P[\mathcal E_{k,j}^{(n)}]=1$. Since $x_k$ is constant on $\mathcal E_{k, j}^{(n)}$, we have $\tfrac{\E [x_k^2\mathbf 1_{\mathcal E_{k,j}^{(n)}} ]}{\sigma^2} = (n-1)\mathbb P[\mathcal E^{(n)}_{k,j}]\E[(\frac{x_k}{\sigma\sqrt{n-1}})^2\mid \mathcal{E}^{(n)}_{k,j}]$. Consequently, combining cases of $j<k-1$ and $j=k-1$, it follows that
\begin{equation}
\lim_{n\to\infty}\tfrac{\E[x_k^2\mathbf 1_{\mathcal E_{k,j}^{(n)}}]}{\sigma^2}=
\begin{cases}
\alpha_j^2\left((1-\alpha_{j+1})^2+C(\alpha_{j+1})\right)\displaystyle\prod_{\ell\in[k-1]\setminus[j+1]}(1-\alpha_\ell)^2,&\text{if }j<k-1,\\
\alpha_{k-1}^2,&\text{if }j=k-1.
\end{cases}
\label{eq:global-scalar-single-noise-limit}
\end{equation}
Define $L_k(\boldsymbol\alpha)=\displaystyle\displaystyle\lim_{n\to+\infty}L_k^{(n)}(\boldsymbol\alpha)$ for every $k\in[N]\backslash\{0\}$, whose limit exists by~\eqref{eq:global-finite-n-lower-values}--\eqref{eq:global-scalar-single-noise-limit}. When $k=1$, by~\eqref{eq:global-scalar-no-noise-limit}--\eqref{eq:global-scalar-single-noise-limit}, $\displaystyle\displaystyle\lim_{n\to+\infty}\tfrac{\E[x_k^2\mathbf 1_{\mathcal E_k^{(n)}}]}{\sigma^2} = \gamma((1-\alpha_0)^2+C(\alpha_0))$ and $\displaystyle\displaystyle\lim_{n\to+\infty}\tfrac{\sum_{j\in[k-1]}\E[x_k^2\mathbf 1_{\mathcal E_{k,j}^{(n)}}]}{\sigma^2} = \alpha_0^2$, respectively. For the same reason, when $k\in[N-1]\backslash\{0\}$, $L_{k+1}(\boldsymbol\alpha) = (1-\alpha_k)^2(L_k(\boldsymbol\alpha) - \alpha^2_{k-1}) + \alpha^2_{k-1}((1-\alpha_k)^2 + C(\alpha_k)) + \alpha^2_k = (1-\alpha_k)^2L_k(\boldsymbol\alpha) +\alpha^2_k + \alpha^2_{k-1}C(\alpha_k)$. It follows $F_a(z)$ defined in~\eqref{eq:global-optimal-defFaCa} that
\begin{equation}
\begin{aligned}
L_1(\alpha)=F_{\alpha_0}(\gamma)+\gamma C(\alpha_0)\  \text{and} \ 
L_{k+1}(\boldsymbol\alpha)=F_{\alpha_k}(L_k(\boldsymbol\alpha))+\alpha_{k-1}^2C(\alpha_k),\ \forall k\in[N-1]\backslash\{0\}.
\end{aligned}
\label{eq:global-unified-lower-recurrence}
\end{equation}
For any fixed dimension $d\ge1$, take $f(x)=\varphi_3([x]_1)+\frac{1}{2}\sum_{j\in[d]\setminus[1]}([x]_j)^2$. Similar to the one-dimensional setting, we consider the initial iterate $x_0\in\R^d$ satisfying $[x_0]_1= \begin{cases} -R&\text{if }\alpha_0\le\bar\alpha\\ R&\text{if }\alpha_0>\bar\alpha \end{cases}$ and $[x_0]_i = 0$ for all $i\in[d]\backslash[1]$. We use the same noise model as~\eqref{eq:global-scalar-noise} along the first coordinate such that the oracle returns stochastic gradient estimates $\nabla f(x_j)+\xi_j\mathbf e_1$, where $\mathbf e_1$ is the first column vector within the $d$-dimensional identity matrix. Such problem instance is admissible in~\eqref{eq:worst-case-error}.
By~\eqref{eq:sg-update}, the first coordinate follows the same iterative update as the one-dimensional example, while all other coordinates stay zero. Thus,~\eqref{eq:worst-case-error} guarantees $\sigma^2L_N^{(n)}\le\D_N(\boldsymbol\alpha)$ for any dimension $d\geq 1$ and every integer $n\ge2$, which further implies that
\begin{equation}
\sigma^2L_N(\boldsymbol\alpha)\le\D_N(\boldsymbol\alpha)\qquad \forall \boldsymbol\alpha\in[0,+\infty)\times[0,1]^{N-1}.
\label{eq:global-proposed-lower}
\end{equation}
\noindent \textbf{Minimizing $L_N(\boldsymbol\alpha)$.}
We prove $L_k(\boldsymbol\alpha)\ge\mathcal T^{\circ k}(\gamma)$ for any $k\in[N]\backslash\{0\}$ by induction. If $\alpha_0\le\bar\alpha$, then $C(\alpha_0)=0$ and~\eqref{eq:global-unified-lower-recurrence} guarantees $L_1(\alpha)=F_{\alpha_0}(\gamma)\ge\mathcal T(\gamma)$.
If $\alpha_0>\bar\alpha$, then $L_1(\alpha)=\gamma(1-M\alpha_0)^2+\alpha_0^2$. Since $\frac{\partial}{\partial a}(\gamma(1-Ma)^2 + a^2)=2M\gamma(Ma-1)+2a>0$ for any $a\ge\bar\alpha>\frac{1}{M}$, we obtain that $L_1(\alpha)\ge\gamma(1-M\bar\alpha)^2+\bar\alpha^2=F_{\bar\alpha}(\gamma)\ge\mathcal T(\gamma)$.

Suppose $L_j(\boldsymbol\alpha)\ge\mathcal T^{\circ j}(\gamma)$ for every $j\in[k]\backslash\{0\}$, where $k\in[N-1]\backslash\{0\}$. If $\alpha_k\le\bar\alpha$, then $F_{\alpha_k}(\mathcal T^{\circ k}(\gamma))\ge\mathcal T^{\circ(k+1)}(\gamma)$ by~\eqref{eq:one-step-optimal-value}. If $\mathcal T^{\circ k}(\gamma)\le\gamma_1$, then $0\le\frac{\mathcal T^{\circ k}(\gamma)}{1+\mathcal T^{\circ k}(\gamma)}\le\frac{\gamma_1}{1+\gamma_1}=\bar\alpha$, then~\eqref{eq:global-unrestricted-one-step-comparison} guarantees $F_{\alpha_k}(\mathcal T^{\circ k}(\gamma))\ge\mathcal T^{\circ(k+1)}(\gamma)$ for every $\alpha_k\ge0$. In either case, \eqref{eq:global-unified-lower-recurrence} and the monotonicity of function $F_{\alpha_k}(\cdot)$ imply that $L_{k+1}(\boldsymbol\alpha)\ge F_{\alpha_k}(L_k(\boldsymbol\alpha))\ge F_{\alpha_k}\left(\mathcal T^{\circ k}(\gamma)\right)\ge\mathcal T^{\circ(k+1)}(\gamma)$.

Now let's suppose $\alpha_k>\bar\alpha$ and $\mathcal T^{\circ k}(\gamma)>\gamma_1$. If $\alpha_{k-1}\le\bar\alpha$, then~\eqref{eq:global-unified-lower-recurrence} and the induction hypothesis lead to $L_k(\boldsymbol\alpha)\ge F_{\alpha_{k-1}}(\mathcal T^{\circ(k-1)}(\gamma))$.
When $k=1$, the same inequality follows directly from $L_1(\alpha)=F_{\alpha_0}(\gamma)+\gamma C(\alpha_0)\ge F_{\alpha_0}(\gamma)$. By~\eqref{eq:one-step-optimal-value}--\eqref{eq:two-step-parameters}, we have $\mathcal T(\gamma_2)=\gamma_1$. It follows $\mathcal T(\cdot)$ being strictly increasing and $\mathcal T(z)\le z$ that $\gamma_2<\mathcal T^{\circ(k-1)}(\gamma)\le\gamma\le\Gamma(M)$ and $\mathcal T^{\circ(k-1)}(\gamma)\ge\mathcal T^{\circ k}(\gamma)>\gamma_1$. Moreover, by~\eqref{eq:global-unified-lower-recurrence} and the first inequality in~\eqref{eq:global-unified-scalar-comparisons}, one has $L_{k+1}(\boldsymbol\alpha)
\ge F_{\alpha_k}(F_{\alpha_{k-1}}(\mathcal T^{\circ(k-1)}(\gamma)))+\alpha_{k-1}^2C(\alpha_k)\ge F_{\bar\alpha}(F_{\bar\alpha}(\mathcal T^{\circ(k-1)}(\gamma))) =\mathcal T^{\circ(k+1)}(\gamma)$, where the equality is from~\eqref{eq:one-step-optimal-value} because $\mathcal T^{\circ(k-1)}(\gamma)>\gamma_1$ and $\mathcal T^{\circ k}(\gamma)>\gamma_1$.
If $\alpha_{k-1}>\bar\alpha$, then $\alpha_{k-1}^2C(\alpha_k)\ge\bar\alpha^2C(\alpha_k)$ and $0\le\mathcal T^{\circ k}(\gamma)\le\mathcal T(\gamma)\le\mathcal T(\Gamma(M))$.
By~\eqref{eq:global-unified-lower-recurrence} and the second inequality in~\eqref{eq:global-unified-scalar-comparisons}, $L_{k+1}(\boldsymbol\alpha)\ge F_{\alpha_k}(\mathcal T^{\circ k}(\gamma))+\bar\alpha^2C(\alpha_k)\ge F_{\bar\alpha}(\mathcal T^{\circ k}(\gamma))=\mathcal T^{\circ(k+1)}(\gamma)$.
Therefore, we have
\begin{equation}
L_N(\boldsymbol\alpha)\ge\mathcal T^{\circ N}(\gamma)\qquad\forall \boldsymbol\alpha\in[0,+\infty)\times[0,1]^{N-1}.
\label{eq:global-lower-minimum-comparison}
\end{equation}
by induction.
When $N=1$,~\eqref{eq:global-lower-minimum-comparison} directly follows from the comparison for $L_1(\cdot)$.

\noindent \textbf{Attainment by the proposed schedule.}
By~\eqref{eq:global-proposed-upper},~\eqref{eq:global-stepsize-above-one-lower}, and \eqref{eq:global-short-value-comparison}, the schedule $\mathcal A_N(\gamma)$ satisfies $\alpha_k\in[0,1]$ for every $k\in[N-1]\backslash\{0\}$. Moreover, by the help of~\eqref{eq:global-proposed-upper},~\eqref{eq:global-proposed-lower}, and~\eqref{eq:global-lower-minimum-comparison}, we know that $\sigma^2\mathcal T^{\circ N}(\gamma)\le\sigma^2L_N(\mathcal A_N(\gamma))\le\D_N(\mathcal A_N(\gamma))\le\sigma^2\mathcal T^{\circ N}(\gamma)$ for all $\gamma\in[0, \Gamma(M)]$. Therefore, $L_N(\mathcal A_N(\gamma))=\mathcal T^{\circ N}(\gamma)$, and~\eqref{eq:global-lower-minimum-comparison} further implies that $\mathcal A_N(\gamma)\in\arg\min_{\boldsymbol\alpha\in[0,+\infty)\times[0,1]^{N-1}}L_N(\boldsymbol\alpha)$ and $\min_{\boldsymbol\alpha\in[0,+\infty)\times[0,1]^{N-1}}L_N(\boldsymbol\alpha)=\mathcal T^{\circ N}(\gamma)$.
For stepsize schedules $\boldsymbol\alpha\in[0,+\infty)\times[0,1]^{N-1}$,~\eqref{eq:global-proposed-lower}--\eqref{eq:global-lower-minimum-comparison} imply $\D_N(\boldsymbol\alpha)\ge\sigma^2\mathcal T^{\circ N}(\gamma)$. Every schedule with $\alpha_i>1$ for some $i\in[N-1]\backslash\{0\}$ satisfies $\D_N(\boldsymbol\alpha)>\sigma^2\mathcal T^{\circ N}(\gamma)$ by~\eqref{eq:global-stepsize-above-one-lower}--\eqref{eq:global-short-value-comparison}. Thus, by~\eqref{eq:minimax-schedule}, we conclude that $\D_N(\mathcal A_N(\gamma))=\inf_{\boldsymbol\alpha\in\mathbb R_{\ge0}^N}\D_N(\boldsymbol\alpha)=\sigma^2\mathcal T^{\circ N}(\gamma)=\D_N^\star$, which concludes the statement.
\end{proof}

\end{document}